\documentclass{article}
\usepackage[english]{babel}
\usepackage{geometry,amsmath,amssymb,graphicx,stmaryrd,enumerate,mathrsfs,bbm,latexsym,ntheorem,hyperref}
\usepackage[titletoc]{appendix}
\catcode`\<=\active \def<{
\fontencoding{T1}\selectfont\symbol{60}\fontencoding{\encodingdefault}}
\newcommand{\assign}{:=}
\newcommand{\cdummy}{\cdot}
\newcommand{\comma}{{,}}
\newcommand{\mathD}{\mathrm{D}}
\newcommand{\mathd}{\mathrm{d}}
\newcommand{\mathpi}{\pi}
\newcommand{\nobracket}{}
\newcommand{\nosymbol}{}
\newcommand{\point}{.}
\newcommand{\precprec}{\prec\!\!\!\prec}
\newcommand{\tmaffiliation}[1]{\\ #1}
\newcommand{\tmdummy}{$\mbox{}$}
\newcommand{\tmemail}[1]{\\ \textit{Email:} \texttt{#1}}
\newcommand{\tmop}[1]{\ensuremath{\operatorname{#1}}}

\newcommand{\tmtextit}[1]{{\itshape{#1}}}
\newenvironment{enumeratealpha}{\begin{enumerate}[a{\textup{)}}] }{\end{enumerate}}
\newenvironment{itemizeminus}{\begin{itemize} }{\end{itemize}}
\newenvironment{proof}{\noindent\textbf{Proof\ }}{\hspace*{\fill}$\Box$\medskip}

\newtheorem{axiom}{Assumption}
\newcounter{thm}
\numberwithin{thm}{section}
{\theorembodyfont{\rmfamily}\newtheorem{remark}[thm]{Remark}}
\theoremstyle{break}
\newtheorem{theorem}[thm]{Theorem}
\newtheorem{corollary}[thm]{Corollary}
\newtheorem{lemma}[thm]{Lemma}

\newcommand{\tmkeywords}{\textbf{Keywords:} }
\newcommand{\tthreethreerprime}[1]{#1^{\prime \resizebox{1.5em}{!}{\includegraphics{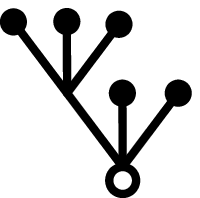}}}}

\newcommand{\tthree}[1]{#1^{\resizebox{1em}{!}{\includegraphics{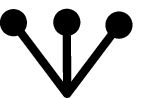}}}}

\newcommand{\tthreeone}[1]{#1^{\resizebox{1em}{!}{\includegraphics{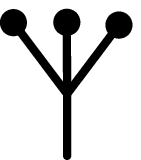}}}}
\newcommand{\ttwo}[1]{#1^{\resizebox{1em}{!}{\includegraphics{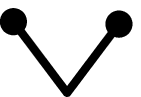}}}}

\newcommand{\tthreethreer}[1]{#1^{\resizebox{1.5em}{!}{\includegraphics{trees/33_tree_res.eps}}}}

\newcommand{\ttwoone}[1]{#1^{\resizebox{1em}{!}{\includegraphics{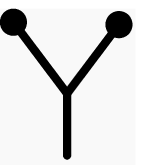}}}}
\newcommand{\ttwothreer}[1]{#1^{\resizebox{1.5em}{!}{\includegraphics{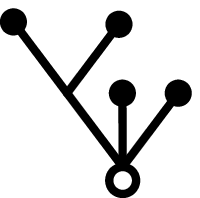}}}}
\newcommand{\tthreetwor}[1]{#1^{\resizebox{1.5em}{!}{\includegraphics{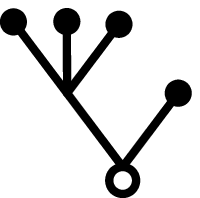}}}}
\newcommand{\tone}[1]{#1^{\resizebox{0.5em}{!}{\includegraphics{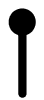}}}}
\newcommand{\tzero}[1]{#1^{\varnothing}}
\newcommand{\CC}{\mathscr{C} \hspace{.1em}}
\newcommand{\CD}{\mathscr{D} \hspace{.1em}}
\newcommand{\CF}{\mathscr{F} \hspace{.1em}}

\newcommand{\DD}{\mathscr{D} \hspace{.1em}}
\newcommand{\LL}{\mathscr{L} \hspace{.2em}}

\numberwithin{equation}{section}
\begin{document}

\title{
  Weak universality for a class of\\
  3d stochastic reaction--diffusion models
}

\author{
  M.~Furlan
  \tmaffiliation{CEREMADE\\
  Universit{\'e} Paris Dauphine, France}
  \tmemail{furlan@ceremade.dauphine.fr}
  \and
  M.~Gubinelli
  \tmaffiliation{IAM \& HCM\\
  Universit{\"a}t Bonn, Germany}
  \tmemail{gubinelli@iam.uni-bonn.de}
}

\maketitle

\begin{abstract}
  We establish the large scale convergence of a class of stochastic weakly
  nonlinear reaction--diffusion models on a three dimensional periodic domain
  to the dynamic $\Phi^4_3$ model within the framework of paracontrolled
  distributions. Our work extends previous results of Hairer and Xu to
  nonlinearities with a finite amount of smoothness (in particular $C^9$ is
  enough). We use the Malliavin calculus to perform a partial chaos expansion
  of the stochastic terms and control their $L^p$ norms in terms of the graphs
  of the standard $\Phi^4_3$ stochastic terms. 
\end{abstract}

\tmkeywords{weak universality, paracontrolled distributions, stochastic
quantisation equation, Malliavin calculus, partial chaos expansion.}

\tableofcontents

\section{Introduction}

Consider a family of stochastic reaction--diffusion equation in a
\tmtextit{weakly nonlinear regime}:
\begin{equation}
  \LL u (t, x) = - \varepsilon^{\alpha} F_{\varepsilon} (u (t, x)) + \eta (t,
  x), \label{e:micro} \qquad (t, x) \in [0, T / \varepsilon^2] \times
  (\mathbbm{T}/ \varepsilon)^3
\end{equation}
with $\varepsilon \in (0, 1]$, $T > 0$, initial condition $\bar{u}_{0,
\varepsilon} : (\mathbbm{T}/ \varepsilon)^3 \rightarrow \mathbbm{R}$,
$F_{\varepsilon} \in C^9 (\mathbbm{R})$ with exponential growth at infinity,
$\alpha > 0$ and $\LL : = (\partial_t - \Delta)$ the heat flow operator and
$\mathbbm{T}=\mathbbm{R}/ (2 \pi \mathbbm{Z})$. Here $\eta$ denotes a family
of centered Gaussian noises on $[0, T / \varepsilon^2] \times (\mathbbm{T}/
\varepsilon)^3$ \ with stationary covariance
\[ \mathbbm{E} (\eta (t, x) \eta (s, y)) = \tilde{C}^{\varepsilon} (t - s, x,
   y) \]
such that $\tilde{C}^{\varepsilon} (t - s, x, y) = \Sigma (t - s, x - y)$ if
$\tmop{dist} (x, y) \leqslant 1$ and $0$ otherwise where $\Sigma : \mathbbm{R}
\times \mathbbm{R}^3 \rightarrow \mathbbm{R}^+$ is a smooth, positive function
compactly supported in $[0, 1] \times B_{\mathbbm{R}^3} (0, 1)$. We assume
also that there exists a compactly supported function $\psi$ such that $\psi
\ast \psi = \Sigma$ (this is true e.g. when $\eta$ is obtained by space-time
convolution of the white noise with $\psi$).

We look for a large scale description of the solution to eq.~(\ref{e:micro})
and we introduce the ``mesoscopic'' scale variable $u_{\varepsilon} (t, x) =
\varepsilon^{- \beta} u (t / \varepsilon^2, x / \varepsilon)$ where $\beta >
0$. Substituting $u_{\varepsilon}$ into (\ref{e:micro}) we get
\begin{equation}
  \LL u_{\varepsilon} (t, x) = - \varepsilon^{\alpha - 2 - \beta}
  F_{\varepsilon} (\varepsilon^{\beta} u_{\varepsilon} (t, x)) +
  \varepsilon^{- 2 - \beta} \eta \left( \frac{t}{\varepsilon^2},
  \frac{x}{\varepsilon} \right) . \label{e:micro-2}
\end{equation}
In order for the term $\varepsilon^{- 2 - \beta} \eta (t / \varepsilon^2, x /
\varepsilon)$ to converge to a non--trivial random limit we need that $\beta =
1 / 2$. Indeed the Gaussian field $\eta_{\varepsilon} (t, x) : =
\varepsilon^{- 5 / 2} \eta (t / \varepsilon^2, x / \varepsilon)$ has
covariance $\varepsilon^{- 5} \tilde{C}^{\varepsilon} (t / \varepsilon^2, x /
\varepsilon)$ and converges in distribution to the space-time white noise on
$\mathbbm{R} \times \mathbbm{T}^3$. For large values of $\alpha$ the
non--linearity will be negligible with respect to the additive noise term.
Heuristically, we can attempt an expansion of the reaction term around the
stationary solution $Y_{\varepsilon}$ to the linear equation
\begin{equation}
  \LL Y_{\varepsilon} = - Y_{\varepsilon} + \eta_{\varepsilon},
  \label{e:Y-stat-eq}
\end{equation}
i.e. $Y_{\varepsilon} (t, x) = \int_{- \infty}^t \check{P} (t - s, x - y)
\eta_{\varepsilon} (s, y) \mathd s \mathd y$ with $\check{P} (t, x) =
\frac{1}{(4 \mathpi t)^{3 / 2}} e^{- \frac{| x |^2}{4 t}} e^{- t} 
\mathbbm{1}_{t \geqslant 0}$.

Let us denote with $C_{\varepsilon}$ the covariance of $Y_{\varepsilon}$. We
approximate the reaction term as
\[ \varepsilon^{\alpha - 5 / 2} F_{\varepsilon} (\varepsilon^{1 / 2}
   u_{\varepsilon} (t, x)) \simeq \varepsilon^{\alpha - 5 / 2} F_{\varepsilon}
   (\varepsilon^{1 / 2} Y_{\varepsilon} (t, x)) . \]
The Gaussian r.v. $\varepsilon^{1 / 2} Y_{\varepsilon} (t, x)$ has variance
$\sigma_{\varepsilon}^2 = \varepsilon \mathbbm{E} [ (Y_{\varepsilon} (t,
x))^2] = \varepsilon \mathbbm{E} [ (Y_{\varepsilon} (0, 0))^2] = \varepsilon
C_{\varepsilon} (0, 0)$ independent of $(t, x)$. Although
$\sigma_{\varepsilon}^2$ depends on $\varepsilon$, it can be bounded from
above and below by two positive constants uniformly on $\varepsilon \in (0,
1]$. We can expand the random variable $F_{\varepsilon} (\varepsilon^{1 / 2}
Y_{\varepsilon} (t, x))$ according to the chaos decomposition relative to
$\varepsilon^{1 / 2} Y_{\varepsilon} (t, x)$ and obtain
\begin{equation}
  F_{\varepsilon} (\varepsilon^{1 / 2} Y_{\varepsilon} (t, x)) = \sum_{n
  \geqslant 0} f_{n, \varepsilon} H_n (\varepsilon^{1 / 2} Y_{\varepsilon} (t,
  x), \sigma_{\varepsilon}^2), \label{e:F-chaos-exp}
\end{equation}
where $H_n (x, \sigma_{\varepsilon}^2)$ are standard Hermite polynomials with
variance $\sigma_{\varepsilon}^2$ and highest-order term normalized to $1$.
Note also that the coefficients $(f_{n, \varepsilon})_{n \geqslant 0}$ do not
depend on $(t, x)$ by stationarity of the law of $\varepsilon^{1 / 2}
Y_{\varepsilon} (t, x)$ since they are given by the formula
\[ f_{n, \varepsilon} = \frac{1}{n! \sigma^{2 n}_{\varepsilon}} \mathbbm{E}
   [F_{\varepsilon} (\varepsilon^{1 / 2} Y_{\varepsilon} (t, x)) H_n
   (\varepsilon^{1 / 2} Y_{\varepsilon} (t, x), \sigma_{\varepsilon}^2)] =
   n!\mathbbm{E} [F_{\varepsilon} (\sigma_{\varepsilon} G) H_n
   (\sigma_{\varepsilon} G, \sigma_{\varepsilon}^2)] \]
where $G$ is a standard Gaussian variable of unit variance.

Let $X$ be the stationary solution to the equation
\[ \LL X = - X + \xi, \]
with $\xi$ the space--time white noise on $\mathbbm{R} \times \mathbbm{T}^3$
and denote by $\llbracket X^N \rrbracket$ the generalized random fields given
by the $N$-th Wick power of $X$, which are well defined as random elements of
$\mathcal{S}' (\mathbbm{R} \times \mathbbm{T}^3)$ as long as $N \leqslant 4$.
Denote with $C_X$ the covariance of $X$. The Gaussian analysis which we set up
in this paper shows in particular that if $\varepsilon^{(n - N) / 2} f_{n,
\varepsilon} \rightarrow g_n$ as $\varepsilon \rightarrow 0$ for every $0
\leqslant n \leqslant N$, $N \leqslant 4$, and
$(F_{\varepsilon})_{\varepsilon} \subseteq C^{N + 1} (\mathbbm{R})$ with
exponential growth, then the family of random fields
\[ \mathbbm{F}^N_{\varepsilon} : (t, x) \mapsto \varepsilon^{- N / 2}
   F_{\varepsilon} (\varepsilon^{1 / 2} Y_{\varepsilon} (t, x)), \qquad (t, x)
   \in \mathbbm{R} \times \mathbbm{T}^3, \]
converges in law in $\mathcal{S}' (\mathbbm{R} \times \mathbbm{T}^3)$ as
$\varepsilon \rightarrow 0$ to $\sum_{n = 0}^N g_n \llbracket X^n \rrbracket$.

Consider the smallest $n$ such that $f_{n, \varepsilon}$ converges to a finite
limit as $\varepsilon \rightarrow 0$. Since $H_n (\varepsilon^{1 / 2}
Y_{\varepsilon}, \sigma_{\varepsilon}^2) = \varepsilon^{n / 2} \llbracket
Y_{\varepsilon}^n \rrbracket$, the $n$-th term in the
expansion~(\ref{e:F-chaos-exp}) is $f_{n, \varepsilon} \varepsilon^{\alpha +
(n - 5) / 2} \llbracket Y_{\varepsilon}^n \rrbracket$. Therefore, the equation
yields a non-trivial limit only if $\alpha = (5 - n) / 2$. We are interested
mainly in the case $n = 3 \Rightarrow \alpha = 1$ and $n = 1 \Rightarrow
\alpha = 2$. The case $\alpha = 2$ gives rise to a Gaussian limit and its
analysis its not very difficult. In the following we will concentrate in the
analysis of the $\alpha = 1$ case where the limiting behaviour of the model is
the most interesting and given by the $\Phi^4_3$ family of singular SPDEs. In
this case we obtain the family of models
\begin{equation}
  \LL u_{\varepsilon} (t, x) = - \varepsilon^{- \frac{3}{2}} F_{\varepsilon}
  (\varepsilon^{\frac{1}{2}} u_{\varepsilon} (t, x)) + \eta_{\varepsilon} (t,
  x) \quad \label{e:univ}
\end{equation}
with initial condition $u_{0, \varepsilon} (\cdummy) \assign \varepsilon^{-
\frac{1}{2}} \bar{u}_{0, \varepsilon} (\varepsilon^{- 1} \cdummy)$ where
$\bar{u}_{0, \varepsilon}$ is the initial condition of the microscopic
model~(\ref{e:micro}).

In order to state our main result, Theorem~\ref{t:maintheorem} below, let us
introduce some notations and specify our assumptions. Let
$\tilde{F}_{\varepsilon}$ be the centering (up to the third Wiener chaos
relative to $\varepsilon^{1 / 2} Y_{\varepsilon} (t, x)$) of the function
$F_{\varepsilon}$, i.e.
\begin{equation}
  \tilde{F}_{\varepsilon}  (x) \assign F_{\varepsilon} (x) - f_{0,
  \varepsilon} - f_{1, \varepsilon} x - f_{2, \varepsilon} H_2 (x,
  \sigma_{\varepsilon}^2) = \sum_{n \geqslant 3} f_{n, \varepsilon} H_n (x,
  \sigma_{\varepsilon}^2) . \label{e:ftilde-def}
\end{equation}
The decomposition of $\tilde{F}_{\varepsilon}$ is obviously the same as
in~(\ref{e:F-chaos-exp}) except for the fact that we have discarded the orders
0,1,2. Let $\tilde{F}_{\varepsilon}^{ (m)}$ be the $m$-th derivative of the
function $\tilde{F}_{\varepsilon}$ for $0 \leqslant m \leqslant 9$ and define
the following $\varepsilon$--dependent constants:
\begin{equation}
  \begin{array}{lll}
    \ttwothreer{d_{\varepsilon}} & \assign & \frac{\varepsilon^{- 2}}{9}
    \int_{s, x} P_s (x) \mathbbm{E} [\tilde{F}_{\varepsilon}^{ (1)}
    (\varepsilon^{1 / 2} Y_{\varepsilon} (s, x)) \tilde{F}_{\varepsilon}^{
    (1)} (\varepsilon^{1 / 2} Y_{\varepsilon} (0, 0))],\\
    \ttwothreer{\bar{d}_{\varepsilon}} & \assign & 2 \varepsilon^{- 1 / 2}
    f_{3, \varepsilon} f_{2, \varepsilon}  \int_{s, x} P_s (x)
    (C_{\varepsilon} (s, x))^2,\\
    \tthreetwor{d_{\varepsilon}} & \assign & \frac{\varepsilon^{- 2}}{6}
    \int_{s, x} P_s (x) \mathbbm{E} [\tilde{F}_{\varepsilon}^{ (0)}
    (\varepsilon^{1 / 2} Y_{\varepsilon} (s, x))_{} \tilde{F}_{\varepsilon}^{
    (2)} (\varepsilon^{1 / 2} Y_{\varepsilon} (0, 0))],\\
    \tthreethreer{d_{\varepsilon}} & \assign & \frac{\varepsilon^{- 5 / 2}}{3}
    \int_{s, x} P_s (x) \mathbbm{E} [\tilde{F}_{\varepsilon}^{ (0)}
    (\varepsilon^{1 / 2} Y_{\varepsilon} (s, x))_{} \tilde{F}_{\varepsilon}^{
    (1)} (\varepsilon^{1 / 2} Y_{\varepsilon} (0, 0))],
  \end{array} \label{eq:d-constants}
\end{equation}
where $P_s (x)$ is the heat kernel and $\int_{s, x}$ denotes integration on
$\mathbbm{R}^+ \times \mathbbm{T}^3$.

\begin{axiom}
  \label{a:main}All along the paper we enforce the following assumptions:
  \begin{enumeratealpha}
    \item $(u_{0, \varepsilon})_{\varepsilon}$ converges in law to a limit
    $u_0$ in $\CC^{- 1 / 2 - \kappa}$ and is independent of $\eta$;
    
    \item $(\bar{u}_{0, \varepsilon})_{\varepsilon}$ is uniformly bounded in
    $L^{\infty}$ in probability, i.e. $\exists C > 0$ such that $\forall
    \varepsilon \in (0, 1]$ $\| \bar{u}_{0, \varepsilon} \|_{L^{\infty}
    ((\mathbbm{T}/ \varepsilon)^3)} \leqslant C$;
    
    \item $(F_{\varepsilon})_{\varepsilon} \subseteq C^9 (\mathbbm{R})$ and
    there exists constants $c, C > 0$ such that
    \begin{equation}
      \sup_{\varepsilon, x} \sum_{k = 0}^9 | \partial_x^k F_{\varepsilon} (x)
      | \leqslant C e^{c | x |}, \label{e:hyp-F-main}
    \end{equation}
    \item the family of vectors $\lambda_{\varepsilon} = (\lambda_{0,
    \varepsilon}, \lambda_{1, \varepsilon}, \lambda_{2, \varepsilon},
    \lambda_{3, \varepsilon}) \in \mathbbm{R}^4$ given by
    \begin{equation}
      \begin{array}{lllllll}
        \lambda_{3, \varepsilon} & \assign & f_{3, \varepsilon} &  &
        \lambda_{1, \varepsilon} & \assign & \varepsilon^{- 1} f_{1,
        \varepsilon} - 9 \ttwothreer{d_{\varepsilon}} - 6
        \tthreetwor{d_{\varepsilon}}\\
        \lambda_{2, \varepsilon} & \assign & \varepsilon^{- 1 / 2} f_{2,
        \varepsilon} &  & \lambda_{0, \varepsilon} & \assign & \varepsilon^{-
        3 / 2} f_{0, \varepsilon} - \varepsilon^{- 1 / 2} f_{2, \varepsilon} 
        \tthreetwor{d_{\varepsilon}} - 3 \tthreethreer{d_{\varepsilon}} - 3
        \ttwothreer{\bar{d}_{\varepsilon}}
      \end{array} \label{e:def-lambda}
    \end{equation}
    has a finite limit $\lambda = (\lambda_0, \lambda_1, \lambda_2, \lambda_3)
    \in \mathbbm{R}^4$ as $\varepsilon \rightarrow 0$. 
  \end{enumeratealpha}
\end{axiom}

\begin{theorem}[Convergence of the solution]
  \label{t:maintheorem}Under Assumption~\ref{a:main} the family of random
  fields $(u_{\varepsilon})_{\varepsilon}$ given by the solution to
  eq.~(\ref{e:univ}) converges in law and locally in time to a limiting random
  field $u (\lambda)$ in the space $C_T \CC^{- \alpha} (\mathbbm{T}^3)$ for
  every $1 / 2 < \alpha < 2 / 3$. The law of $u (\lambda)$ depends only on the
  value of $\lambda$ and not on the other details of the nonlinearity or on
  the covariance of the noise term. We call this limit the dynamic $\Phi^4_3$
  model with parameter vector $\lambda \in \mathbbm{R}^4$. \ 
\end{theorem}

Here $C_T \CC^{- \alpha} (\mathbbm{T}^3)$ denotes the space of continuous
functions from $[0, T]$ to the Besov space $\CC^{- \alpha} (\mathbbm{T}^3) =
B^{- \alpha}_{\infty, \infty} (\mathbbm{T}^3)$ (see Appendix~\ref{s:notations}
for the notation on Besov spaces and paraproducts).
Theorem~\ref{t:maintheorem} is actually just a corollary of the more precise
result Theorem~\ref{t:lim-ident}, in which we identify the paracontrolled
equation satisfied by the limit random field $u (\lambda)$.

\begin{remark}
  We are interested only in local-in-time convergence of $u_{\varepsilon}$, as
  a way to show the potential of our method for controlling stochastic terms
  with infinite chaos decomposition (developed in Section~\ref{s:comparing}).
  Nevertheless, we expect it to be possible to obtain global-in-time
  convergence of the solution with more stringent assumptions on
  $F_{\varepsilon}$, although we do not treat this problem here.
\end{remark}

\begin{remark}
  \label{r:usualphi4}As a special case we can take
  \[ \begin{array}{lll}
       F_{\varepsilon} (x) & = & \lambda_3 H_3 (x, \sigma_{\varepsilon}^2) +
       \varepsilon^{1 / 2} \lambda_2 H_2 (x, \sigma_{\varepsilon}^2) +
       \varepsilon (\lambda_1 + \gamma_{1, \varepsilon}) H_1 (x,
       \sigma_{\varepsilon}^2) + \varepsilon^{3 / 2} (\lambda_0 + \gamma_{0,
       \varepsilon})
     \end{array} \]
  so that
  \[ f_{3, \varepsilon} = \lambda_3, \quad \varepsilon^{- 1 / 2} f_{2,
     \varepsilon} = \lambda_2, \quad \varepsilon^{- 1} f_{1, \varepsilon} =
     \lambda_1 + \gamma_{1, \varepsilon}, \quad \varepsilon^{- 3 / 2} f_{0,
     \varepsilon} = \lambda_0 + \gamma_{0, \varepsilon}, \]
  and
  \[ \ttwothreer{d_{\varepsilon}} = (\lambda_3)^2 L_{\varepsilon}, \quad
     \ttwothreer{\bar{d}_{\varepsilon}} = \lambda_3 \lambda_2 L_{\varepsilon},
     \quad \tthreetwor{d_{\varepsilon}} = \tthreethreer{d_{\varepsilon}} = 0,
  \]
  where $L_{\varepsilon} \assign 2 \int_{s, x} P_s (x) (C_{Y, \varepsilon} (s,
  x))^2 .$ Choosing
  \[ \gamma_{1, \varepsilon} \assign 9 \ttwothreer{d_{\varepsilon}} = 9
     (\lambda_3)^2 L_{\varepsilon}, \qquad \gamma_{0, \varepsilon} \assign 3
     \ttwothreer{\bar{d}_{\varepsilon}} = 3 \lambda_3 \lambda_2
     L_{\varepsilon}, \]
  we obtain $\lambda_{\varepsilon} \rightarrow (\lambda_0, \lambda_1,
  \lambda_2, \lambda_3)$. This shows that all the possible limits $\lambda \in
  \mathbbm{R}^4$ are attainable. In this case~(\ref{e:univ}) takes the form
  \begin{equation}
    \LL u_{\varepsilon} = - \lambda_3 u_{\varepsilon}^3 - \lambda_2
    u_{\varepsilon}^2 - [\lambda_1 - 3 \lambda_3 \varepsilon^{- 1}
    \sigma^2_{\varepsilon} + 9 (\lambda_3)^2 L_{\varepsilon}] u_{\varepsilon}
    - \lambda_0 + \lambda_2 \sigma^2_{\varepsilon} - 3 \lambda_3 \lambda_2
    L_{\varepsilon} + \eta_{\varepsilon} . \label{eq:cubic}
  \end{equation}
  The name dynamic $\Phi^4_3$ equation (or stochastic quantisation equation)
  derives from the fact that the simplest class of models which approximate
  the limiting random field $u (\lambda)$ is precisely obtained by choosing a
  cubic polynomial like in~(\ref{eq:cubic}) as non-linear term (which is the
  gradient of a fourth order polinomial playing the role of local potential).
\end{remark}

In two dimensions, this model has been subject of various studies since more
than thirty
years~{\cite{jona-lasinio_stochastic_1985,albeverio_stochastic_1991,da_prato_strong_2003}}.
For the three dimensional case, the kind of convergence results described
above are originally due to
Hairer~{\cite{hairer_theory_2014,hairer_regularity_2015}} and constitute one
of the first groundbreaking applications of his theory of regularity
structures. Similar results were later obtained by Catellier and
Chouk~{\cite{catellier_paracontrolled_2013}} using the paracontrolled approach
of Gubinelli, Imkeller and Perkowski~{\cite{gubinelli_paracontrolled_2012}}.
Kupiainen~{\cite{kupiainen_renormalization_2014}} described a third approach
using renormalization group ideas.

\tmtextit{Weak universality} is the observation that the same limiting object
describes the large scale behaviour of the solutions of more general
equations, in particular that of the many parameters present in a general
model, only a finite number of their combinations survive in the limit to
describe the limiting object. The adjective ``weak'' is related to the fact
that in order to control the large scale limit the non-linearity has to be
very small in the microscopic scale. This sets up a perturbative regime which
is well suited to the analysis via regularity structures or paracontrolled
distributions.

The first result of weak universality for a singular stochastic PDE has been
given by Hairer and Quastel~{\cite{hairer_class_2015}} in the context the
Kardar--Parisi--Zhang equation. Using the machinery developed there Hairer and
Wu~{\cite{hairer_large_2016}} proved a weak universality result for three
dimensional reaction--diffusion equations in the case of Gaussian noise and a
polynomial non--linearity, within the context of regularity structures. Weak
universality for reaction--diffusion equations driven by non Gaussian noise is
analysed in Shen and Wu~{\cite{shen_weak_2016}}. Recently, important results
concerning the stochastic quantisation equation we obtained by Mourrat and
Weber. In particular the convergence to the dynamic $\Phi^4_2$ model for a
class of Markovian dynamics of discrete spin
systems~{\cite{mourrat_convergence_2014}} and also the global wellposedness of
$\Phi^4_2$ in space and time~{\cite{mourrat_global_2015}} and in time for
$\Phi^4_3$~{\cite{mourrat_global_2016}}. The recent
preprint~{\cite{gubinelli_renormalization_2017}} analyzes an hyperbolic
version of the stochastic quantisation equation in two dimensions, including
the associated universality in the small noise regime.

The present work is the first which considers in detail the weak universality
problem in the context of paracontrolled distributions, showing that on the
analytic side the apriori estimates can be obtained via standard arguments and
that the major difficulty is related to showing the convergence of a finite
number of random fields to universal limiting objects. The main novelty of our
work is our use of the Malliavin
calculus~{\cite{nualart_2006,nourdin_peccati_2012}} to perform the analysis of
these stochastic terms without requiring polynomial non--linearity as in the
previous works cited above. In particular we were inspired by the computations
in~{\cite{nourdin_nualart_2007}} and in general by the use of the Malliavin
calculus to establish normal approximations~{\cite{nourdin_peccati_2012}}. The
main technical results of our paper, Theorem~\ref{t:stoch-conv} below, is not
particularly linked to paracontrolled distributions. A similar analysis is
conceivable \ for the stochastic models in regularity structures. Moreover the
same tools can also allow to prove similar non-polynomial weak universality
statements for the KPZ along the lines of the present analysis. This is the
subject of ongoing work.

\begin{acknowledgments*}
  The authors would like to thank the anonymous referee for the detailed and
  constructive critique which contributed to improve the overall exposition of
  the results. Support via SFB CRC 1060 is also gratefully acknowledged. 
\end{acknowledgments*}

\section{\label{s:universal-analytic}Analysis of the mesoscopic model}

The goal of this section is to obtain a paracontrolled structure for
equation~(\ref{e:univ}) analogous to that introduced
in~{\cite{catellier_paracontrolled_2013}} for the cubic polynomial case and
use it to set up the limiting procedure. Convergence of the stochastic terms
and some apriori estimates will be the subject of the following sections.
Definitions and a reminder of the basic results of paradifferential calculus
needed here can be found in Appendix~\ref{a:basics}.

\subsection{\label{s:paracontrolled-struc}Paracontrolled structure}

Let us start our analysis by centering the reaction term $F_{\varepsilon}
(\varepsilon^{1 / 2} u_{\varepsilon})$ in~(\ref{e:univ}) using
decomposition~(\ref{e:ftilde-def}) to obtain:
\begin{eqnarray*}
  \LL u_{\varepsilon} & = & - \varepsilon^{- \frac{3}{2}}
  \tilde{F}_{\varepsilon} (\varepsilon^{\frac{1}{2}} u_{\varepsilon} (t, x)) +
  \eta_{\varepsilon}\\
  &  & - \varepsilon^{- 3 / 2} f_{0, \varepsilon} - \varepsilon^{- 1} f_{1,
  \varepsilon} u_{\varepsilon} - \varepsilon^{- 3 / 2} f_{2, \varepsilon} H_2
  \left( \varepsilon^{\frac{1}{2}} u_{\varepsilon}, \sigma_{\varepsilon}^2
  \right) .
\end{eqnarray*}
We write $u_{\varepsilon} = Y_{\varepsilon} + v_{\varepsilon}$ with
$Y_{\varepsilon}$ as in~(\ref{e:Y-stat-eq}), and perform a Taylor expansion of
$\tilde{F}_{\varepsilon} (\varepsilon^{1 / 2} Y_{\varepsilon} + \varepsilon^{1
/ 2} v_{\varepsilon})$ around $\varepsilon^{1 / 2} Y_{\varepsilon}$ up to the
third order to get
\begin{equation}
  \begin{array}{lll}
    \LL u_{\varepsilon} & = & \eta_{\varepsilon} - \varepsilon^{- \frac{3}{2}}
    \tilde{F}_{\varepsilon} (\varepsilon^{\frac{1}{2}} Y_{\varepsilon}) -
    \varepsilon^{- 1} \tilde{F}^{(1)}_{\varepsilon} (\varepsilon^{\frac{1}{2}}
    Y_{\varepsilon}) v_{\varepsilon} - \frac{1}{2} \varepsilon^{- \frac{1}{2}}
    \tilde{F}^{(2)}_{\varepsilon} (\varepsilon^{\frac{1}{2}} Y_{\varepsilon})
    v_{\varepsilon}^2 - \frac{1}{6} \tilde{F}^{(3)}_{\varepsilon}
    (\varepsilon^{\frac{1}{2}} Y_{\varepsilon}) v_{\varepsilon}^3\\
    &  & - \varepsilon^{- 3 / 2} f_{0, \varepsilon} - \varepsilon^{- 1} f_{1,
    \varepsilon} (Y_{\varepsilon} + v_{\varepsilon}) - \varepsilon^{- 1 / 2}
    f_{2, \varepsilon} (\llbracket Y_{\varepsilon}^2 \rrbracket + 2
    v_{\varepsilon} Y_{\varepsilon} + v_{\varepsilon}^2) - R_{\varepsilon}
    (v_{\varepsilon}) .
  \end{array} \label{eq:first-exp}
\end{equation}
where $R_{\varepsilon} (v_{\varepsilon})$ is the remainder of the Taylor
series and we used the fact that $H_2 (\varepsilon^{1 / 2} Y_{\varepsilon},
\sigma_{\varepsilon}^2) = \varepsilon \llbracket Y_{\varepsilon}^2
\rrbracket$. Notice that we stopped the Taylor expansion at the first term for
which $\varepsilon$ does not appear anymore with a negative exponent (that is
$\tilde{F}^{(3)}_{\varepsilon} (\varepsilon^{\frac{1}{2}} Y_{\varepsilon})$).
One can then expect the remainder $R_{\varepsilon} (v_{\varepsilon})$ to
converge to zero in some sense. On the other hand, all the other terms except
$\tilde{F}^{(3)}_{\varepsilon} (\varepsilon^{\frac{1}{2}} Y_{\varepsilon})$
and $R_{\varepsilon} (v_{\varepsilon})$ appear to diverge in the limit
$\varepsilon \rightarrow 0$, but in analogy with well-known renormalization
methods for random fields, we try to find a combination of them that can be
made to converge in some function space. Define the following random fields:

\begin{equation}
  \begin{array}{lllllll}
    \LL Y_{\varepsilon} & \assign & - Y_{\varepsilon} + \eta_{\varepsilon} &
    \hspace{4em} &  &  & \\
    \ttwo{\bar{Y}_{\varepsilon}} & \assign & \varepsilon^{- 1 / 2} f_{2,
    \varepsilon}  \llbracket Y_{\varepsilon}^2 \rrbracket &  & \LL
    \ttwoone{\bar{Y}_{\varepsilon}} & \assign &
    \ttwo{\bar{Y}_{\varepsilon}},\\
    \tthree{Y_{\varepsilon}} & \assign & \varepsilon^{- \frac{3}{2}}
    \tilde{F}_{\varepsilon} (\varepsilon^{\frac{1}{2}} Y_{\varepsilon}) &  &
    \LL \tthreeone{Y_{\varepsilon}} & \assign & \tthree{Y_{\varepsilon}},\\
    \ttwo{Y_{\varepsilon}} & \assign & \frac{1}{3} \varepsilon^{- 1}
    \tilde{F}^{(1)}_{\varepsilon} (\varepsilon^{\frac{1}{2}} Y_{\varepsilon})
    &  & \LL \ttwoone{Y_{\varepsilon}} & \assign & \ttwo{Y_{\varepsilon}}\\
    \tone{Y_{\varepsilon}} & \assign & \frac{1}{6} \varepsilon^{- \frac{1}{2}}
    \tilde{F}^{(2)}_{\varepsilon} (\varepsilon^{\frac{1}{2}} Y_{\varepsilon})
    &  & \tzero{Y_{\varepsilon}} & \assign & \frac{1}{6}
    \tilde{F}^{(3)}_{\varepsilon} (\varepsilon^{\frac{1}{2}}
    Y_{\varepsilon})\\
    \ttwothreer{\bar{Y}_{\varepsilon}} & \assign &
    \ttwoone{\bar{Y}_{\varepsilon}} \circ \ttwo{Y_{\varepsilon}} -
    \ttwothreer{\bar{d}_{\varepsilon}} &  & \ttwothreer{Y_{\varepsilon}} &
    \assign & \ttwoone{Y_{\varepsilon}} \circ \ttwo{Y_{\varepsilon}} -
    \ttwothreer{d_{\varepsilon}},\\
    \tthreetwor{Y_{\varepsilon}} & \assign & \tthreeone{Y_{\varepsilon}} \circ
    \tone{Y_{\varepsilon}} - \tthreetwor{d_{\varepsilon}}, &  &
    \tthreethreer{Y_{\varepsilon}} & \assign & \tthreeone{Y_{\varepsilon}}
    \circ \ttwo{Y_{\varepsilon}} - \tthreethreerprime{d_{\varepsilon}}
    Y_{\varepsilon} - \tthreethreer{d_{\varepsilon}},
  \end{array} \label{e:trees-def}
\end{equation}
with $Y_{\varepsilon}$ stationary solution, while $Y_{\varepsilon},
\tthreeone{Y_{\varepsilon}}, \ttwoone{Y_{\varepsilon}},
\ttwoone{\bar{Y}_{\varepsilon}}$ have $0$ initial condition in $t = 0$. The
last four trees $\ttwothreer{\bar{Y}_{\varepsilon}}$,
$\ttwothreer{Y_{\varepsilon}}$, $\tthreetwor{Y_{\varepsilon}}$,
$\tthreethreer{Y_{\varepsilon}}$ are obtained from the others via the resonant
Bony's paraproduct $\circ$ recalled in Appendix~\ref{a:basics}, and
$\ttwothreer{\bar{d}_{\varepsilon}}$, $\ttwothreer{d_{\varepsilon}}$,
$\tthreetwor{d_{\varepsilon}}$, $\tthreethreerprime{d_{\varepsilon}}$,
$\tthreethreer{d_{\varepsilon}}$ are just $\varepsilon$-dependent constants
whose exact value will matter only in Section~\ref{s:comparing}. Indeed, in
the scope of this section we only need the following relation to be verified:
\begin{equation}
  \begin{array}{lll}
    \tthreethreerprime{d_{\varepsilon}} & = & 2 \tthreetwor{d_{\varepsilon}} + 3
    \ttwothreer{d_{\varepsilon}} .
  \end{array} \label{e:ren-constraint}
\end{equation}
The notation $\ttwo{\bar{Y}_{\varepsilon}}$ denotes that this tree has finite
chaos expansion and can be treated with the well-known techniques
of~{\cite{catellier_paracontrolled_2013}}
or~{\cite{mourrat_construction_2016}} (we put a bar on
$\ttwothreer{\bar{Y}_{\varepsilon}}$ just because is the only tree obtained
from $\ttwo{\bar{Y}_{\varepsilon}}$)

With the definitions~(\ref{e:trees-def}), equation~(\ref{eq:first-exp}) takes
the form
\begin{equation}
  \begin{array}{lll}
    \LL v_{\varepsilon} & = & Y_{\varepsilon} - \ttwo{\bar{Y}_{\varepsilon}} -
    \tthree{Y_{\varepsilon}} - 3 \ttwo{Y_{\varepsilon}} v_{\varepsilon} - 3
    \tone{Y_{\varepsilon}} v_{\varepsilon}^2 - \tzero{Y_{\varepsilon}}
    v_{\varepsilon}^3\\
    &  & - \varepsilon^{- 3 / 2} f_{0, \varepsilon} - \varepsilon^{- 1} f_{1,
    \varepsilon}  (Y_{\varepsilon} + v_{\varepsilon}) - \varepsilon^{- 1 / 2}
    f_{2, \varepsilon}  (2 Y_{\varepsilon} v_{\varepsilon} +
    v_{\varepsilon}^2) - R_{\varepsilon} (v_{\varepsilon}) .
  \end{array} \label{eq1}
\end{equation}

At this point it is worth noting that the trivial case
$\tilde{F}_{\varepsilon} (x) = H_3 (x, \sigma^2_{\varepsilon})$ yields
$\tthree{Y_{\varepsilon}} = \llbracket Y_{\varepsilon}^3 \rrbracket$,
$\ttwo{Y_{\varepsilon}} = \llbracket Y^2_{\varepsilon} \rrbracket$,
$\tone{Y_{\varepsilon}} = Y_{\varepsilon}$, $\tzero{Y} = 1$. By comparing
these random fields to the ones defined
in~{\cite{catellier_paracontrolled_2013}} we can \tmtextit{guess} that
$\tthreeone{Y_{\varepsilon}}, \ttwo{Y_{\varepsilon}}, \tone{Y_{\varepsilon}},
\tzero{Y}$ can be controlled respectively in $\CC^{1 / 2 - \kappa}, \CC^{- 1 -
\kappa}, \CC^{- 1 / 2 - \kappa}, \CC^{- \kappa}$ \ $\forall \kappa > 0$ for
any $F_{\varepsilon}$ satisfying Assumption~\ref{a:main}, and carry on the
paracontrolled analysis of~(\ref{eq1}) as if it were the case. Clearly, the
paracontrolled structure is robust and does not depend on how the terms
$\tthreeone{Y_{\varepsilon}}, \ttwo{Y_{\varepsilon}}, \tone{Y_{\varepsilon}},
\tzero{Y}$ are defined as long as they have the desired regularity.

From these observations, we do not expect to be able to control the products
$\ttwo{Y_{\varepsilon}} v_{\varepsilon}$, $\tone{Y_{\varepsilon}}
v_{\varepsilon}^2$ and $Y_{\varepsilon} v_{\varepsilon}$ in eq.~(\ref{eq1})
uniformly in $\varepsilon > 0$. In order to proceed with the analysis we make
the Ansatz:
\begin{equation}
  \begin{array}{lll}
    u_{\varepsilon} & = & Y_{\varepsilon} + v_{\varepsilon},\\
    v_{\varepsilon} & = & - \tthreeone{Y_{\varepsilon}} -
    \ttwoone{\bar{Y}_{\varepsilon}} - 3 v_{\varepsilon} \precprec
    \ttwoone{Y_{\varepsilon}} + v^{\natural}_{\varepsilon}
  \end{array} \label{e:ansatz-univ}
\end{equation}
and proceed to decompose the ill-defined products using the paracontrolled
techniques recalled in Appendix~\ref{a:basics}. We start by writing
$v_{\varepsilon} \ttwo{Y_{\varepsilon}} = v_{\varepsilon} \prec
\ttwo{Y_{\varepsilon}} + v_{\varepsilon} \succ \ttwo{Y_{\varepsilon}} +
v_{\varepsilon} \circ \ttwo{Y_{\varepsilon}}$. The resonant term, together
with Ansatz~(\ref{e:ansatz-univ}), yields:
\begin{eqnarray*}
  v_{\varepsilon} \circ \ttwo{Y_{\varepsilon}} & = & -
  \tthreeone{Y_{\varepsilon}} \circ \ttwo{Y_{\varepsilon}} -
  \ttwoone{\bar{Y}_{\varepsilon}} \circ \ttwo{Y_{\varepsilon}} - 3
  v_{\varepsilon}  (\ttwoone{Y_{\varepsilon}} \circ \ttwo{Y_{\varepsilon}})\\
  &  & - 3 \overline{\tmop{com}}_1 (v_{\varepsilon},
  \ttwoone{Y_{\varepsilon}}, \ttwo{Y_{\varepsilon}}) +
  v^{\natural}_{\varepsilon} \circ \ttwo{Y_{\varepsilon}},
\end{eqnarray*}
with the definition and bounds of $\overline{\tmop{com}}_1 (\cdummy, \cdummy,
\cdummy)$ given in Lemma~\ref{lem:all-commutators}. Then we define
\[ \begin{array}{lll}
     \ttwo{Y_{\varepsilon}} \hat{\diamond} v_{\varepsilon} & \assign &
     v_{\varepsilon} \ttwo{Y_{\varepsilon}} - v_{\varepsilon} \prec
     \ttwo{Y_{\varepsilon}} + (3 v_{\varepsilon}  \ttwothreer{d_{\varepsilon}}
     + \tthreethreerprime{d_{\varepsilon}} Y_{\varepsilon} +
     \tthreethreer{d_{\varepsilon}} + \ttwothreer{\bar{d}_{\varepsilon}})\\
     & = & v_{\varepsilon} \succ \ttwo{Y_{\varepsilon}} -
     \ttwothreer{\bar{Y}_{\varepsilon}} - \tthreethreer{Y_{\varepsilon}} - 3
     v_{\varepsilon}  \ttwothreer{Y_{\varepsilon}} +
     v^{\natural}_{\varepsilon} \circ \ttwo{Y_{\varepsilon}} - 3
     \overline{\tmop{com}}_1 (v_{\varepsilon}, \ttwoone{Y_{\varepsilon}},
     \ttwo{Y_{\varepsilon}}) .
   \end{array} \]
Moreover we have for $v_{\varepsilon} Y_{\varepsilon}$:
\[ \begin{array}{lll}
     v_{\varepsilon} Y_{\varepsilon} & = & \varphi_{\varepsilon}
     Y_{\varepsilon} - \tthreeone{Y_{\varepsilon}} \prec Y_{\varepsilon} -
     \tthreeone{Y_{\varepsilon}} \succ Y_{\varepsilon} -
     \tthreeone{Y_{\varepsilon}} \circ Y_{\varepsilon},
   \end{array} \]
where we introduced the shorthand $\varphi_{\varepsilon} \assign
v_{\varepsilon} + \tthreeone{Y_{\varepsilon}}$. So we let
\[ v_{\varepsilon} \diamond Y_{\varepsilon} \assign v_{\varepsilon}
   Y_{\varepsilon} + \tthreetwor{d_{\varepsilon}} = \varphi_{\varepsilon}
   Y_{\varepsilon} - \tthreeone{Y_{\varepsilon}} \prec Y_{\varepsilon} -
   \tthreeone{Y_{\varepsilon}} \succ Y_{\varepsilon} -
   \tthreetwor{Y_{\varepsilon}}, \]
Finally to analyse the product $\tone{Y_{\varepsilon}} v_{\varepsilon}^2$ we
write
\[ \tone{Y_{\varepsilon}} v_{\varepsilon}^2 = \tone{Y_{\varepsilon}}
   (\tthreeone{Y_{\varepsilon}})^2 - 2 \tone{Y_{\varepsilon}}
   \tthreeone{Y_{\varepsilon}} \varphi_{\varepsilon} + \tone{Y_{\varepsilon}}
   \varphi_{\varepsilon}^2, \]
and consider the products involving only $Y^{\tau}$ factors: first
\[ \tone{Y_{\varepsilon}} \tthreeone{Y_{\varepsilon}} = \tone{Y_{\varepsilon}}
   \succ \tthreeone{Y_{\varepsilon}} + \tone{Y_{\varepsilon}} \prec
   \tthreeone{Y_{\varepsilon}} + \tthreetwor{Y_{\varepsilon}} +
   \tthreetwor{d_{\varepsilon}} = : \tone{Y_{\varepsilon}} \diamond
   \tthreeone{Y_{\varepsilon}} + \tthreetwor{d_{\varepsilon}}, \]
and then we define the term $\tone{Y_{\varepsilon}} \diamond
(\tthreeone{Y_{\varepsilon}})^2$ as follows:
\begin{eqnarray*}
  \tone{Y_{\varepsilon}} \diamond (\tthreeone{Y_{\varepsilon}})^2 & \assign &
  \tone{Y_{\varepsilon}} (\tthreeone{Y_{\varepsilon}})^2 - 2
  \tthreetwor{d_{\varepsilon}} \tthreeone{Y_{\varepsilon}}\\
  & = & \tone{Y_{\varepsilon}} \prec (\tthreeone{Y_{\varepsilon}})^2 +
  \tone{Y_{\varepsilon}} \succ (\tthreeone{Y_{\varepsilon}})^2 +
  \tone{Y_{\varepsilon}} \circ (\tthreeone{Y_{\varepsilon}} \circ
  \tthreeone{Y_{\varepsilon}}) + 2 \tmop{com}_1 (\tthreeone{Y_{\varepsilon}},
  \tthreeone{Y_{\varepsilon}}, \tone{Y_{\varepsilon}}) + 2
  \tthreeone{Y_{\varepsilon}} \tthreetwor{Y_{\varepsilon}},
\end{eqnarray*}
so that
\[ \tone{Y_{\varepsilon}} \diamond v_{\varepsilon}^2 \assign
   \tone{Y_{\varepsilon}} v_{\varepsilon}^2 + 2 \tthreetwor{d_{\varepsilon}}
   v_{\varepsilon} = \tone{Y_{\varepsilon}} \diamond
   (\tthreeone{Y_{\varepsilon}})^2 - 2 (\tone{Y_{\varepsilon}} \diamond
   \tthreeone{Y_{\varepsilon}}) \varphi_{\varepsilon} + \tone{Y_{\varepsilon}}
   \varphi_{\varepsilon}^2 . \]
We note also that
\[ \LL v_{\varepsilon} = - \LL \tthreeone{Y_{\varepsilon}} - \LL
   \ttwoone{\bar{Y}_{\varepsilon}} + \LL v^{\natural}_{\varepsilon} - 3
   v_{\varepsilon} \prec \LL \ttwoone{Y_{\varepsilon}} - 3 \tmop{com}_3
   (v_{\varepsilon}, \ttwoone{Y_{\varepsilon}}) - 3 \tmop{com}_2
   (v_{\varepsilon}, \ttwo{Y_{\varepsilon}}), \]
with $\tmop{com}_2 (\cdummy, \cdummy)$ and $\tmop{com}_3 (\cdummy, \cdummy)$
specified in Lemma~\ref{lem:all-commutators}. Substituting
these renormalized products into (\ref{eq1}) and recalling the
definition~(\ref{e:def-lambda}) for $\lambda_{\varepsilon} = (\lambda_{0,
\varepsilon}, \lambda_{1, \varepsilon}, \lambda_{2, \varepsilon}, \lambda_{3,
\varepsilon})$, we obtain the following equation for
$v^{\natural}_{\varepsilon}$ :
\begin{eqnarray*}
  \LL v_{\varepsilon}^{\natural} & = & 3 \tmop{com}_3 (v_{\varepsilon},
  \ttwoone{Y_{\varepsilon}}) + 3 \tmop{com}_2 (v_{\varepsilon},
  \ttwo{Y_{\varepsilon}})\\
  &  & - \tzero{Y_{\varepsilon}} v_{\varepsilon}^3 - 3 \tone{Y_{\varepsilon}}
  \diamond v_{\varepsilon}^2 - 3 \ttwo{Y_{\varepsilon}} \hat{\diamond}
  v_{\varepsilon}\\
  &  & + Y_{\varepsilon} - \lambda_{2, \varepsilon} (2 v_{\varepsilon}
  \diamond Y_{\varepsilon} + v_{\varepsilon}^2)\\
  &  & - \lambda_{1, \varepsilon}  (Y_{\varepsilon} + v_{\varepsilon}) + [9
  \ttwothreer{d_{\varepsilon}} + 6 \tthreetwor{d_{\varepsilon}} - 3
  \tthreethreerprime{d_{\varepsilon}}] v_{\varepsilon} - \lambda_{0, \varepsilon}
  - R_{\varepsilon} (v_{\varepsilon}),
\end{eqnarray*}
where we can use the constraint~(\ref{e:ren-constraint}) to remove the term
proportional to $v_{\varepsilon}$. Summarizing, we obtain the following
equation, together with Ansatz~(\ref{e:ansatz-univ}):
\begin{equation}
  \left\{ \begin{array}{lll}
    v_{\varepsilon} & = & - \tthreeone{Y_{\varepsilon}} -
    \ttwoone{\bar{Y}_{\varepsilon}} - 3 v_{\varepsilon} \precprec
    \ttwoone{Y_{\varepsilon}} + v^{\natural}\\
    \LL v_{\varepsilon}^{\natural} & = & U (\lambda_{\varepsilon},
    \mathbbm{Y}_{\varepsilon} ; v_{\varepsilon}, v_{\varepsilon}^{\natural}) -
    R_{\varepsilon} (v_{\varepsilon})
  \end{array} \right. \label{e:last-phisharp}
\end{equation}
with initial condition $v_{\varepsilon, 0} = u_{0, \varepsilon} -
Y_{\varepsilon} (0)$ and $U$ given by
\begin{equation}
  \begin{array}{lll}
    U (\lambda_{\varepsilon}, \mathbbm{Y}_{\varepsilon} ; v_{\varepsilon},
    v_{\varepsilon}^{\natural}) & : = & 3 \tmop{com}_3 (v_{\varepsilon},
    \ttwoone{Y_{\varepsilon}}) + 3 \tmop{com}_2 (v_{\varepsilon},
    \ttwo{Y_{\varepsilon}}) - \tzero{Y_{\varepsilon}} v_{\varepsilon}^3\\
    &  & - 3 \tone{Y_{\varepsilon}} \diamond v_{\varepsilon}^2 - 3
    \ttwo{Y_{\varepsilon}} \hat{\diamond} v_{\varepsilon} + Y_{\varepsilon} -
    \lambda_{2, \varepsilon} (2 v_{\varepsilon} \diamond Y_{\varepsilon} +
    v_{\varepsilon}^2)\\
    &  & - \lambda_{1, \varepsilon}  (Y_{\varepsilon} + v_{\varepsilon}) -
    \lambda_{0, \varepsilon} - R_{\varepsilon} (v_{\varepsilon}) .
  \end{array} \label{e:def-bigU}
\end{equation}
The enhanced noise vector $\mathbbm{Y}_{\varepsilon}$ is defined by
\begin{eqnarray}
  \mathbbm{Y}_{\varepsilon} & \assign & (\tzero{Y_{\varepsilon}},
  \tone{Y_{\varepsilon}}, \ttwo{Y_{\varepsilon}},
  \ttwo{\bar{Y}_{\varepsilon}}, \tthreeone{Y_{\varepsilon}},
  \tthreetwor{Y_{\varepsilon}}, \ttwothreer{Y_{\varepsilon}},
  \ttwothreer{\bar{Y}_{\varepsilon}}, \tthreethreer{Y_{\varepsilon}})
  \nonumber\\
  &  &  \nonumber\\
  \mathcal{X}_T & \assign & C_T^{} \CC^{- \kappa} \times C_T \CC^{-
  \frac{1}{2} - \kappa} \times \left( C_T \CC^{- 1 - \kappa} \right)^2 \times
  \LL_T^{1 / 2 - \kappa} \times \left( C_T \CC^{- \kappa} \right)^3 \times C_T
  \CC^{- \frac{1}{2} - \kappa}  \label{e:y-space}
\end{eqnarray}
for every $\kappa > 0$, $T > 0$. We use the notation $\|
\mathbbm{Y}_{\varepsilon} \|_{\mathcal{X}_T} = \sum_{\tau} \|
\mathbbm{Y}^{\tau}_{\varepsilon} \|_{\mathcal{X}^{\tau}}$ for the associated
norm where $Y_{\varepsilon}^{\tau}$ is a generic tree in
$\mathbbm{Y}_{\varepsilon}$. The homogeneities $| \tau | \in \mathbbm{R}$ are
given by
\[ \begin{array}{|c|c|c|c|c|c|c|c|c|c|c|}
     \hline
     Y_{\varepsilon}^{\tau} & = & \tzero{Y_{\varepsilon}} &
     \tone{Y_{\varepsilon}} & \ttwo{Y_{\varepsilon}} &
     \ttwo{\bar{Y}_{\varepsilon}} & \tthreeone{Y_{\varepsilon}} &
     \tthreetwor{Y_{\varepsilon}} & \ttwothreer{Y_{\varepsilon}} &
     \ttwothreer{\bar{Y}_{\varepsilon}} & \tthreethreer{Y_{\varepsilon}}\\
     \hline
     | \tau | & = & 0 & - 1 / 2 & - 1 & - 1 & 1 / 2 & 0 & 0 & 0 & - 1 / 2\\
     \hline
   \end{array} \]
Note that for every $\varepsilon > 0$ eq.~(\ref{e:last-phisharp}) is
equivalent to eq.~(\ref{e:univ}).

\begin{remark}
  \label{r:solmap-cont} The paracontrolled structure we developed in this
  section is the same as in the work of Catellier and
  Chouk~{\cite{catellier_paracontrolled_2013}}, plus an additive source term
  (which is $R_{\varepsilon} (v_{\varepsilon})$ in
  equation~(\ref{e:last-phisharp})). Therefore, there exists $T = T \left( \|
  \mathbbm{Y}_{\varepsilon} \|_{\mathcal{X}_T}, \| u_{\varepsilon, 0}
  \|_{\CC^{- 1 / 2 - \kappa}}, | \lambda_{\varepsilon} | \right)$ such that we
  can define for $\alpha \in (1 / 2, 2 / 3)$, $p \in [4, \infty)$, $\gamma >
  \frac{1}{4} + \frac{3}{2} \kappa$ a solution map
  \[ \begin{array}{llll}
       \Gamma : & \CC^{- 1 / 2 - \kappa} \times \mathcal{X}^{\tau} \times
       \mathbbm{R}^4 \times \mathcal{M}^{\gamma, p}_T L^p (\mathbbm{T}^3) &
       \rightarrow & C_T \CC^{- \alpha} (\mathbbm{T}^3)\\
       & (u_{\varepsilon, 0}, \mathbbm{Y}_{\varepsilon},
       \lambda_{\varepsilon}, \mathcal{R}) & \mapsto & u_{\varepsilon}
     \end{array} \]
  so that $u_{\varepsilon} = \Gamma (u_{\varepsilon, 0},
  \mathbbm{Y}_{\varepsilon}, \lambda_{\varepsilon}, \mathcal{R})$ with
  $u_{\varepsilon} = Y_{\varepsilon} + v_{\varepsilon}$ and $v_{\varepsilon}$
  that solves~(\ref{e:last-phisharp}) with the remainder~$R_{\varepsilon}
  (v_{\varepsilon})$ replaced by~$\mathcal{R}$. The space
  $\mathcal{M}^{\gamma}_T L^p (\mathbbm{T}^3)$ is specified in
  Appendix~\ref{s:notations}. Indeed, we can use Lemma~\ref{l:L-interpol} and
  Lemma~\ref{l:lp-integ} to control $I\mathcal{R}$ as
  \[ \|I\mathcal{R}\|_{\LL^{1 / 2, 1 + 2 \kappa}_T} \lesssim T^{\delta} \|
     \mathcal{R} \|_{\mathcal{M}^{\gamma, p}_T L^p} \]
  for $\delta > 0$ small enough, and thanks to this bound it is easy to see
  that the the fixed point procedure of Section~3
  of~{\cite{catellier_paracontrolled_2013}} still holds with a
  \tmtextit{fixed} additive source term~$\mathcal{R}$. In the same way, the
  continuity of the solution map~$\Gamma$ follows easily as in Theorem~1.2
  of~{\cite{catellier_paracontrolled_2013}}. 
\end{remark}

\subsection{Identification of the limit}\label{s:limit-ident}

In order to identify interesting limits for equation~(\ref{e:univ}), we
introduce $\forall \lambda = (\lambda_0, \lambda_1, \lambda_2, \lambda_3) \in
\mathbbm{R}^4$ the enhanced noise $\mathbbm{Y} (\lambda)$ which is constructed
from universal noises $X^{\tau}$ as
\begin{equation}
  \begin{array}{lll}
    \mathbbm{Y} (\lambda) & \assign & (\tzero{Y} (\lambda), \tone{Y}
    (\lambda), \ttwo{Y} (\lambda), \ttwo{\bar{Y}} (\lambda), \tthreeone{Y}
    (\lambda), \tthreetwor{Y} (\lambda), \ttwothreer{Y} (\lambda),
    \ttwothreer{Y} (\lambda), \tthreethreer{Y} (\lambda))\\
    & \assign & (\lambda_3, \lambda_3 X, \lambda_3 \ttwo{X}, \lambda_2
    \ttwo{X}, \lambda_3 \tthreeone{X}, (\lambda_3)^2 \tthreetwor{X},
    (\lambda_3)^2 \ttwothreer{X}, \lambda_3 \lambda_2 \ttwothreer{X},
    (\lambda_3)^2 \tthreethreer{X})
  \end{array} \label{e:def-ylambda}
\end{equation}
where $X$ is the stationary solution to to the linear equation $\LL X = - X +
\xi$ and $\xi$ is the time-space white noise on $\mathbbm{R} \times
\mathbbm{T}^3$. We will sometimes use the shorter notation $\mathbbm{Y}
(\lambda) = (Y^{\tau} (\lambda))_{\tau}$ for~(\ref{e:def-ylambda}).

We define the universal fields $X^{\tau}$ through their Littlewood-Paley
decomposition $\forall (t, \bar{x}) \in \mathbbm{R}^+ \times \mathbbm{T}^3$
as:
\begin{equation}
  \begin{array}{lll}
    \tthree{X} & \assign & \llbracket X^3 \rrbracket, \qquad \LL \tthreeone{X}
    = \tthree{X} \quad \text{with $\tthreeone{X} (t = 0) = 0$},\\
    \ttwo{X} & \assign & \llbracket X^2 \rrbracket,\\
    \Delta_q \ttwothreer{X} (t, \bar{x}) & \assign & \Delta_q (1 - J_0)
    (\ttwoone{X} \circ \ttwo{X}) (t, \bar{x}) = \int_{\zeta_1, \zeta_2}  (1 -
    J_0) (\llbracket X^2 (\zeta_1) \rrbracket \llbracket X^2 (\zeta_2)
    \rrbracket) \mu_{q, \zeta_1, \zeta_2},\\
    \Delta_q \tthreetwor{X} (t, \bar{x}) & \assign & \Delta_q (\tthreeone{X}
    \circ X) (t, \bar{x}) = \int_{\zeta_1, \zeta_2}  \llbracket X^3 (\zeta_1)
    \rrbracket X (\zeta_2) \mu_{q, \zeta_1, \zeta_2},\\
    \Delta_q \tthreethreer{X} (t, \bar{x}) & \assign & \int_{\zeta_1, \zeta_2}
    (1 - J_1) (\llbracket X^3 (\zeta_1) \rrbracket \llbracket X^2 (\zeta_2)
    \rrbracket) \mu_{q, \zeta_1, \zeta_2} \\
    &  & + 6 \int_{s_{}, x} [\Delta_q X (t + s, \bar{x} - x) - \Delta_q X (t,
    \bar{x})] P_s (x) [C (s, x)]^2,
  \end{array} \label{e:stoch-usual}
\end{equation}
where as before $\llbracket \cdummy \rrbracket$ stands for the Wick product,
$\zeta_i = (x_i, s_i) \in \mathbbm{R} \times \mathbbm{T}^3$, $C (\cdummy,
\cdummy)$ is the covariance of $X$ and $\mu_{q, \zeta_1, \zeta_2}$ is the
measure
\[ \mu_{q, \zeta_1, \zeta_2} \assign \delta (t - s_2) \mathd \zeta_1 \mathd
   \zeta_2 \int_{x, y} K_{q, \bar{x}} (x) \sum_{i \sim j} K_{i, x} (y) K_{j,
   x} (x_2) P_{t - s_1} (y - x_1) \]
with the usual heat kernel $P_t (x) = \frac{1}{(4 \mathpi t)^{3 / 2}} e^{-
\frac{| x |^2}{4 t}}  \mathbbm{1}_{t \geqslant 0}$. We commit an abuse of
notation by writing $X (\zeta)$ since $X$ is actually a distribution in space:
the integrals in~(\ref{e:stoch-usual}) should obviously be intended as
functionals.

Standard computations (see e.g.~{\cite{catellier_paracontrolled_2013}}
or~{\cite{mourrat_construction_2016}}) show that, $\forall \lambda \in
\mathbbm{R}^4$ and for any $T > 0$, $0 < \kappa < \kappa'$
\[ \mathbbm{Y} (\lambda) \in C^{\kappa}_T \CC^{- \frac{1}{2} - 2 \kappa'}
   \times \left( C^{\kappa}_T \CC^{- 1 - 2 \kappa'} \right)^2 \times
   C^{\kappa}_T \CC^{\frac{1}{2} - 2 \kappa'} \times \left( C^{\kappa}_T
   \CC^{0 - 2 \kappa'} \right)^2 \times C^{\kappa}_T \CC^{- \frac{1}{2} - 2
   \kappa'}_T, \]
almost surely.

Using the paracontrolled structure we developed in
Section~\ref{s:paracontrolled-struc} we can identify the limiting solution $u
(\lambda)$ introduced in Theorem~\ref{t:maintheorem}.

\begin{theorem}
  \label{t:lim-ident} The family of random fields $u_{\varepsilon}$ given by
  the solutions of eq.~(\ref{e:univ}) converges in law and locally in time to
  a limiting random field $u (\lambda)$ in the space $C_T \CC^{- \alpha}
  (\mathbbm{T}^3)$ for every $1 / 2 < \alpha < 2 / 3$. The limiting random
  field $u (\lambda)$ solves the paracontrolled equation
  \begin{equation}
    \left\{ \begin{array}{lll}
      u (\lambda) & = & X + v (\lambda)\\
      v (\lambda) & = & - \lambda_3 \tthreeone{X} - \lambda_2 \ttwoone{X} - 3
      \lambda_3 v (\lambda) \precprec \ttwoone{X} + v^{\natural} (\lambda)\\
      \LL v^{\natural} (\lambda) & = & U (\lambda, \mathbbm{Y} (\lambda) ; v
      (\lambda), v^{\natural} (\lambda))\\
      v^{\natural} (\lambda) (t = 0) & = & v_0 + \lambda_3 \tthreeone{X} (t =
      0) + \lambda_2 \ttwoone{X} (t = 0) + 3 \lambda_3 v_{\varepsilon, 0}
      \prec \ttwoone{X} (t = 0)
    \end{array} \right. \label{eq:final-system}
  \end{equation}
  with $U$ defined in~(\ref{e:def-bigU}) and $v_0 = u_0 - X (t = 0)$.
\end{theorem}

\begin{proof}
  Fix $T > 0$. Let $u_{\varepsilon} = Y_{\varepsilon} + v_{\varepsilon}$ be
  the solution of eq.~(\ref{e:univ}) for fixed $\varepsilon > 0$, which is
  seen to be unique in the ($\varepsilon$-dependent) time interval $[0,
  T_{\varepsilon}]$ by a usual fixed-point argument on the original equation
  (without resorting to the paracontrolled decomposition). Let
  $u_{\varepsilon} = \Gamma (u_{\varepsilon, 0}, \mathbbm{Y}_{\varepsilon},
  \lambda_{\varepsilon}, R_{\varepsilon} (v_{\varepsilon}))$ on $[0,
  T_{\varepsilon}]$ with $\Gamma$ defined as in Remark~\ref{r:solmap-cont} and
  $R_{\varepsilon} (v_{\varepsilon})$ seen as an exogenous source term. We
  know from the a-priori estimations of Section~\ref{s:a-priori} that there
  exists a time $T_{\star} = T_{\star} \left( \| \mathbbm{Y}_{\varepsilon}
  \|_{\mathcal{X}_T}, \| u_{\varepsilon, 0} \|_{\CC^{- 1 / 2 - \kappa}}, |
  \lambda_{\varepsilon} | \right)$ and a family of events
  $(\mathcal{E}_{\varepsilon})_{\varepsilon > 0}$ such that $\mathbbm{P}
  (\mathcal{E}_{\varepsilon}) \rightarrow 1$ for $\varepsilon \rightarrow 0$
  and we can control $\| v_{\varepsilon} \|_{\mathcal{M}^{1 / 4 + 3 \kappa /
  2}_{T_{\star}} L^{\infty}}$. Thus, we can control the $L^{\infty}$ norm of
  $v_{\varepsilon} (t)$ in $[T_{\varepsilon} / 2, T_{\varepsilon}]$ and extend
  the solution $v_{\varepsilon}$ on $[0, T_{\star}]$ for every $\varepsilon$.
  Denote by $u_{\varepsilon}^{\star} = \Gamma_{\star} (u_{\varepsilon, 0},
  \mathbbm{Y}_{\varepsilon}, \lambda_{\varepsilon}, R_{\varepsilon}
  (v_{\varepsilon}))$ the process $u_{\varepsilon}$ stopped at time
  $T_{\star}$ and $\Gamma_{\star}$ the corresponding stopped solution map. \
  
  Note that $u (\lambda)$ solves the same equation as
  $u_{\varepsilon}^{\star}$ with $\mathbbm{Y}_{\varepsilon}$ replaced by
  $\mathbbm{Y} (\lambda)$, $u_{\varepsilon, 0}$ replaced by $u_0$,
  $\lambda_{\varepsilon}$ replaced by $\lambda$ and $R_{\varepsilon}
  (v_{\varepsilon}) = 0$. So $u (\lambda) = u^{\star} = \Gamma_{\star} (u_0,
  \mathbbm{Y} (\lambda), \lambda, 0)$ up to time $T_{\star}$. Let us introduce
  the random field $\bar{u}_{\varepsilon}^{\star} = \Gamma_{\star}
  (u_{\varepsilon, 0}, \mathbbm{Y}_{\varepsilon}, \lambda_{\varepsilon}, 0)$
  with $\bar{v}^{\star}_{\varepsilon} = \bar{u}_{\varepsilon}^{\star} -
  Y_{\varepsilon}$ that solves the paracontrolled
  equation~(\ref{e:last-phisharp}) but with remainder $R_{\varepsilon}
  (v_{\varepsilon}) = 0$.
  
  Consider the $n$-tuple of random variables $(u_{\varepsilon, 0},
  \mathbbm{Y}_{\varepsilon}, u_{\varepsilon}^{\star},
  \bar{u}_{\varepsilon}^{\star})$ and let $\mu_{\varepsilon}$ be its law on
  $\mathcal{Z}= \CC^{- \alpha} \times \mathcal{X}_T \times \left( C_T \CC^{-
  \alpha} \right)^2$ conditionally on $\mathcal{E}_{\varepsilon}$. Observe
  that $\Gamma_{\star}$ is continuous as discussed in
  Remark~\ref{r:solmap-cont}, and this gives that $\forall \delta > 0$,
  $\mu_{\varepsilon} \left( \| u_{\varepsilon}^{\star} -
  \bar{u}_{\varepsilon}^{\star} \|_{C_T \CC^{- \alpha}} > \delta \right)
  \rightarrow 0$ as $\varepsilon \rightarrow 0$. Indeed $R_{\varepsilon}
  (v_{\varepsilon}) \rightarrow 0$ in probability in the space
  $\mathcal{M}^{\gamma, p}_{T_{\star}} L^p (\mathbbm{T}^3)$ \ by
  Lemma~\ref{l:remconv-last}. This shows that $\mu_{\varepsilon}$ concentrates
  on $\CC^{- \alpha} \times \mathcal{X}_T \times \left\{ (z, z) \in \left( C_T
  \CC^{- \alpha} \right)^2 \right\}$. Let $\mu$ any accumulation point of
  $(\mu_{\varepsilon})_{\varepsilon}$. Then $\mu \left( \CC^{- \alpha} \times
  \mathcal{X}_T \times \left\{ (z, z) \in \left( C_T \CC^{- \alpha} \right)^2
  \right\} \right) = 1$. The apriori estimations of Section~\ref{s:a-priori}
  yield the tightness of $\mu_{\varepsilon}$ and from the concentration of
  $\mu$ on the diagonal we know that there exists a subsequence such that for
  any test function $\varphi$,
  \begin{equation}
    \int_{\mathcal{Z}} \varphi (x, y, z, t) \mathd \mu_{\varepsilon} (x, y, z,
    t) \rightarrow \int_{\mathcal{Z}} \varphi (x, y, z, t) \mathd \mu_{} (x,
    y, z, t) = \int_{\mathcal{Z}} \varphi (x, y, t, t) \mathd \mu_{} (x, y, z,
    t) . \label{e:conv-law-1}
  \end{equation}
  Moreover, still along subsequences we have that for any bounded continuous
  function $\varphi$
  \[ \mathbbm{E} (\varphi (u_{\varepsilon, 0}, \mathbbm{Y}_{\varepsilon},
     \bar{u}_{\varepsilon}^{\star})) =\mathbbm{E} (\varphi (u_{\varepsilon,
     0}, \mathbbm{Y}_{\varepsilon}, \Gamma_{\star} (u_{\varepsilon, 0},
     \mathbbm{Y}_{\varepsilon}, \lambda_{\varepsilon}, 0))) \rightarrow
     \mathbbm{E} (\varphi (u_0, \mathbbm{Y} (\lambda), \Gamma_{\star} (u_0,
     \mathbbm{Y} (\lambda), \lambda, 0))) \]
  since by Theorem~\ref{t:stoch-conv} the vector $\mathbbm{Y}_{\varepsilon}$
  converges in law to $\mathbbm{Y} (\lambda)$, $u_{\varepsilon, 0}$ to $u_0$,
  and $\Gamma_{\star}$ is a continuous function as discussed in
  Remark~\ref{r:solmap-cont}. This shows that
  \begin{equation}
    \int_{\mathcal{Z}} \varphi (x, y, t, t) \mathd \mu_{\varepsilon} (x, y, z,
    t) \rightarrow \int_{\mathcal{Z}} \varphi (x, y, \Gamma_{\star} (x, y,
    \lambda, 0), \Gamma_{\star} (x, y, \lambda, 0)) \mathd \mu (x, y, z, t)
    \label{e:conv-law-2}
  \end{equation}
  and then by comparing~(\ref{e:conv-law-1}) and~(\ref{e:conv-law-2}) we can
  conclude that there exists a subsequence such that
  \[ \int_{\mathcal{Z}} \varphi (x, y, z, t) \mathd \mu_{\varepsilon} (x, y,
     z, t) \rightarrow \int_{\mathcal{Z}} \varphi (x, y, \Gamma_{\star} (x, y,
     \lambda, 0), \Gamma_{\star} (x, y, \lambda, 0)) \mathd \mu (x, y, z, t) .
  \]
  We can identify the limit distribution $\mu$ by noting that since
  $\mathbbm{P} (\mathcal{E}_{\varepsilon}) \rightarrow 1$ we have
  \[ \mathbbm{E} [\psi (u_{\varepsilon, 0}, \mathbbm{Y}_{\varepsilon})
     |\mathcal{E}_{\varepsilon}] = \frac{\mathbbm{E} [\psi (u_{\varepsilon,
     0}, \mathbbm{Y}_{\varepsilon})
     \mathbbm{I}_{\mathcal{E}_{\varepsilon}}]}{\mathbbm{P}
     (\mathcal{E}_{\varepsilon})} \rightarrow \mathbbm{E} [\psi (u_0,
     \mathbbm{Y} (\lambda))] \]
  for any test function $\psi$. So the first two marginals of $\mu$ have the
  law of $(u_0, \mathbbm{Y} (\lambda))$ and they are independent since
  $(u_{\varepsilon, 0}, \mathbbm{Y}_{\varepsilon})$ are independent for any
  $\varepsilon$. Calling $\nu$ the law of $(u_0, \mathbbm{Y} (\lambda))$ we
  have that
  \[ \int_{\mathcal{Z}} \varphi (x, y, z, t) \mathd \mu (x, y, z, t) =
     \int_{\CC^{- \alpha} \times \mathcal{X}_T} \varphi (x, y, \Gamma_{\star}
     (x, y, \lambda, 0), \Gamma_{\star} (x, y, \lambda, 0)) \mathd \nu_{} (x,
     y) \]
  which implies that $\mu$ is unique and that the whole family
  $(\mu_{\varepsilon})_{\varepsilon}$ converges to $\mu$. We can conclude that
  $u^{\star}_{\varepsilon} \rightarrow u^{\star}$ in law with $u^{\star} = u
  (\lambda)$ up to the time $T_{\star} \left( \| \mathbbm{Y} (\lambda)
  \|_{\mathcal{X}_T}, \| u_0 \|_{\CC^{- 1 / 2 - \kappa}}, | \lambda | \right)$
  since the function $T^{\star}$ is lower semicontinuous (as obtained from the
  a-priori estimates).
\end{proof}

\section{\label{s:comparing}Convergence of the enhanced noise}

This is the central section of the paper, in which we present a new method to
estimate certain random fields that do not have a finite chaos decomposition,
and we apply it to the treatment of the random fields
$\mathbbm{Y}_{\varepsilon}$ of~(\ref{e:trees-def}).

\subsection{An example of convergence}\label{s:example}

We choose to give first a complete example (the convergence of the tree
$\ttwo{Y_{\varepsilon}}$ to $\ttwo{Y} (\lambda)$) in order to put in evidence
the main idea in the proof of Theorem~\ref{t:stoch-conv}. Recall its
definition~(\ref{e:trees-def}):
\[ \ttwo{Y_{\varepsilon}} = \frac{\varepsilon^{- 1}}{3}
   \tilde{F}^{(1)}_{\varepsilon} (\varepsilon^{\frac{1}{2}} Y_{\varepsilon}),
\]
with $\tilde{F}_{\varepsilon}^{(1)}$ being the first derivative of the
centered function $\tilde{F}_{\varepsilon}$ defined in~(\ref{e:ftilde-def}).
Since $\frac{\mathd}{\mathd x} H_n (x, \sigma^2_{\varepsilon}) = n H_{n - 1}
(x, \sigma^2_{\varepsilon})$ the Wiener chaos decomposition of
$\ttwo{Y_{\varepsilon}}$ reads:
\begin{eqnarray}
  \frac{\varepsilon^{- 1}}{3} \tilde{F}^{(1)}_{\varepsilon}
  (\varepsilon^{\frac{1}{2}} Y_{\varepsilon}) & = & \frac{\varepsilon^{-
  1}}{3} \sum_{n \geqslant 3} nf_{n, \varepsilon} H_{n - 1} (\varepsilon^{1 /
  2} Y_{\varepsilon}, \sigma_{\varepsilon}^2) \nonumber\\
  & = & \varepsilon^{- 1} f_{3, \varepsilon} H_2 (\varepsilon^{1 / 2}
  Y_{\varepsilon}, \sigma_{\varepsilon}^2) + \frac{\varepsilon^{- 1}}{3}
  \sum_{n \geqslant 4} nf_{n, \varepsilon} H_{n - 1} (\varepsilon^{1 / 2}
  Y_{\varepsilon}, \sigma_{\varepsilon}^2) \nonumber\\
  & = & f_{3, \varepsilon} \llbracket Y_{\varepsilon}^2 \rrbracket +
  \frac{\varepsilon^{- 1}}{3} \sum_{n \geqslant 4} nf_{n, \varepsilon} H_{n -
  1} (\varepsilon^{1 / 2} Y_{\varepsilon}, \sigma_{\varepsilon}^2), 
  \label{e:partial-decomp-F}
\end{eqnarray}
where $\llbracket \cdummy \rrbracket$ is the Wiener product and we used the
fact that $\varepsilon^{- \frac{n}{2}} H_n (\varepsilon^{1 / 2}
Y_{\varepsilon}, \sigma_{\varepsilon}^2) = \llbracket Y_{\varepsilon}^n
\rrbracket$. Now one can use hypercontractivity (as done
in~{\cite{hairer_class_2015}}, {\cite{hairer_large_2016}}) to control the
$L^p$ norm of each chaos order by its $L^2$ norms. However this strategy does
not give useful bounds for the infinite series in the second term
of~(\ref{e:partial-decomp-F}). Instead, we just observe that
\[ \sum_{n \geqslant 4} nf_{n, \varepsilon} H_{n - 1} (\varepsilon^{1 / 2}
   Y_{\varepsilon}, \sigma_{\varepsilon}^2) = (\tmop{id} - J_0 - \ldots - J_2)
   \tilde{F}^{(1)}_{\varepsilon} (\varepsilon^{\frac{1}{2}} Y_{\varepsilon}),
\]
where $J_i$ is the projection of on the $i$-th chaos, and look for a different
way to write this remainder. One of the main insights of this paper is that we
can write it as:
\[ (\tmop{id} - J_0 - \ldots - J_2) \tilde{F}^{(1)}_{\varepsilon}
   (\varepsilon^{\frac{1}{2}} Y_{\varepsilon}) = \delta^3 G_{[1]}^{[3]}
   \mathD^3  \tilde{F}^{(1)}_{\varepsilon} (\varepsilon^{\frac{1}{2}}
   Y_{\varepsilon}) \]
where $\mathD$, $\delta$ are the Malliavin derivative and divergence
operators, and $G_{[1]}^{[3]} = (1 - L)^{- 1} (2 - L)^{- 1} (3 - L)^{- 1}$
with $L$ the Ornstein-Uhlenbeck operator. This is proven in
Lemma~\ref{l:Fexpansion}.

To compute the Malliavin derivative of $\tilde{F}^{(1)}_{\varepsilon}
(\varepsilon^{\frac{1}{2}} Y_{\varepsilon})$ we observe that for every
$\varepsilon > 0$ $(t, x) \in \mathbbm{R} \times \mathbbm{T}^3$ there exists
$h_{(t, x)} \in L^2 (\mathbbm{R} \times \mathbbm{T}^3)$ such that the Gaussian
random variable $Y_{\varepsilon, (t, x)} \assign Y_{\varepsilon} (t, x)$ can
be written as
\begin{equation}
  Y_{\varepsilon, (t, x)} \overset{\text{law}}{=} \langle \xi, h_{(t, x)}
  \rangle . \label{e:Y-law-equiv}
\end{equation}
Here $\xi$ is the Gaussian white noise on $\mathbbm{R} \times \mathbbm{T}^3$,
which can be seen as a Gaussian Hilbert space $\langle \xi, h \rangle_{h \in
H} = \{ W (h) \}_{h \in H}$ indexed by the Hilbert space $H \assign L^2
(\mathbbm{R} \times \mathbbm{T}^3)$. This is the framework in which we apply
the Malliavin calculus results of Appendix~\ref{a:malliavin}. Notice that by
construction
\begin{equation}
  \langle h_{(t, x)}, h_{(t', x')} \rangle = C_{\varepsilon} (t - t', x - x')
  \assign \mathbbm{E} [Y_{\varepsilon, (t, x)} Y_{\varepsilon, (t', x')}] .
  \label{e:covariance-h}
\end{equation}
The function $h_{(t, x)}$ can actually be written as the space-time
convolution
\[ h_{(t, x)} = \check{P} \ast \psi_{\varepsilon} (t, x), \]
with $\check{P} (t, x) = \frac{1}{(4 \mathpi t)^{3 / 2}} e^{- \frac{| x |^2}{4
t}} e^{- t}  \mathbbm{1}_{t \geqslant 0}$ and $\psi_{\varepsilon}$ such that
$\eta_{\varepsilon} = \psi_{\varepsilon} \ast \xi$, $\psi_{\varepsilon} (t, x)
= \varepsilon^{- 5 / 2} \psi (\varepsilon^{- 2} t, \varepsilon^{- 1} x)$. We
omit the dependence on $\varepsilon$ of $h_{(t, x)}$ not to burden the
notation.

Going back to the calculations, we obtain from~(\ref{e:Y-law-equiv}) that
$\mathD \tilde{F}^{(1)}_{\varepsilon} (\varepsilon^{\frac{1}{2}}
Y_{\varepsilon}) = \varepsilon^{\frac{1}{2}} \tilde{F}^{(2)}_{\varepsilon}
(\varepsilon^{\frac{1}{2}} Y_{\varepsilon}) h .$ Then (noting that
$\tilde{F}^{(4)}_{\varepsilon} = F^{(4)}_{\varepsilon}$):
\begin{equation}
  \begin{array}{lll}
    \ttwo{Y_{\varepsilon}} & = & f_{3, \varepsilon} \llbracket
    Y_{\varepsilon}^2 \rrbracket + \frac{\varepsilon^{1 / 2}}{3} \delta^3
    G_{[1]}^{[3]} F^{(4)}_{\varepsilon} (\varepsilon^{\frac{1}{2}}
    Y_{\varepsilon}) h^{\otimes 3}\\
    & = & f_{3, \varepsilon} \llbracket Y_{\varepsilon}^2 \rrbracket +
    \ttwo{\hat{Y}_{\varepsilon}},
  \end{array} . \label{e:example-decomp}
\end{equation}
It can be easily seen from~(\ref{e:Y-law-equiv}) that $Y_{\varepsilon}$ has
the same law of a time-space mollification of $X$ by convolution (with $X$
defined in Section~\ref{s:limit-ident}). Then the convergence in law of $f_{3,
\varepsilon} \llbracket Y_{\varepsilon}^2 \rrbracket$ to $\ttwo{Y} (\lambda)$
can be easily established by standard techniques (see
{\cite{catellier_paracontrolled_2013}} or~{\cite{mourrat_construction_2016}}).
We are only left to show that $\ttwo{\hat{Y}_{\varepsilon}}$
in~(\ref{e:example-decomp}) converges to zero in $C_T \CC^{- 1 - \kappa}$.

It is well known (see Section~\ref{s:mainproof} for details) that in order to
control the norm of $\ttwo{\hat{Y}_{\varepsilon}} (t, \cdummy)$ for $t \in (0,
T]$ in the Besov space $\CC^{- \alpha - \kappa}$ $\forall \kappa > 0$ and in
probability, it is enough to have suitable estimates for
\[ \sup_{x \in \mathbbm{T}^3}  \left\| \Delta_q \ttwo{\hat{Y}_{\varepsilon}}
   (t, \bar{x}) \right\|_{L^p (\Omega)} = \left\| \Delta_q
   \ttwo{\hat{Y}_{\varepsilon}} (t, \bar{x}) \right\|_{L^p (\Omega)}, \]
for any $\bar{x} \in \mathbbm{T}^3$ since $\ttwo{\hat{Y}_{\varepsilon}}$ is
stationary in space. We then proceed to compute:
\begin{eqnarray*}
  \left\| \Delta_q \ttwo{\hat{Y}_{\varepsilon}} (t, \bar{x}) \right\|_{L^p
  (\Omega)} & = & \frac{\varepsilon^{1 / 2}}{3} \left\| \delta^3 G_{[1]}^{[3]}
  \int K_{q, \bar{x}} (x) F^{(4)}_{\varepsilon} (\varepsilon^{\frac{1}{2}}
  Y_{\varepsilon, (t, x)}) h_{(t, x)}^{\otimes 3} \mathd x \right\|_{L^p
  (\Omega)} 
\end{eqnarray*}
where $K_{q, \bar{x}} (x)$ is the kernel associated to the Littlewood-Paley
block $\Delta_q$. Observe that $\| F^{(4)}_{\varepsilon} (\varepsilon^{1 / 2}
Y_{\varepsilon, (t, x)}) \|^p_{L^p} = \int_{\mathbbm{R}} |
F^{(4)}_{\varepsilon} (z) |^p \gamma (\mathd z)$ where $\gamma (\mathd z)$ is
the density of a centered Gaussian with variance $\sigma^2_{\varepsilon}$.
This norm is then finite by the bound~(\ref{e:hyp-F-main}) of
Assumption~\ref{a:main}.

Another fundamental idea of this work is that we can \ ``estimate out'' the
bounded term $\| F^{(4)}_{\varepsilon} (\varepsilon^{1 / 2} Y_{\varepsilon,
(t, x)}) \|_{L^p}$, which has an infinite chaos decomposition, and obtain a
standard~$\Phi^4_3$ diagram that can be treated with well understood
techniques. Using Lemma~\ref{l:delta-norm} and Corollary~\ref{l:bounded-q},
one has
\begin{eqnarray*}
  &  & \left\| \delta^3 G_{[1]}^{[3]} \int K_{q, \bar{x}} (x)
  F^{(4)}_{\varepsilon} (\varepsilon^{\frac{1}{2}} Y_{\varepsilon, (t, x)})
  h_{(t, x)}^{\otimes 3} \mathd x \right\|_{L^p (\Omega)}\\
  & \lesssim & \sum_{k = 0}^3 \left\| \mathD^k G_{[1]}^{[3]} \int K_{q,
  \bar{x}} (x) F^{(4)}_{\varepsilon} (\varepsilon^{\frac{1}{2}}
  Y_{\varepsilon, (t, x)}) h_{(t, x)}^{\otimes 3} \mathd x \right\|_{L^p
  (\Omega, H^{\otimes 3 + k})}\\
  & \lesssim & \left\| \int K_{q, \bar{x}} (x) F^{(4)}_{\varepsilon}
  (\varepsilon^{\frac{1}{2}} Y_{\varepsilon, (t, x)}) h_{(t, x)}^{\otimes 3}
  \mathd x \right\|_{L^p (\Omega, H^{\otimes 3})} .
\end{eqnarray*}
Then note that we can decompose the norm $\| u \|_{L^p (\Omega, H^{\otimes
3})} = \| \| u \|^2_{H^{\otimes 3}}  \|^{1 / 2}_{L^{p / 2}}$ for $u \in L^p
(\Omega, H^{\otimes 3})$ and since the norm of the Hilbert space $H^{\otimes
3}$ is given by the scalar product $\| h^{\otimes 3} \|^2_{H^{\otimes 3}} =
\langle h, h \rangle_H^3$ we obtain
\begin{eqnarray*}
  \left\| \Delta_q \ttwo{\hat{Y}_{\varepsilon}} (t, \bar{x}) \right\|_{L^p} &
  \lesssim & \varepsilon^{1 / 2} \left\| \left\| \int K_{q, \bar{x}} (x)
  F^{(4)}_{\varepsilon} (\varepsilon^{\frac{1}{2}} Y_{\varepsilon, (t, x)})
  h_{(t, x)}^{\otimes 3} \mathd x \right\|_{H^{\otimes 3}}^2 \right\|_{L^{p /
  2}}^{1 / 2}
\end{eqnarray*}
\[ \lesssim \left[ \varepsilon \int | K_{q, \bar{x}} (x) K_{q, \bar{x}} (x') |
   \left\| F^{(4)}_{\varepsilon} (\varepsilon^{\frac{1}{2}} Y_{\varepsilon,
   (t, x)}) F^{(4)}_{\varepsilon} (\varepsilon^{\frac{1}{2}} Y_{\varepsilon,
   (t, x')}) \right\|_{L^{p / 2}} | \langle h_{(t, x)}, h_{(t, x')} \rangle
   |^3 \mathd x \mathd x' \right]^{\frac{1}{2}} . \]
The norm containing $\tilde{F}^{(4)}_{\varepsilon}$ can then be easily
estimated using H{\"o}lder's inequality as
\[ \left\| F^{(4)}_{\varepsilon} (\varepsilon^{\frac{1}{2}} Y_{\varepsilon,
   (t, x)}) F^{(4)}_{\varepsilon} (\varepsilon^{\frac{1}{2}} Y_{\varepsilon,
   (t, x')}) \right\|_{L^{p / 2}} \leqslant \left\| F^{(4)}_{\varepsilon}
   (\varepsilon^{\frac{1}{2}} Y_{\varepsilon, (t, x)}) \right\|_{L^p} \left\|
   F^{(4)}_{\varepsilon} (\varepsilon^{\frac{1}{2}} Y_{\varepsilon, (t, x')})
   \right\|_{L^p} \lesssim 1. \]
This yields
\[ \left\| \Delta_q \ttwo{\hat{Y}_{\varepsilon}} (t, \bar{x}) \right\|_{L^p
   (\Omega)} \lesssim \left[ \varepsilon \int | K_{q, \bar{x}} (x) K_{q,
   \bar{x}} (x') |  | \langle h_{(t, x)}, h_{(t, x')} \rangle |^3 \mathd x
   \mathd x' \right]^{\frac{1}{2}} \]
which is a standard $\Phi^4_3$~diagram that can be analysed with the
techniques of~{\cite{hairer_theory_2014}} (recalled in
Appendix~\ref{a:ker-est}). We just remark that $\langle h_{(t, x)}, h_{(t,
x')} \rangle = C_{\varepsilon} (0, x - x')$ and the bound $\varepsilon |
C_{\varepsilon} (t, x) | \lesssim 1$ of Lemma~\ref{l:cov-est-eps} yields
$\forall \delta \in (0, 1)$:
\begin{eqnarray*}
  \left\| \Delta_q \ttwo{\hat{Y}_{\varepsilon}} (t, \bar{x}) \right\|_{L^p
  (\Omega)} & \lesssim & \varepsilon^{\frac{\delta}{2}} \left[ \int | K_{q,
  \bar{x}} (x) K_{q, \bar{x}} (x') |  | C_{\varepsilon} (0, x - x') |^{2 +
  \delta} \mathd x \mathd x' \right]^{\frac{1}{2}}\\
  & \lesssim & \varepsilon^{\frac{\delta}{2}} 2^{(1 + \delta / 2) q}
\end{eqnarray*}
and then $\ttwo{\hat{Y}_{\varepsilon}} (t, \cdummy)$ converges to zero in
probability in the space $\CC^{- 1 - \delta / 2}$ $\forall \delta \in (0, 1)$,
as $\varepsilon \rightarrow 0$. The time regularity of
$\ttwo{\hat{Y}_{\varepsilon}}$ needed to obtain the convergence in $C_T \CC^{-
1 - \delta / 2}$ does not need new ideas, and it is done in
Section~\ref{s:first-trees-conv}.

The method shown in this section is valid verbatim for the trees
$\tzero{Y_{\varepsilon}}$, $\tone{Y_{\varepsilon}}$,
$\tthreeone{Y_{\varepsilon}}$, while for the composite trees
in~(\ref{e:trees-def}) (namely $\tthreetwor{Y_{\varepsilon}},
\ttwothreer{Y_{\varepsilon}}, \ttwothreer{\bar{Y}_{\varepsilon}},
\tthreethreer{Y_{\varepsilon}}$) that are obtained via paraproducts of simple
trees, one has to be able to write the remainder
$\hat{Y}_{\varepsilon}^{\tau}$ as an iterated Skorohod integral $\delta^n
(\ldots)$ in order to exploit the boundedness of this operator. Moreover,
these trees require a second renormalization (on top of the Wick ordering)
which is not easy to control for infinite chaos decompositions. We deal with
both these difficulties introducing the product
formula~(\ref{e:productformula-easy}), which allows to write products of
iterated Skorohod integrals as combinations of iterated Skorohod integrals.
The details and calculations for composite trees can be found in
Section~\ref{s:comp-trees-all}.

\subsection{Main theorem and overview of the proof}\label{s:mainproof}

\begin{theorem}
  \label{t:stoch-conv}Under Assumption~\ref{a:main} there exists $C > 0$ such
  that for any $p \in [2, \infty)$ we have $\| \mathbbm{Y}_{\varepsilon}
  \|_{\mathcal{X}_T} < C$ in $L^p (\Omega)$. Moreover,
  $\mathbbm{Y}_{\varepsilon} \rightarrow \mathbbm{Y} (\lambda) \in
  \mathcal{X}_T$ and $Y_{\varepsilon} \rightarrow X \in C_T \CC^{- 1 / 2 -
  \kappa}$ in law.
\end{theorem}

The rest of Section~\ref{s:comparing} is dedicated to the proof of
Theorem~\ref{t:stoch-conv}.

From the definition $\mathbbm{Y} (\lambda) = (Y^{\tau} (\lambda))_{\tau}$
of~(\ref{e:def-ylambda}) it is clear that we need to prove that
$Y^{\tau}_{\varepsilon} \rightarrow Y^{\tau}_{} (\lambda)$ for every
tree~$\tau$. Note that we can write each tree $Y^{\tau} (\lambda)$ as
$Y^{\tau} (\lambda) = f_{\tau} (\lambda) K^{\tau} (X)$ for a the measurable
function $K^{\tau}$ of the Gaussian process $X \in C_T \CC^{- 1 / 2 - \kappa}$
defined~(\ref{e:stoch-usual}), and a suitable deterministic function $f_{\tau}
(\lambda)$ of $\lambda$. For example, we can write $\ttwo{Y} (\lambda) =
\lambda_3 \llbracket X^2 \rrbracket$ with $\ttwo{K} (X) = \ttwo{X} =
\llbracket X^2 \rrbracket$ and
$f_{\resizebox{6pt}{4.5pt}{\includegraphics{trees/2_tree.eps}}} (\lambda) =
\lambda_3$.

We will show (eqs.~(\ref{e:firstorder-after-renorm})
and~(\ref{e:after-renorm})) that every random field $Y^{\tau}_{\varepsilon}$
defined in~(\ref{e:trees-def}) can be decomposed with the same functions
$f_{\tau} (\cdummy)$ and $K^{\tau}$ as
\begin{equation}
  Y^{\tau}_{\varepsilon} = f_{\tau} (\lambda_{\varepsilon}) K^{\tau}
  (Y_{\varepsilon}) + \hat{Y}^{\tau}_{\varepsilon}  \label{eq:decomp-chaos}
\end{equation}
where $\hat{Y}^{\tau}_{\varepsilon}$ are suitable remainder terms. For all $p
\geqslant 2$ it is well-known (see {\cite{catellier_paracontrolled_2013}},
{\cite{hairer_theory_2014}}) that the term $f_{\tau} (\lambda_{\varepsilon})
K^{\tau} (Y_{\varepsilon})$ is uniformly bounded in $L^p (\Omega ;
\mathcal{X}^{\tau})$ (with $\mathcal{X}^{\tau}$ given by~(\ref{e:y-space})).
Thus, we will prove that $\hat{Y}^{\tau}_{\varepsilon}$ converges to zero in
$L^p (\Omega ; \mathcal{X}^{\tau})$. This can be done by showing that, by
Besov embedding, for $p \in [2, \infty)$ and $\forall \alpha < | \tau |$ we
have
\begin{equation}
  \mathbbm{E} (\| \hat{Y}^{\tau}_{\varepsilon} (t) \|_{\CC^{\alpha - 3 /
  p}}^p) \lesssim \mathbbm{E} (\| \hat{Y}_{\varepsilon}^{\tau} (t)
  \|^p_{B^{\alpha}_{p, p}}) \leqslant \sum_q 2^{\alpha pq}
  \int_{\mathbbm{T}^3}  \| \Delta_q \hat{Y}_{\varepsilon}^{\tau} (t, x)
  \|^p_{L^p (\Omega)} \mathd x \rightarrow 0 \label{e:besov-bound}
\end{equation}
thanks to the stationarity of the process $Y (t, x)$. In order to prove the
bound~(\ref{e:besov-bound}) it suffices to show that $\forall t \in [0, T]$
\begin{equation}
  \sum_q 2^{\alpha pq} \sup_x  \| \Delta_q \hat{Y}_{\varepsilon}^{\tau} (t, x)
  \|^p_{L^p (\Omega)} \rightarrow 0 \quad \text{as $\varepsilon \rightarrow
  0$,}  \label{e:check-conv-1}
\end{equation}
which is one of the key estimation of this paper and will be performed in
Sections~\ref{s:first-trees-conv} and~\ref{s:second-trees-est}.

In order to obtain uniform convergence for $t \in [0, T]$ it suffices to show
that $\forall \sigma \in [0, 1 / 2]$, $q \geqslant - 1$:
\begin{equation}
  \sup_x  \| \Delta_q \hat{Y}_{\varepsilon}^{\tau} (t, x) - \Delta_q
  \hat{Y}_{\varepsilon}^{\tau} (s, x) \|^p_{L^p (\Omega)} \leqslant
  C_{\varepsilon} | t - s |^{\sigma p} 2^{- (\alpha - 2 \sigma) pq} \quad
  \text{with $C_{\varepsilon} \rightarrow 0$} . \label{e:check-conv-2}
\end{equation}
Indeed, by the Garsia-Rodemich-Rumsey inequality we obtain for $\delta > 0$
small enough and $p$ large enough
\begin{eqnarray*}
  \sup_{\varepsilon} \mathbbm{E} (\| \hat{Y}_{\varepsilon}^{\tau}
  \|^p_{C_T^{\sigma - 2 / p} B^{\alpha - 2 \sigma - \delta}_{p, p}}) &
  \leqslant & T^2 \sum_q 2^{(\alpha - 2 \sigma - \delta) pq} \sup_{s < t \in
  [0, T]} \sup_x \frac{\| \Delta_q \hat{Y}_{\varepsilon}^{\tau} (t, x) -
  \Delta_q \hat{Y}_{\varepsilon}^{\tau} (s, x) \|_{L^p (\Omega)}^p}{| t - s
  |^{\sigma p}}\\
  & \leqslant & C_{\varepsilon} T^2 \sum_q 2^{- \delta pq}
\end{eqnarray*}
which by Besov embedding yields an estimation on $\mathbbm{E} \left( \|
Y_{\varepsilon}^{\tau} \|_{C_T^{\sigma - \kappa / 2} \CC^{\alpha - 2 \sigma -
\kappa}} \right)$ for $\kappa > 0$ small enough. This gives us the necessary
tightness to claim that $\mathbbm{Y}_{\varepsilon}$ has weak limits along
subsequences.

The only thing left after proving~(\ref{eq:decomp-chaos}),
(\ref{e:check-conv-1}) and~(\ref{e:check-conv-2}) is that for each $\tau$ we
have $K^{\tau} (Y_{\varepsilon}) \rightarrow K^{\tau} (X)$ in law. However
this is clear and already well-known, since by hypothesis we can introduce a
space-time convolution regularisation of $X$ (let's call it $X_{\varepsilon}$)
which has the same law of $Y_{\varepsilon}$ for any $\varepsilon > 0$. This
yields immediately the convergence $Y_{\varepsilon} \rightarrow X$ in law. At
this point an approximation argument gives that $K^{\tau} (Y_{\varepsilon})$
has the same law of $K^{\tau} (X_{\varepsilon})$. Transposing the
regularisation to the kernels of the chaos expansion we can write $K^{\tau}
(X_{\varepsilon}) = K_{\varepsilon}^{\tau} (X)$ and now it is easy to check
that \ $K_{\varepsilon}^{\tau} (X) \rightarrow K^{\tau} (X)$ in probability
(as done systematically in~{\cite{catellier_paracontrolled_2013}},
{\cite{mourrat_construction_2016}}). We can then conclude that $K^{\tau}
(Y_{\varepsilon}) \rightarrow K^{\tau} (X)$ and therefore
$Y^{\tau}_{\varepsilon} \rightarrow Y^{\tau} (\lambda)$ in law for every
$\tau$, since from Assumption~\ref{a:main} we have immediately $f_{\tau}
(\lambda_{\varepsilon}) \rightarrow f_{\tau} (\lambda)$.

Let us give some more details on how to prove the
decomposition~(\ref{eq:decomp-chaos}) and the bounds~(\ref{e:check-conv-1})
and~(\ref{e:check-conv-2}). As seen in Section~\ref{s:example} we have
$Y_{\varepsilon, \zeta} = \langle \xi, h_{\zeta} \rangle$ in law for $\zeta =
(t, x) \in \mathbbm{R} \times \mathbbm{T}^3$ and this gives
\[ \mathD^n \tilde{F}^{(m)}_{\varepsilon} (\varepsilon^{1 / 2} Y_{\varepsilon,
   \zeta}) = \tilde{F}^{(m + n)}_{\varepsilon} (\varepsilon^{1 / 2}
   Y_{\varepsilon, \zeta}) h^{\otimes n}_{\zeta} . \]
We define for $m \in \mathbbm{N}$, $\zeta \in \mathbbm{R} \times
\mathbbm{T}^3$:
\begin{equation}
  \Phi_{\zeta}^{[m]} \assign \varepsilon^{\frac{m - 3}{2}}
  \tilde{F}_{\varepsilon}^{ (m)} (\varepsilon^{1 / 2} Y_{\varepsilon, \zeta}) 
  \label{e:def-phi}
\end{equation}
Note that the term $\Phi^{[m]}$ above is \tmtextit{not} the $m$-th derivative
of some function $\Phi$ (we use the square parenthesis notation to emphasize
this fact). It easy to see from~(\ref{e:def-phi}) that $\mathD^k
\Phi^{[m]}_{\zeta} = \Phi^{[m + k]}_{\zeta} h_{\zeta}^{\otimes k}$. Therefore,
the partial chaos expansion~(\ref{e:Fexpansion}) takes a more explicit form
when applied to $\Phi_{\zeta}^{[m]}$:
\begin{equation}
  \begin{array}{lll}
    \Phi_{\zeta}^{[m]} & = & \sum_{k = 0}^{n - 1} \frac{\mathbbm{E} (\Phi^{[m
    + k]}_{\zeta})}{k!} \llbracket Y^k_{\varepsilon, \zeta} \rrbracket +
    \delta^n (G_{[1]}^{[n]} \Phi^{[m + n]}_{\zeta} h_{\zeta}^{\otimes n})\\
    & = & \sum_{k = 0}^{n - 1} \varepsilon^{(m + k - 3) / 2} \frac{(m + k)
    !}{k!} \tilde{f}_{m + k, \varepsilon} \llbracket Y^k_{\varepsilon, \zeta}
    \rrbracket + \delta^n (G_{[1]}^{[n]} \Phi^{[m + n]}_{\zeta}
    h_{\zeta}^{\otimes n})
  \end{array} \label{e:Fexp-explicit}
\end{equation}
with $G_{[1]}^{[n]}$ defined in~(\ref{e:green-op}). Here we used the fact that
$\delta^n (h_{\zeta}^{\otimes n}) = \llbracket Y^n_{\varepsilon, \zeta}
\rrbracket$ (see Remark~\ref{r:delta-tensor}) and that by the definition of
$\Phi^{[m]}_{\zeta}$ we obtain $\forall \zeta \in \mathbbm{R} \times
\mathbbm{T}^3$:
\[ \mathbbm{E} (\Phi^{[m + k]}_{\zeta}) = \varepsilon^{(m + k - 3) / 2}  (m +
   k) ! \tilde{f}_{m + k, \varepsilon} \]
with $\tilde{f}_{n, \varepsilon}$ the coefficients in the decomposition
$\tilde{F}_{\varepsilon}  (\varepsilon^{\frac{1}{2}} Y_{\varepsilon}) \assign
\sum_{n \geqslant 0} \tilde{f}_{n, \varepsilon} H_n (\varepsilon^{\frac{1}{2}}
Y_{\varepsilon}, \sigma_{\varepsilon}^2)$. Choosing $n = 4 - m$ in
eq.~(\ref{e:Fexp-explicit}) we obtain
\begin{equation}
  \Phi_{\zeta}^{[m]} = \frac{3!}{(3 - m) !} f_{3, \varepsilon}  \llbracket
  Y^{3 - m}_{\varepsilon, \zeta} \rrbracket + \hat{\Phi}_{\zeta}^{[m]}
  \label{e:Fexp-explicit-centered}
\end{equation}
and a remainder with an infinite chaos decomposition strictly greater that $3
- m$:
\begin{equation}
  \hat{\Phi}^{[m]}_{\zeta} = \delta^{4 - m} (G^{[4 - m]}_{[1]}
  \Phi^{[4]}_{\zeta} h_{\zeta}^{\otimes 4 - m}) .
  \label{e:Fexp-explicit-remainder}
\end{equation}

This is a key step in the proof of Theorem~\ref{t:stoch-conv}. Indeed, it
suffices to substitute~(\ref{e:Fexp-explicit-centered}) into
definition~(\ref{e:trees-def}) to identify the remainder
$\hat{Y}^{\tau}_{\varepsilon}$ in decomposition~(\ref{eq:decomp-chaos}) that
has to converge to zero, and see that it always contains the term
$\hat{\Phi}^{[m]}_{\zeta}$. Moreover, the
structure~(\ref{e:Fexp-explicit-remainder}) of $\hat{\Phi}^{[m]}_{\zeta}$
makes it possible to bound its $L^p$ norm and
obtain~(\ref{e:check-conv-1}),~(\ref{e:check-conv-2}) in the same way as done
in Section~\ref{s:example} for $\ttwo{Y_{\varepsilon}}$. We will consider
separately \tmtextit{simple trees} (namely $\tzero{Y_{\varepsilon}},
\tone{Y_{\varepsilon}}, \ttwo{Y_{\varepsilon}}, \ttwo{\bar{Y}_{\varepsilon}},
\tthreeone{Y_{\varepsilon}}$) which are linear functions of $\Phi^{[m]}$ in
Section~\ref{s:first-trees-conv}, and \tmtextit{composite trees} (namely
$\tthreetwor{Y_{\varepsilon}}, \ttwothreer{Y_{\varepsilon}},
\ttwothreer{\bar{Y}_{\varepsilon}}, \tthreethreer{Y_{\varepsilon}}$) which are
quadratic in simple trees and need to be further renormalized in order to
converge to some limit as $\varepsilon \rightarrow 0$. We will show the
decomposition~(\ref{eq:decomp-chaos}) for composite trees in
Section~\ref{s:second-trees-ren} and the
bounds~(\ref{e:check-conv-1}),~(\ref{e:check-conv-2}) in
Section~\ref{s:second-trees-est}.

\begin{remark}
  \label{r:phi-greater-est}We can easily estimate terms of the form
  $\varepsilon^{- (m - 3) / 2} \Phi^{[m]}_{\zeta}$ for $3 \leqslant m
  \leqslant 9$ and every $p \in [2, \infty)$. We have (as already observed for
  $F^{(1)}_{\varepsilon}$):
  \[ \left\| \varepsilon^{- \frac{m - 3}{2}} \Phi^{[m]}_{\zeta}
     \right\|^p_{L^p} = \| F^{(m)}_{\varepsilon} (\varepsilon^{1 / 2}
     Y_{\varepsilon, \zeta}) \|^p_{L^p} = \int_{\mathbbm{R}} |
     F^{(m)}_{\varepsilon} (x) |^p \gamma (\mathd x) \]
  where $\gamma (\mathd x)$ is the density of a centered Gaussian with
  variance $\sigma^2_{\varepsilon}$. The integral is finite by
  Assumption~\ref{a:main}: in particular we only need to assume that the first
  $m$ derivatives of $F_{\varepsilon}$ have exponential growth (actually, it
  is easy to see that one can require even weaker growth conditions).
\end{remark}

\subsection{Analysis of simple trees}\label{s:first-trees-conv}

First of all note that the term $\ttwo{\bar{Y}_{\varepsilon}}$ has no
remainder, and then it can be shown to converge in law to $\lambda^{(2)}
\ttwo{Y}$ by usual techniques (see~{\cite{catellier_paracontrolled_2013}}). In
this section we show the convergence of the trees $\tzero{Y_{\varepsilon}}$,
$\tone{Y_{\varepsilon}}$, $\ttwo{Y_{\varepsilon}}$,
$\tthreeone{Y_{\varepsilon}}$. We obtain easily
from~(\ref{e:Fexp-explicit-centered}):
\begin{eqnarray}
  \Delta_q Y_{\varepsilon}^{\tau} (t, x) & \assign & \frac{(3 - m) !}{3!}
  \int_{\zeta} \Phi^{[m]}_{\zeta} \mu_{q, \zeta} \nonumber\\
  & = & f_{3, \varepsilon} \int_{\zeta} \llbracket Y^{(3 - m)}_{\varepsilon,
  \zeta} \rrbracket \mu_{\zeta} + \frac{(3 - m) !}{3!} \int_{\zeta}
  \hat{\Phi}^{[m]}_{\zeta} \mu_{q, \zeta} \nonumber\\
  & = & f_{\tau} (\lambda_{\varepsilon}) \Delta_q K^{\tau} (Y_{\varepsilon})
  (t, x) + \Delta_q \hat{Y}^{\tau}_{\varepsilon} (t, x), 
  \label{e:firstorder-after-renorm}
\end{eqnarray}
with $\zeta = (s, y)$ and either
\begin{equation}
  \begin{array}{llllll}
    \mu_{q, \zeta} & = & \delta (t - s) K_{q, x} (y) \mathd s \mathd y, &  &
    \text{for} & \tzero{\Delta_q Y_{\varepsilon}}, \Delta_q
    \tone{Y_{\varepsilon}}, \Delta_q \ttwo{Y_{\varepsilon}},\\
    \mu_{q, \zeta} & = & \mathd s \mathd y \int K_{q, x} (z) P_{t - s} (z - y)
    \mathd z, & \quad & \text{for} & \tthreeone{\Delta_q Y_{\varepsilon}},
  \end{array} \label{e:simple-measure-def}
\end{equation}
where $K_{q, x} (y)$ is the kernel associated to the Littlewood-Paley block
$\Delta_q$ and $P_t (x)$ is the heat kernel.

As said before, $f_{3, \varepsilon} \int_{\zeta} \llbracket Y^{(3 -
m)}_{\varepsilon, \zeta} \rrbracket \mu_{q, \zeta}$ converges in law in $L^p$
for every $2 \leqslant p < + \infty$ to $\lambda_3^{} \int_{\zeta} \llbracket
Y^{(3 - m)}_{\zeta} \rrbracket \mu_{q, \zeta}$ since $f_{3, \varepsilon}
\rightarrow \lambda_3$ by Assumption~\ref{a:main}. We can bound the remainder
term $\int_{\zeta} \hat{\Phi}^{[m]}_{\zeta} \mu_{q, \zeta}$ in $L^p (\Omega)$
using Lemma~\ref{l:delta-norm} and the definition of the norm $\| \cdummy
\|_{\mathbbm{D}^{4 - m, p} (H^{\otimes 4 - m})}$ to obtain:
\begin{eqnarray*}
  \left\| \int_{\zeta} \hat{\Phi}^{[m]}_{\zeta} \mu_{q, \zeta} \right\|_{L^p
  (\Omega)} & = & \left\| \delta^{4 - m} G^{[4 - m]}_{[1]} \int_{\zeta}
  \Phi^{[4]}_{\zeta} h^{\otimes 4 - m}_{\zeta} \mu_{q, \zeta} \right\|_{L^p
  (\Omega)}\\
  & \lesssim & \left\| G^{[4 - m]}_{[1]} \int_{\zeta} \Phi^{[4]}_{\zeta}
  h^{\otimes 4 - m}_{\zeta} \mu_{q, \zeta} \right\|_{\mathbbm{D}^{4 - m, p}
  (H^{\otimes 4 - m})}\\
  & \lesssim & \sum_{k = 0}^{4 - m} \left\| \mathD^k G^{[4 - m]}_{[1]}
  \int_{\zeta} \Phi^{[4]}_{\zeta} h^{\otimes 4 - m}_{\zeta} \mu_{q, \zeta}
  \right\|_{L^p (\Omega, H^{\otimes 4 - m + k})} .\\
  &  & 
\end{eqnarray*}
From Corollary~\ref{l:bounded-q} we know that $(j - L)^{- 1}$ and $\mathD (j -
L)^{- 1}$ are bounded in $L^p$ for every $p \in [2, \infty)$ and every $j
\geqslant 1$. Applying repeatedly these estimations we obtain:
\[ \left\| \mathD^k G^{[4 - m]}_{[1]} \int_{\zeta} \Phi^{[4]}_{\zeta}
   h^{\otimes 4 - m}_{\zeta} \mu_{q, \zeta} \right\|_{L^p (\Omega, H^{\otimes
   4 - m + k})} \lesssim \left\| \int_{\zeta} \Phi^{[4]}_{\zeta} h^{\otimes 4
   - m}_{\zeta} \mu_{q, \zeta} \right\|_{L^p (\Omega, H^{\otimes 4 - m})} . \]
Now we can proceed to implement the idea we already described in
Section~\ref{s:example}, i.e. estimating out the term $\left\| \varepsilon^{-
\frac{1}{2}} \Phi^{[4]}_{\zeta} \right\|_{L^p (\Omega)}$ (which is bounded by
Remark~\ref{r:phi-greater-est} but with infinite chaos decomposition) and
considering the finite-chaos term that is left. We do this by decomposing the
$L^p (\Omega, H^{\otimes 4 - m})$ norm as norms on $H^{\otimes 4 - m}$ and
$L^{p / 2} (\Omega)$ as follows: \
\begin{eqnarray*}
  \left\| \int_{\zeta} \Phi^{[4]}_{\zeta} h^{\otimes 4 - m}_{\zeta} \mu_{q,
  \zeta} \right\|_{L^p (\Omega, H^{\otimes 4 - m})} & \lesssim & \left\|
  \left\| \int_{\zeta} \Phi_{\zeta}^{[4]} h_{\zeta}^{\otimes 4 - m} \mu_{q,
  \zeta} \right\|^2_{H^{\otimes 4 - m}} \right\|_{L^{p / 2} (\Omega)}^{1 /
  2}\\
  & \lesssim & \left\| \int_{\zeta} \Phi_{\zeta}^{[4]} \Phi_{\zeta'}^{[4]}
  \langle h_{\zeta}^{\otimes 4 - m}, h_{\zeta'}^{\otimes 4 - m}
  \rangle_{H^{\otimes 4 - m}} \mu_{q, \zeta} \mu_{q, \zeta'}  \right\|_{L^{p /
  2} (\Omega)}^{1 / 2}\\
  & \lesssim & \left[ \int_{\zeta, \zeta'} \| \Phi_{\zeta}^{[4]}
  \Phi_{\zeta'}^{[4]} \|_{L^{p / 2} (\Omega)}  | \langle h_{\zeta}, h_{\zeta'}
  \rangle |^{4 - m}  | \mu_{q, \zeta} \mu_{q, \zeta'} | \right]^{1 / 2} .
\end{eqnarray*}
Finally, putting the estimations together and using H{\"o}lder's inequality,
together with the bound $\varepsilon | \langle h_{\zeta}, h_{\zeta'} \rangle |
= \varepsilon | C_{\varepsilon} (\zeta - \zeta') | \lesssim 1$ of
Lemma~\ref{l:cov-est-eps}, we obtain for every $\delta \in (0, 1]$:
\begin{eqnarray*}
  \left\| \int_{\zeta} \hat{\Phi}^{[m]}_{\zeta} \mu_{q, \zeta} \right\|_{L^p
  (\Omega)} & \lesssim & \left[ \varepsilon \int_{\zeta, \zeta'} \left\|
  \varepsilon^{- \frac{1}{2}} \Phi^{[4]}_{\zeta} \right\|_{L^p (\Omega)}
  \left\| \varepsilon^{- \frac{1}{2}} \Phi^{[4]}_{\zeta'} \right\|_{L^p
  (\Omega)} | \langle h_{\zeta}, h_{\zeta'} \rangle |^{4 - m} | \mu_{q, \zeta}
  \mu_{q, \zeta'}  | \right]^{\frac{1}{2}}\\
  & \lesssim & \left[ \varepsilon^{\delta} \int_{\zeta, \zeta'} \left\|
  \varepsilon^{- \frac{1}{2}} \Phi^{[4]}_{\zeta} \right\|_{L^p (\Omega)}
  \left\| \varepsilon^{- \frac{1}{2}} \Phi^{[4]}_{\zeta'} \right\|_{L^p
  (\Omega)} | \langle h_{\zeta}, h_{\zeta'} \rangle |^{3 - m + \delta} |
  \mu_{q, \zeta} \mu_{q, \zeta'}  | \right]^{\frac{1}{2}} .
\end{eqnarray*}
Now using Remark~\ref{r:phi-greater-est} (note that to bound $\varepsilon^{-
\frac{1}{2}} \Phi^{[4]}$ we only need to control the first 4 derivatives of
$F_{\varepsilon}$) and the fact that $\langle h_{\zeta}, h_{\zeta'} \rangle_H
= C_{\varepsilon} (\zeta - \zeta')$ we obtain as a final estimation
\begin{equation}
  \left\| \int_{\zeta} \hat{\Phi}^{(m)}_{\zeta} \mu_{q, \zeta} \right\|_{L^p
  (\Omega)} \lesssim \varepsilon^{\frac{\delta}{2}} \left[ \int |
  C_{\varepsilon} (\zeta - \zeta') |^{3 - m + \delta}  | \mu_{q, \zeta}
  \mu_{q, \zeta'}  | \right]^{1 / 2} . \label{e:firstorder-last-estim}
\end{equation}

From the definition~(\ref{e:simple-measure-def}) of the measure $\mu_{q,
\zeta}$, the l.h.s of~(\ref{e:firstorder-last-estim}) can be estimated in a
standard way using Lemma~\ref{l:integral-firstorder-est} to obtain for every
$x \in \mathbbm{T}^3$, $q > 0$:
\[ \begin{array}{lllllll}
     \left\| \Delta_q \tthreeone{\hat{Y}_{\varepsilon}} (t, x) \right\|_{L^p
     (\Omega)} & \lesssim & \varepsilon^{\frac{\delta}{2}} 2^{- \frac{1 -
     \delta}{2} q} &  & \left\| \Delta_q \tone{\hat{Y}_{\varepsilon}} (t, x)
     \right\|_{L^p (\Omega)} & \lesssim & \varepsilon^{\frac{\delta}{2}}
     2^{\frac{1 + \delta}{2} q},\\
     \left\| \Delta_q \ttwo{\hat{Y}_{\varepsilon}} (t, x) \right\|_{L^p
     (\Omega)} & \lesssim & \varepsilon^{\frac{\delta}{2}} 2^{\frac{2 +
     \delta}{2} q} &  & \left\| \Delta_q \tzero{\hat{Y}_{\varepsilon}} (t, x)
     \right\|_{L^p (\Omega)} & \lesssim & \varepsilon^{\frac{\delta}{2}}
     2^{\frac{\delta}{2} q} .
   \end{array} \]

\paragraph{Time regularity of trees}

We want to show~(\ref{e:check-conv-2}). In order to do that, we compute in the
same way as before:
\begin{eqnarray*}
  \left\| \int_{\zeta} (\hat{\Phi}^{[m]}_{t, x} - \hat{\Phi}^{[m]}_{s, x})
  \mu_{q, \zeta} \right\|_{L^p (\Omega)} & \lesssim & \left\| \delta^{4 - m}
  \int_{\zeta} G^{[4 - m]}_{[1]} (\Phi^{[4]}_{t, x} h^{\otimes 4 - m}_{t, x} -
  \Phi^{[4]}_{s, x} h_{s, x}^{\otimes 4 - m}) \mu_{q, \zeta} \right\|_{L^p
  (\Omega)}\\
  & \lesssim & \left\| \left\| \int_{\zeta} (\Phi^{[4]}_{t, x} -
  \Phi^{[4]}_{s, x}) h_{s, x}^{\otimes 4 - m} \mu_{q, \zeta}
  \right\|^2_{H^{\otimes 4 - m}} \right\|_{L^{p / 2} (\Omega)}^{1 / 2}\\
  &  & + \left\| \left\| \int_{\zeta} \Phi^{[4]}_{s, x} (h_{t, x}^{\otimes 4
  - m} - h_{s, x}^{\otimes 4 - m}) \mu_{q, \zeta} \right\|^2_{H^{\otimes 4 -
  m}} \right\|_{L^{p / 2} (\Omega)}^{1 / 2} .
\end{eqnarray*}
We focus on the first term above to obtain that it is bounded by
\begin{eqnarray*}
  &  & \left\| \int_{\zeta} (\Phi^{[4]}_{t, x} - \Phi^{[4]}_{s, x})
  (\Phi^{[4]}_{t, x'} - \Phi^{[4]}_{s, x'}) \langle h_{s, x}^{\otimes 4 - m},
  h_{s, x'}^{\otimes 4 - m} \rangle_{H^{\otimes 4 - m}} \mu_{q, \zeta} \mu_{q,
  \zeta'}  \right\|_{L^{p / 2} (\Omega)}^{1 / 2}\\
  & \lesssim & \left[ \int_{\zeta, \zeta'} \| (\Phi^{[4]}_{t, x} -
  \Phi^{[4]}_{s, x}) (\Phi^{[4]}_{t, x'} - \Phi^{[4]}_{s, x'}) \|_{L^{p / 2}
  (\Omega)}  | \langle h_{s, x}, h_{s, x'} \rangle |^{4 - m}  | \mu_{q, \zeta}
  \mu_{q, \zeta'} | \right]^{1 / 2}\\
  & \lesssim & \left[ \varepsilon \int_{\zeta, \zeta'} \| \varepsilon^{- 1}
  (\Phi^{[4]}_{t, x} - \Phi^{[4]}_{s, x}) (\Phi^{[4]}_{t, x'} - \Phi^{[4]}_{s,
  x'}) \|_{L^{p / 2} (\Omega)} | \langle h_{s, x}, h_{s, x'} \rangle |^{4 - m}
  | \mu_{q, \zeta} \mu_{q, \zeta'} | \right]^{\frac{1}{2}}\\
  & \lesssim & \left[ \varepsilon^{\delta} \int_{\zeta, \zeta'} \|
  \varepsilon^{- 1} (\Phi^{[4]}_{t, x} - \Phi^{[4]}_{s, x}) (\Phi^{[4]}_{t,
  x'} - \Phi^{[4]}_{s, x'}) \|_{L^{p / 2} (\Omega)} | \langle h_{s, x}, h_{s,
  x'} \rangle |^{3 - m + \delta} | \mu_{q, \zeta} \mu_{q, \zeta'} |
  \right]^{\frac{1}{2}} .
\end{eqnarray*}
Now note that
\begin{eqnarray*}
  \varepsilon^{- \frac{1}{2}} (\Phi^{[4]}_{t, x} - \Phi^{[4]}_{s, x}) & = &
  F^{(4)} (\varepsilon^{\frac{1}{2}} Y_{\varepsilon} (t, x)) - F^{(4)}
  (\varepsilon^{\frac{1}{2}} Y_{\varepsilon} (s, x))\\
  & = & \varepsilon^{\frac{1}{2}} \int_0^1 F^{(5)} [\varepsilon^{\frac{1}{2}}
  Y_{\varepsilon} (s, x) + \tau \varepsilon^{\frac{1}{2}} (Y_{\varepsilon} (t,
  x) - Y_{\varepsilon} (s, x))]  (Y_{\varepsilon} (t, x) - Y_{\varepsilon} (s,
  x)),
\end{eqnarray*}
and we can estimate $\left\| \varepsilon^{- \frac{1}{2}} (\Phi^{[4]}_{t, x} -
\Phi^{[4]}_{s, x}) \right\|_{L^p (\Omega)}$ by hypercontractivity and using
Lemma~\ref{l:covar-difference} as

\begin{eqnarray*}
  & \lesssim_p & \varepsilon^{1 / 2} \left\| \int_0^1 F^{(5)}
  [\varepsilon^{\frac{1}{2}} Y_{\varepsilon} (s, x) + \tau
  \varepsilon^{\frac{1}{2}} (Y_{\varepsilon} (t, x) - Y_{\varepsilon} (s, x))]
  \right\|_{L^{2 p} (\Omega)}  \| Y_{\varepsilon} (t, x) - Y_{\varepsilon} (s,
  x) \|_{L^2 (\Omega)}\\
  & \lesssim & \varepsilon^{1 / 2} [C_{\varepsilon} (0, 0) - C_{\varepsilon}
  (t - s, 0)]^{1 / 2} \\
  & \lesssim & \varepsilon^{- 2 \sigma} | t - s |^{\sigma}
\end{eqnarray*}
for any $\sigma \in [0, 1 / 2]$. The other term can be estimated more easily
by
\begin{eqnarray*}
  &  & \left[ \varepsilon^{\delta} \int_{\zeta, \zeta'} | \langle h_{s, x},
  h_{s, x'} \rangle |^{2 - m + \delta} | \langle h_{t, x} - h_{s, x}, h_{t,
  x'} - h_{s, x'} \rangle | | \mu_{q, \zeta} \mu_{q, \zeta'} |
  \right]^{\frac{1}{2}}\\
  & \lesssim & \varepsilon^{- 2 \kappa} | t - s |^{\sigma} \left[
  \varepsilon^{\delta} \int_{\zeta, \zeta'} | \langle h_{s, x}, h_{s, x'}
  \rangle |^{3 - m + \delta + 2 \sigma} | \mu_{q, \zeta} \mu_{q, \zeta'} |
  \right]^{\frac{1}{2}},
\end{eqnarray*}
and finally obtain
\[ \left\| \int_{\zeta} (\hat{\Phi}^{[m]}_{t, x} - \hat{\Phi}^{[m]}_{s, x})
   \mu_{q, \zeta} \right\|_{L^p (\Omega)} \lesssim \varepsilon^{\delta / 2 - 2
   \kappa} | t - s |^{\sigma} \left[ \int_{\zeta, \zeta'} | \langle h_{s, x},
   h_{s, x'} \rangle |^{3 - m + \delta + 2 \sigma} | \mu_{q, \zeta} \mu_{q,
   \zeta'} | \right]^{\frac{1}{2}} . \]
Which yields estimation~(\ref{e:check-conv-2}) by applying
Lemma~\ref{l:integral-firstorder-est} as before. This concludes the treatment
of simple trees. Notice that in this section we only needed $F_{\varepsilon}
\in C^5 (\mathbbm{R})$ with the first 5 derivatives having exponential growth:
indeed we need to take 4 derivatives to bound $\varepsilon^{1 / 2}
\Phi^{[4]}_{\zeta}$ as of Remark~\ref{r:phi-greater-est}, plus one more
derivative for the time regularity of $\varepsilon^{1 / 2}
\Phi^{[4]}_{\zeta}$.

\subsection{Analysis of composite trees}\label{s:comp-trees-all}

In this section we show the decomposition~(\ref{eq:decomp-chaos}) and the
bound~(\ref{e:check-conv-1}) for the trees $\tthreetwor{Y_{\varepsilon}},
\ttwothreer{Y_{\varepsilon}}, \ttwothreer{\bar{Y}_{\varepsilon}},
\tthreethreer{Y_{\varepsilon}}$. The time regularity~(\ref{e:check-conv-2}) of
$\hat{Y}^{\tau}_{\varepsilon}$ can be obtained with the same technique as in
the previous section, assuming that we can control one more derivative of
$F_{\varepsilon}$ than what is needed to prove the boundedness of
$Y^{\tau}_{\varepsilon}$ (thus we will need $F_{\varepsilon} \in C^9
(\mathbbm{R})$ with exponential growth, as discussed in
Remark~\ref{r:howmanyder}). Looking at the definitions in~(\ref{e:trees-def})
it is clear that we can write the $q$-th Littlewood-Paley blocks of
$\tthreetwor{Y_{\varepsilon}}$, $\ttwothreer{Y_{\varepsilon}}$,
$\ttwothreer{\bar{Y}_{\varepsilon}}$ and $\tthreethreer{Y_{\varepsilon}}$
$\forall \varepsilon > 0$ as:
\begin{equation}
  \begin{array}{lll}
    \Delta_q \tthreetwor{Y_{\varepsilon}} (\bar{\zeta}) & = & \frac{1}{6}
    \int_{\zeta_1, \zeta_2} \Phi^{[0]}_{\zeta_1} \Phi^{[2]}_{\zeta_2} \mu_{q,
    \zeta_1, \zeta_2} - \tthreetwor{d_{\varepsilon}} \Delta_q (1)
    (\bar{\zeta}),\\
    \Delta_q \ttwothreer{Y_{\varepsilon}} (\bar{\zeta}) & = & \frac{1}{9}
    \int_{\zeta_1, \zeta_2} \Phi^{[1]}_{\zeta_1} \Phi^{[1]}_{\zeta_2} \mu_{q,
    \zeta_1, \zeta_2} - \ttwothreer{d_{\varepsilon}} \Delta_q (1)
    (\bar{\zeta}),\\
    \Delta_q \ttwothreer{\bar{Y}_{\varepsilon}} (\bar{\zeta}) & = &
    \frac{1}{3} \int_{\zeta_1, \zeta_2} \bar{\Phi}^{[1]}_{\zeta_1}
    \Phi^{[1]}_{\zeta_2} \mu_{q, \zeta_1, \zeta_2} -
    \ttwothreer{\bar{d}_{\varepsilon}} \Delta_q (1) (\bar{\zeta}),\\
    \Delta_q \tthreethreer{Y_{\varepsilon}} (\bar{\zeta}) & = & \frac{1}{3}
    \int_{\zeta_1, \zeta_2} \Phi^{[0]}_{\zeta_1} \Phi^{[1]}_{\zeta_2} \mu_{q,
    \zeta_1, \zeta_2} - \tthreethreerprime{d_{\varepsilon}} \Delta_q
    Y_{\varepsilon} (\bar{\zeta}) - \tthreethreer{d_{\varepsilon}} \Delta_q
    (1) (\bar{\zeta}),
  \end{array} \label{e:second-trees-blocks}
\end{equation}
for any time-space point $\bar{\zeta} = (t, \bar{x})$ that we keep fixed
throughout this section. In order to keep the notation shorter we defined
\[ \bar{\Phi}^{[1]}_{\zeta_1} \assign \varepsilon^{- 1 / 2} f_{2, \varepsilon}
   \llbracket Y^2_{\varepsilon} (\zeta_1) \rrbracket, \]
which can be thought of as a finite-chaos equivalent of
$\Phi^{[1]}_{\varepsilon}$ (modulo a constant $f_{2, \varepsilon} / f_{3,
\varepsilon}$) in the same way as $\ttwo{\bar{Y}_{\varepsilon}}$ is a
finite-chaos equivalent of $\ttwo{Y_{\varepsilon}}$. The measure $\mu_{q,
\zeta_1, \zeta_2}$on $(\mathbbm{R} \times \mathbbm{T}^3)^2$ is given by
\[ \mu_{q, \zeta_1, \zeta_2} \assign [\int_{x, y} K_{q, \bar{x}} (x) \sum_{i
   \sim j} K_{i, x} (y) K_{j, x} (x_2) P_{t - s_1} (y - x_1)] \delta (t - s_2)
   \mathd \zeta_1 \mathd \zeta_2, \]
with $\zeta_{i \assign} (s_i, x_i)$ for $i = 1, 2$, $K$ being the kernel
associated to the Littlewood-Paley decomposition and $P$ being the heat
kernel. The first step for decomposing~(\ref{e:second-trees-blocks}) is to
expand them using the partial chaos expansion~(\ref{e:Fexpansion}) to obtain
\begin{equation}
  \begin{array}{lll}
    \Phi^{[0]}_{\zeta_1} \Phi^{[2]}_{\zeta_2} & = & \mathbbm{E}
    [\Phi^{[0]}_{\zeta_1} \Phi^{[2]}_{\zeta_2}] + \delta G_1 \mathD
    (\Phi^{[0]}_{\zeta_1} \Phi^{[2]}_{\zeta_2}),\\
    \Phi^{[1]}_{\zeta_1} \Phi^{[1]}_{\zeta_2} & = & \mathbbm{E}
    [\Phi^{[1]}_{\zeta_1} \Phi^{[1]}_{\zeta_2}] + \delta G_1 \mathD
    (\Phi^{[1]}_{\zeta_1} \Phi^{[1]}_{\zeta_2}),\\
    \Phi^{[0]}_{\zeta_1} \Phi_{\zeta_2}^{[1]} & = & \mathbbm{E}
    [\Phi^{[0]}_{\zeta_1} \Phi^{[1]}_{\zeta_2}] + \delta [J_0 \mathD
    (\Phi^{[0]}_{\zeta_1} \Phi^{[1]}_{\zeta_2})] + \delta^2 G_{[1]}^{[2]}
    \mathD^2 (\Phi^{[0]}_{\zeta_1} \Phi^{[1]}_{\zeta_2})\\
    & = & \mathbbm{E} [\Phi^{[0]}_{\zeta_1} \Phi^{[1]}_{\zeta_2}] +
    Y_{\varepsilon} (\zeta_1) \mathbbm{E} [\Phi^{[1]}_{\zeta_1}
    \Phi^{[1]}_{\zeta_2}] + Y_{\varepsilon} (\zeta_2) \mathbbm{E}
    [\Phi^{[0]}_{\zeta_1} \Phi^{[2]}_{\zeta_2}] + \delta^2 G_{[1]}^{[2]}
    \mathD^2 (\Phi^{[0]}_{\zeta_1} \Phi^{[1]}_{\zeta_2}) .
  \end{array} \label{e:partial-chaos}
\end{equation}
Like the trees appearing in the $\Phi^4_3$ model, we expect composite trees to
require a further renormalisation, on top of the Wick ordering. We
developed~(\ref{e:partial-chaos}) to the smallest order that allows us to see
the effect of renormalization.

\subsubsection{Renormalisation of composite trees }\label{s:second-trees-ren}

In this section we show how to renormalize~(\ref{e:second-trees-blocks}) by
estimating terms of the type $\mathbbm{E} [\Phi^{[m]}_{\zeta_1}
\Phi^{[n]}_{\zeta_2}]$ in expansion~(\ref{e:partial-chaos}). This poses an
additional difficulty, as in principle we would need to compute an infinite
number of contractions between $\Phi^{[m]}_{\zeta_1}$ and
$\Phi^{[n]}_{\zeta_2}$. However, we can again decompose $\Phi^{[m]}$ as
in~(\ref{e:Fexp-explicit-centered}), and then the product
formula~(\ref{e:productformula-easy}) ensures that we only need to control a
finite number of contractions. This is another important step in the proof and
will be carried out in Lemma~\ref{l:bound-corr}. First we need some
preparatory results:

\begin{lemma}
  We have
  \[ \int_{\zeta_1, \zeta_2} Y_{\varepsilon} (\zeta_1) \mathbbm{E}
     [\Phi^{[1]}_{\zeta_1} \Phi^{[1]}_{\zeta_2}] \mu_{q, \zeta_1, \zeta_2} =
     \int_{s, x} \Delta_q Y_{\varepsilon} (s, \bar{x} - x) G (t - s, x) . \]
  and
  \[ \int_{\zeta_1, \zeta_2} Y_{\varepsilon} (\zeta_2) \mathbbm{E}
     [\Phi^{[0]}_{\zeta_1} \Phi^{[2]}_{\zeta_2}] \mu_{q, \zeta_1, \zeta_2} =
     \int_x \Delta_q Y_{\varepsilon} (t, \bar{x} - x) H (t, x), \]
  where we introduced the kernels:
  \begin{eqnarray*}
    G (t - s, x) & \assign & \int_{x_1', x_2}  \sum_{i \sim j} K_{i, x} (x_1')
    K_{j, x} (x_2) P_{t - s} (x_1') \mathbbm{E} [\Phi^{[1]}_0 \Phi^{[1]}_{(t -
    s, x_2)}],\\
    H (t, x) & \assign & \int_{s, x_1, x_1'}  \sum_{i \sim j} K_{i, x} (x_1')
    K_{j, x} (0) P_{t - s} (x_1' - x_1) \mathbbm{E} [\Phi^{[0]}_0
    \Phi^{[2]}_{(t - s, - x_1)}] .
  \end{eqnarray*}
\end{lemma}

\begin{remark}
  Some caveat on the notation: although we use the same letter for the kernel
  $G (\cdummy, \cdummy)$ and the Green operator $G_{[m]}^{[n]}$, those two are
  not related in any possible way. It is always clear which one the notation
  refers to.
\end{remark}

\begin{proof}
  We have
  \begin{eqnarray*}
    &  & \int_{\zeta_1, \zeta_2} Y_{\varepsilon} (\zeta_1) \mathbbm{E}
    [\Phi^{[1]}_{\zeta_1} \Phi^{[1]}_{\zeta_2}] \mu_{q, \zeta_1, \zeta_2}\\
    & = & \int_{s_1, x_1, x_2, x, x_1'} K_{q, \bar{x}} (x) \sum_{i \sim j}
    K_{i, x} (x_1') K_{j, x} (x_2) P_{t - s_1} (x_1' - x_1) Y_{\varepsilon}
    (s_1, x_1) \mathbbm{E} [\Phi^{[1]}_0 \Phi^{[1]}_{(t - s_1, x_2 - x_1)}]
  \end{eqnarray*}
  and by change of variables, exploiting the translation invariance of the
  problem we obtain:
  \[ = \int_{s_1, x_1, x} K_{q, \bar{x}} (x + x_1) Y (s_1, x_1) \int_{x_1',
     x_2}  \sum_{i \sim j} K_{i, x} (x_1') K_{j, x} (x_2) P_{t - s_1} (x_1')
     \mathbbm{E} [\Phi^{[1]}_0 \Phi^{[1]}_{(t - s_1, x_2)}] . \]
  Using the definition of $K_q$ we have
  \[ = \int_{s_1, x} \Delta_q Y_{\varepsilon} (s_1, \bar{x} - x) \int_{x_1',
     x_2}  \sum_{i \sim j} K_{i, x} (x_1') K_{j, x} (x_2) P_{t - s_1} (x_1')
     \mathbbm{E} [\Phi^{[1]}_0 \Phi^{[1]}_{(t - s_1, x_2)}] . \]
  Finally we can write
  \[ \int_{\zeta_1, \zeta_2} Y_{\varepsilon} (\zeta_1) \mathbbm{E}
     [\Phi^{[1]}_{\zeta_1} \Phi^{[1]}_{\zeta_2}] \mu_{q, \zeta_1, \zeta_2} =
     \int_{s_1, x} \Delta_q Y_{\varepsilon} (s_1, \bar{x} - x) G (t - s_1, x)
     . \]
  Similar computations holds for the other term, indeed
  \begin{eqnarray*}
    &  & \int_{\zeta_1, \zeta_2} Y_{\varepsilon} (\zeta_2) \mathbbm{E}
    [\Phi^{[0]}_{\zeta_1} \Phi^{[2]}_{\zeta_2}] \mu_{q, \zeta_1, \zeta_2}\\
    & = & \int_{s_1, x_1, x_2, x, x_1'} K_{q, \bar{x}} (x) \sum_{i \sim j}
    K_{i, x} (x_1') K_{j, x} (x_2) P_{t - s_1} (x_1' - x_1) Y_{\varepsilon}
    (t, x_2) \mathbbm{E} [\Phi^{[0]}_0 \Phi^{[2]}_{(t - s_1, x_2 - x_1)}]\\
    & = & \int_{x_2} K_{q, \bar{x}} (x + x_2) Y_{\varepsilon} (t, x_2)
    \int_{s_{1,} x_1, x, x_1'}  \sum_{i \sim j} K_{i, x} (x_1') K_{j, x} (0)
    P_{t - s_1} (x_1' - x_1) \mathbbm{E} [\Phi^{[0]}_0 \Phi^{[2]}_{(t - s_1, -
    x_1)}]\\
    & = & \int_x \Delta_q Y_{\varepsilon} (t, \bar{x} - x) \int_{s_{1,} x_1,
    x_1'}  \sum_{i \sim j} K_{i, x} (x_1') K_{j, x} (0) P_{t - s_1} (x_1' -
    x_1) \mathbbm{E} [\Phi^{[0]}_0 \Phi^{[2]}_{(t - s_1, - x_1)}]\\
    & = & \int_x \Delta_q Y_{\varepsilon} (t, \bar{x} - x) H (t, x)
  \end{eqnarray*}
\end{proof}

Substituting the lemma above and~(\ref{e:partial-chaos}) in the
expressions~(\ref{e:second-trees-blocks}), we can write them as:
\begin{eqnarray*}
  \Delta_q \ttwothreer{Y_{\varepsilon}} (\bar{\zeta}) & = & \frac{1}{9}
  \int_{\zeta_1, \zeta_2} \delta G_1 \mathD (\Phi^{[1]}_{\zeta_1}
  \Phi^{[1]}_{\zeta_2}) \mu_{q, \zeta_1, \zeta_2} + \Delta_q (1) (\bar{\zeta})
  \left[ \frac{1}{9} \int_{s, x} G (t - s, x) - \ttwothreer{d_{\varepsilon}}
  \right]\\
  \Delta_q \ttwothreer{\bar{Y}_{\varepsilon}} (\bar{\zeta}) & = & \frac{1}{3}
  \int_{\zeta_1, \zeta_2} \delta G_1 \mathD (\bar{\Phi}^{[1]}_{\zeta_1}
  \Phi^{[1]}_{\zeta_2}) \mu_{q, \zeta_1, \zeta_2} + \Delta_q (1) (\bar{\zeta})
  \left[ \frac{1}{3} \int_{s, x}  \bar{G} (t - s, x) -
  \ttwothreer{\bar{d}_{\varepsilon}} \right]\\
  \Delta_q \tthreetwor{Y_{\varepsilon}} (\bar{\zeta}) & = & \frac{1}{6}
  \int_{\zeta_1, \zeta_2} \delta G_1 \mathD (\Phi^{[0]}_{\zeta_1}
  \Phi^{[2]}_{\zeta_2}) \mu_{q, \zeta_1, \zeta_2} + \Delta_q (1) (\bar{\zeta})
  \left[ \frac{1}{6} \int_x H (t, x) - \tthreetwor{d_{\varepsilon}} \right]\\
  \Delta_q \tthreethreer{Y_{\varepsilon}} (\bar{\zeta}) & = & \frac{1}{3}
  \int_{\zeta_1, \zeta_2} \delta^2 G_1^2 \mathD^2 (\Phi^{[0]}_{\zeta_1}
  \Phi^{[1]}_{\zeta_2}) \mu_{q, \zeta_1, \zeta_2} + \Delta_q (1) (\bar{\zeta})
  \left[ \frac{1}{3} \int_{\zeta_1, \zeta_2} \mathbbm{E} [\Phi^{[0]}_{\zeta_1}
  \Phi^{[1]}_{\zeta_2}] \mu_{q, \zeta_1, \zeta_2} -
  \ttwothreer{d_{\varepsilon}} \right]\\
  &  & + \Delta_q Y_{\varepsilon} (\bar{\zeta}) \left[ \frac{1}{3} \int_{s,
  x} G (t - s, x) + \frac{1}{3} \int_x H (t, x) -
  \tthreethreerprime{d_{\varepsilon}} \right]\\
  &  & + \frac{1}{3} \Delta_q \ttwothreer{R_{\varepsilon}} (\bar{\zeta}) +
  \frac{1}{3} \Delta_q \tthreetwor{R_{\varepsilon}} (\bar{\zeta})
\end{eqnarray*}
with the additional definitions
\begin{eqnarray*}
  \bar{G} (t - s, x) & \assign & \int_{x_1, x_1'}  \sum_{i \sim j} K_{i, x}
  (x_1') K_{j, x} (0) P_{t - s_1} (x_1' - x_1) \mathbbm{E} [\bar{\Phi}^{[1]}_0
  \Phi^{[1]}_{(t - s, - x_1)}],\\
  \Delta_q \ttwothreer{R_{\varepsilon}} (\bar{\zeta}) & \assign & \int_{s_{},
  x} [\Delta_q Y_{\varepsilon} (s, \bar{x} - x) - \Delta_q Y_{\varepsilon} (t,
  \bar{x})] G (t - s_{}, x),\\
  \Delta_q \tthreetwor{R_{\varepsilon}} (\bar{\zeta}) & \assign & \int_x
  [\Delta_q Y_{\varepsilon} (t, \bar{x} - x) - \Delta_q Y_{\varepsilon} (t,
  \bar{x})] H (t, x) .
\end{eqnarray*}
Now we can characterise the local behaviour of $\mathbbm{E}
[\Phi^{[m]}_{\zeta_1} \Phi^{[n]}_{\zeta_2}]$ appearing in the integrals above.
Decomposing separately $\Phi^{[m]}_{\zeta_1}$ and $\Phi^{[n]}_{\zeta_2}$ as
in~(\ref{e:Fexp-explicit-centered}) we obtain:
\begin{eqnarray*}
  \mathbbm{E} [\Phi^{[m]}_{\zeta_1} \Phi^{[n]}_{\zeta_2}] & = & \frac{3!^2}{(3
  - m) ! (3 - n) !} (f_{3, \varepsilon})^2 \mathbbm{E} [\llbracket Y^{3 -
  m}_{\varepsilon, \zeta_1} \rrbracket \llbracket Y^{3 - n}_{\varepsilon,
  \zeta_2} \rrbracket] + \frac{3!}{(3 - m) !} f_{3, \varepsilon} \mathbbm{E}
  [\llbracket Y^{3 - m}_{\varepsilon, \zeta_1} \rrbracket
  \hat{\Phi}_{\zeta_2}^{[n]}]\\
  &  & + \frac{3!}{(3 - n) !} f_{3, \varepsilon} \mathbbm{E} [\llbracket Y^{3
  - n}_{\varepsilon, \zeta_2} \rrbracket \hat{\Phi}_{\zeta_1}^{[m]}]
  +\mathbbm{E} [\hat{\Phi}_{\zeta_1}^{[m]} \hat{\Phi}_{\zeta_2}^{[n]}],
\end{eqnarray*}
where $\mathbbm{E} [\llbracket Y^{3 - m}_{\varepsilon, \zeta_1} \rrbracket
\llbracket Y^{3 - n}_{\varepsilon, \zeta_2} \rrbracket] = (3 - m) ! \delta (3
- m, 3 - n) C_{\varepsilon} (\zeta_1 - \zeta_2)^{3 - n}$ and to bound all
other terms we introduce the following result.

\begin{lemma}
  \label{l:bound-corr}Under Assumption~\ref{a:main} (in particular if $F \in
  C^8 (\mathbbm{R})$ with exponentially growing derivatives) we have, for
  every $0 \leqslant m, n \leqslant 3$ and $m \leqslant n$:
  \[ | \mathbbm{E} [\hat{\Phi}^{[m]}_{\zeta_1} \hat{\Phi}^{[n]}_{\zeta_2}] |
     \lesssim \sum_{i = 0}^{4 - n} \varepsilon^{1 + \frac{n - m}{2} + i} |
     \langle h_{\zeta_1}, h_{\zeta_2} \rangle |^{4 - m + i} \lesssim
     \varepsilon^{\delta} | \langle h_{\zeta_1}, h_{\zeta_2} \rangle |^{3 -
     \frac{m + n}{2} + \delta}, \qquad \forall \delta \in [0, 1] . \]
  Moreover for every $0 \leqslant m, n \leqslant 3$,
  \[ \begin{array}{lllll}
       | \mathbbm{E} [\llbracket Y^m_{\varepsilon, \zeta_1} \rrbracket
       \hat{\Phi}^{[n]}_{\zeta_2}] | & \lesssim & \varepsilon^{\frac{m + n -
       3}{2}} | \langle h_{\zeta_1}, h_{\zeta_2} \rangle |^m & \text{if} & m
       \geqslant 4 - n,\\
       \mathbbm{E} [\llbracket Y^m_{\varepsilon, \zeta_1} \rrbracket
       \hat{\Phi}^{[n]}_{\zeta_2}] & = & 0 & \text{if} & m < 4 - n.
     \end{array} \]
\end{lemma}

\begin{proof}
  Using formula~(\ref{e:productformula-variant}) we decompose
  \begin{eqnarray*}
    \mathbbm{E} [\hat{\Phi}^{[m]}_{\zeta_1} \hat{\Phi}^{[n]}_{\zeta_2}] & = &
    \mathbbm{E} [\delta^{4 - m} (G_{[1]}^{[4 - m]} \Phi^{[4]}_{\zeta_1}
    h_{\zeta_1}^{\otimes 4 - m}) \delta^{4 - n} (G_{[1]}^{[4 - n]}
    \Phi^{[4]}_{\zeta_2} h_{\zeta_2}^{\otimes 4 - n})]\\
    & = & \sum_{i = 0}^{4 - n} \binom{4 - m}{i} \binom{4 - n}{i}
    i!\mathbbm{E} (G_{[5 - n - i]}^{[8 - m - n - i]} \Phi_{\zeta_1}^{[8 - n -
    i]} G_{[5 - m - i]}^{[8 - m - n - i]} \Phi^{[8 - m - i]}_{\zeta_2})
    \langle h_{\zeta_1}, h_{\zeta_2} \rangle^{8 - m - n - i} .
  \end{eqnarray*}
  \[ = \sum_{i = 0}^{4 - n} \binom{4 - m}{i} \binom{4 - n}{i} i!\mathbbm{E}
     (G_{[5 - n - i]}^{[8 - m - n - i]} \Phi_{\zeta_1}^{[8 - n - i]} G_{[5 - m
     - i]}^{[8 - m - n - i]} \Phi^{[8 - m - i]}_{\zeta_2}) \langle
     h_{\zeta_1}, h_{\zeta_2} \rangle^{8 - m - n - i} \]
  We can bound the term
  \[ \varepsilon^{\frac{m + n}{2} + i - 5} \mathbbm{E} (G_{[5 - n - i]}^{[8 -
     m - n - i]} \Phi_{\zeta_1}^{[8 - n - i]} G_{[5 - m - i]}^{[8 - m - n -
     i]} \Phi^{[8 - m - i]}_{\zeta_2}) \lesssim \left\| \varepsilon^{\frac{n +
     i - 5}{2}} \Phi_{\zeta_1}^{[8 - n - i]} \right\|_{L^2} \left\|
     \varepsilon^{\frac{m + i - 5}{2}} \Phi_{\zeta_2}^{[8 - m - i]}
     \right\|_{L^2} \]
  knowing $8 - n - i \vee 8 - m - i \leqslant 8$ derivatives of
  $F_{\varepsilon}$ (see Remark~\ref{r:phi-greater-est}) and using the bound
  $\varepsilon | C_{\varepsilon} (\zeta_1 - \zeta_2) | \lesssim 1$ of
  Lemma~\ref{l:cov-est-eps} with $| \langle h_{\zeta_1}, h_{\zeta_2} \rangle |
  = | C_{\varepsilon} (\zeta_1 - \zeta_2) |$ we have:
  \[ | \mathbbm{E} [\hat{\Phi}^{[m]}_{\zeta_1} \hat{\Phi}^{[n]}_{\zeta_2}] |
     \lesssim \sum_{i = 0}^{4 - n} \varepsilon^{1 + \frac{n - m}{2} + i} |
     \langle h_{\zeta_1}, h_{\zeta_2} \rangle |^{4 - m + i} \lesssim
     \varepsilon^{\delta} | \langle h_{\zeta_1}, h_{\zeta_2} \rangle |^{3 -
     \frac{m + n}{2} + \delta} . \]
  For the second bound we recall that $\llbracket Y^m_{\varepsilon, \zeta_1}
  \rrbracket = \delta^m (h_{\zeta_1}^{\otimes m})$
  (Remark~\ref{r:delta-tensor}) and compute
  \begin{eqnarray*}
    \mathbbm{E} [\llbracket Y^m_{\varepsilon, \zeta_1} \rrbracket
    \hat{\Phi}^{[n]}_{\zeta_2 \zeta_2}] & = & \mathbbm{E} [\delta^m
    (h_{\zeta_1}^{\otimes m}) \delta^{4 - n} (G_{[1]}^{[4 - n]}
    \Phi^{[4]}_{\zeta_2} h_{\zeta_2}^{\otimes 4 - n})]\\
    & = & \sum_{i = 0}^{m \wedge 4 - n} \binom{m}{i} \binom{4 - n}{i}
    i!\mathbbm{E} (\langle \mathD^{4 - n - i} (h_{\zeta_1}^{\otimes m}), G_{[m
    + 1 - i]}^{[m + 4 - n - i]} \Phi^{[4 + m - i]}_{\zeta_2}
    h_{\zeta_2}^{\otimes m + 4 - n - i} \rangle_{H^{\otimes m + 4 - n - i}}) .
  \end{eqnarray*}
  Since $\mathD h^{\otimes m}_{\zeta_1} = 0$ we obtain $\mathbbm{E}
  [\llbracket Y^m_{\varepsilon, \zeta_1} \rrbracket
  \hat{\Phi}^{[n]}_{\zeta_2}] = 0$ if $m < 4 - n$ and
  \begin{eqnarray*}
    | \mathbbm{E} [\llbracket Y^m_{\varepsilon, \zeta_1} \rrbracket
    \hat{\Phi}^{[n]}_{\zeta_2}] | & \lesssim & \varepsilon^{\frac{m + n -
    3}{2}} \mathbbm{E} [\varepsilon^{- \frac{3 - m - n}{2}} G_{[m + n -
    3]}^{[m]} \Phi^{[m + n]}_{\zeta_2}] | \langle h_{\zeta_1}, h_{\zeta_2}
    \rangle |^m
  \end{eqnarray*}
  if $m \geqslant 4 - n$, with
  \[ \mathbbm{E} [\varepsilon^{- \frac{3 - m - n}{2}} G_{[m + n - 3]}^{[m]}
     \Phi^{[m + n]}_{\zeta_2}] \lesssim 1. \]
\end{proof}

Using Lemma~\ref{l:bound-corr} we obtain
\begin{eqnarray*}
  \mathbbm{E} [\Phi^{[1]}_{\zeta_1} \Phi^{[1]}_{\zeta_2}] & = & 9\mathbbm{E}
  [(f_{3, \varepsilon} \llbracket Y^2_{\varepsilon, \zeta_1} \rrbracket +
  \hat{\Phi}^{[1]}_{\varepsilon, \zeta_1}) (f_{3, \varepsilon} \llbracket
  Y^2_{\varepsilon, \zeta_2} \rrbracket + \hat{\Phi}^{[1]}_{\zeta_2})] = 18
  (f_{3, \varepsilon})^2  [C_{\varepsilon} (\zeta_1 - \zeta_2)]^2 +\mathbbm{E}
  [\hat{\Phi}^{[1]}_{\zeta_1} \hat{\Phi}^{[1]}_{\zeta_2}]
\end{eqnarray*}
and thus $G (t - s, x) = 18 (f_{3, \varepsilon})^2 \int_{x_1', x_2}  \sum_{i
\sim j} K_{i, x} (x_1') K_{j, x} (x_2) P_{t - s} (x_1') [C_{\varepsilon}
(\zeta_1 - \zeta_2)]^2 + \hat{G} (t - s, x)$ with the remainder term defined
by
\[ \hat{G} (t - s, x) \assign \int_{x_1', x_2}  \sum_{i \sim j} K_{i, x}
   (x_1') K_{j, x} (x_2) P_{t - s} (x_1') \mathbbm{E} [\hat{\Phi}^{[1]}_0
   \hat{\Phi}^{[1]}_{(t - s, x_2)}] . \]
We have the estimation
\begin{equation}
  | \mathbbm{E} [\hat{\Phi}^{[1]}_{\zeta_1} \hat{\Phi}^{[1]}_{\zeta_2}] |
  \lesssim \varepsilon^{\delta} C_{\varepsilon} (\zeta_1 - \zeta_2)^{2 +
  \delta} . \label{e:bound-corr-11}
\end{equation}
Similarly
\begin{eqnarray*}
  \mathbbm{E} [\bar{\Phi}^{[1]}_{\zeta_1} \Phi^{[1]}_{\zeta_2}] & = & 3
  \varepsilon^{- 1 / 2} f_{2, \varepsilon} \mathbbm{E} [\llbracket
  Y^2_{\varepsilon, \zeta_1} \rrbracket (f_{3, \varepsilon} \llbracket
  Y^2_{\varepsilon, \zeta_2} \rrbracket + \hat{\Phi}^{[1]}_{\zeta_2})] = 6
  \varepsilon^{- 1 / 2} f_{2, \varepsilon} f_{3, \varepsilon} [C_{\varepsilon}
  (\zeta_1 - \zeta_2)]^2,
\end{eqnarray*}
and
\begin{equation}
  \begin{array}{lll}
    | \mathbbm{E} [\Phi^{[0]}_{\zeta_1} \Phi^{[2]}_{\zeta_2}] | & = & |
    \mathbbm{E} [(f_{3, \varepsilon} \llbracket Y^3_{\varepsilon, \zeta_1}
    \rrbracket + \hat{\Phi}^{[0]}_{\zeta_1}) (6 f_{3, \varepsilon}
    Y_{\varepsilon, \zeta_2} + \hat{\Phi}^{[2]}_{\zeta_2})] | \lesssim | f_{3,
    \varepsilon} | | \mathbbm{E} [\llbracket Y^3_{\varepsilon, \zeta_1}
    \rrbracket \hat{\Phi}^{[2]}_{\zeta_2}] | + | \mathbbm{E}
    [\hat{\Phi}^{[0]}_{\zeta_1} \hat{\Phi}^{[2]}_{\zeta_2}] |\\
    & \lesssim & \varepsilon^{\delta} (| f_{3, \varepsilon} | + 1)
    C_{\varepsilon} (\zeta_1 - \zeta_2)^{2 + \delta},
  \end{array} \label{e:bound-corr-02}
\end{equation}
and
\begin{equation}
  \begin{array}{lll}
    | \mathbbm{E} [\Phi^{[0]}_{\zeta_1} \Phi^{[1]}_{\zeta_2}] | & = & |
    \mathbbm{E} [\Phi^{[0]} (3 f_{3, \varepsilon} \llbracket Y^2_{\varepsilon
    \comma \zeta_2} \rrbracket + \hat{\Phi}^{[1]})] | \lesssim | f_{3,
    \varepsilon} | \mathbbm{E} [\llbracket Y^3_{\varepsilon, \zeta_1}
    \rrbracket \hat{\Phi}^{[1]}_{\zeta_2}] +\mathbbm{E}
    [\hat{\Phi}^{[0]}_{\zeta_1} \hat{\Phi}^{[1]}_{\zeta_2}]\\
    & \lesssim & \varepsilon^{1 / 2} (| f_{3, \varepsilon} | + 1)
    C_{\varepsilon} (\zeta_1 - \zeta_2)^3 .
  \end{array} \label{e:bound-corr-01}
\end{equation}
We have by Lemma~\ref{l:res-product-1} that for all $\delta \in (0, 1)$
$\begin{array}{lll}
  | \hat{G} (t - s, x) | & \lesssim & \varepsilon^{\delta} (| t - s |^{1 / 2}
  + | x |)^{- 5 - \delta}
\end{array} .$ Using estimate~(\ref{e:est-KP}) together with
Lemma~\ref{l:hom-conv}, we have that for all $\delta \in (0, 1)$, $\delta' \in
(0, \delta)$ that $| H (t, x) | \lesssim \varepsilon^{\delta'} (| t - s |^{1 /
2} + | x |)^{- \delta}$. Furthermore, we have
\[ \frac{1}{3} \Delta_q \ttwothreer{R_{\varepsilon}} = 6 (f_{3,
   \varepsilon})^2  \int_{s_{}, x} [\Delta_q Y_{\varepsilon} (t + s, \bar{x} -
   x) - \Delta_q Y_{\varepsilon} (t, \bar{x})] P_s (x) [C_{\varepsilon} (s,
   x)]^2 + \frac{1}{3} \ttwothreer{\Delta_q \hat{R}_{\varepsilon}} \]
with the remainder term $\ttwothreer{\Delta_q \hat{R}_{\varepsilon}}$ given by
\[ \Delta_q \ttwothreer{\hat{R}_{\varepsilon}} = \int_{s_{}, x} [\Delta_q
   Y_{\varepsilon} (t, \bar{x} - x) - \Delta_q Y_{\varepsilon} (t, \bar{x})] 
   \hat{G} (t - s_{}, x) . \]
The term
\[ 6 (f_{3, \varepsilon})^2  \int_{s_{}, x} [\Delta_q Y_{\varepsilon} (t + s,
   \bar{x} - x) - \Delta_q Y_{\varepsilon} (t, \bar{x})] P_s (x)
   [C_{\varepsilon} (s, x)]^2 \]
can be shown to converge in law to
\[ 6 (\lambda_3)^2 \int_{s_{}, x} [\Delta_q Y (t + s, \bar{x} - x) - \Delta_q
   Y (t, \bar{x})] P_s (x) \mathbbm{E} (Y (0, 0) Y (s, x))_{}^2 \]
in $C^{\kappa}_T \CC^{- 1 / 2 - 2 \kappa}$ with the standard techniques used
in the analysis of the $\Phi^4_3$ model. On the other hand, for all $\delta >
0$ sufficiently small we have the bounds
\begin{eqnarray*}
  \left\| \Delta_q \tthreetwor{R_{\varepsilon}} \right\|_{L^{\infty}} +
  \left\| \Delta_q \ttwothreer{\hat{R}_{\varepsilon}} \right\|_{L^{\infty}} &
  \leqslant & \varepsilon^{\delta} \| Y_{\varepsilon} \|_{C^{\kappa}_T \CC^{-
  1 / 2 - 2 \kappa}} 2^{q (1 / 2 + 2 \kappa + 2 \delta)} \int_{s, x} (| x | +
  | t - s |^{1 / 2})^{\delta - 5}\\
  & \lesssim & \varepsilon^{\delta} \| Y_{\varepsilon} \|_{C^{\kappa}_T
  \CC^{- 1 / 2 - 2 \kappa}} 2^{q (1 / 2 + 2 \kappa + 2 \delta)},
\end{eqnarray*}
which shows that these remainders go to zero in $\CC^{- 1 / 2 - 2 \kappa}_{}$
as $\varepsilon \rightarrow 0$, since $\| Y_{\varepsilon} \|_{C^{\kappa}_T
\CC^{- 1 / 2 - 2 \kappa}}$ is bounded in $L^p (\Omega)$. Moreover, it is easy
to see that $\| \Delta_q \ttwothreer{R_{\varepsilon}} - \ttwothreer{\Delta_q
\hat{R}_{\varepsilon}} \|_{L^p (\Omega)} \sim O_{L^{\infty}} (2^{q (1 / 2 + 2
\kappa + 2 \delta)})$. Note that
\[ \int_{s, x} G (t - s, x) = \int_{s, x} P_s (x) \mathbbm{E} [\Phi^{[1]}_0
   \Phi^{[1]}_{(s, x)}] = 18 (f_{3, \varepsilon})^2  \int_{s, x} P_s (x)
   [C_{\varepsilon} (s, x)]^2 + \int_{s, x} P_s (x) \mathbbm{E}
   [\hat{\Phi}^{[1]}_0 \hat{\Phi}^{[1]}_{(s, x)}], \]
\[ \int_x H (t, x) = \int_{s, x} P_s (x) \mathbbm{E} [\Phi^{[0]}_0
   \Phi^{[2]}_{(s, x)}] = \int_{s, x} P_s (x) \mathbbm{E} [\Phi^{[0]}_0
   \hat{\Phi}^{[2]}_{(s, x)}] . \]
Here we used the fact that
\[ \int_x \sum_{i \sim j} K_{i, x} (x_1') K_{j, x} (0) = \int_x \sum_{i, j}
   K_{i, x} (x_1') K_{j, x} (0) = \delta (x_1'), \]
since $\int_x K_{i, x} (x_1') K_{j, x} (0) = 0$, where $| i - j | > 1$. This
is readily seen in Fourier space taking into account the support properties of
the Littlewood-Paley blocks. Now,
\[ \int_{s, x} P_s (x) \mathbbm{E} [\hat{\Phi}^{[1]}_0 \hat{\Phi}^{[1]}_{(s,
   x)}], \qquad \int_{s, x} P_s (x) \mathbbm{E} [\Phi^{[0]}_0
   \hat{\Phi}^{[2]}_{(s, x)}], \]
converge to finite constants due to the bounds~(\ref{e:bound-corr-11})
and~(\ref{e:bound-corr-02}) and by Lemma~\ref{l:PC-integrals} $\int_{s, x} P_s
(x) C_{\varepsilon} (s, x)^2 \lesssim | \log \varepsilon | .$

Finally, from~(\ref{e:bound-corr-01}) we have
\[ \int_{\zeta_1, \zeta_2} \mathbbm{E} [\Phi^{[0]}_{\zeta_1}
   \Phi^{[1]}_{\zeta_2}] \mu_{q, \zeta_1, \zeta_2} = \int_{s, x} P_s (x)
   \mathbbm{E} [\Phi^{[0]}_0 \Phi^{[1]}_{(s, x)}] = O (\varepsilon^{- 1 / 2})
   . \]
Indeed Lemma~\ref{l:PC-integrals} again yields $\varepsilon \int_{s, x} P_s
(x) C_{\varepsilon} (s, x)^3 \lesssim 1$.

Thus $\int_{\zeta_1, \zeta_2} \mathbbm{E} [\Phi^{[0]}_{\zeta_1}
\Phi^{[1]}_{\zeta_2}] \mu_{q, \zeta_1, \zeta_2}$ gives a diverging constant
which depends on all the $(f_{n, \varepsilon})_n$. Making the choice to define
the renormalisation constants $d^{\tau}$ as in eq.~(\ref{eq:d-constants}) we
cancel exactly these contributions which are either
$(F_{\varepsilon})_{\varepsilon}$ dependent and/or diverging. In particular we
verify that we can satisfy the constraint~(\ref{e:ren-constraint}).

Finally, noting that $\delta G_1 \mathD = (1 - J_0)$ and $\delta^2
G_{[1]}^{[2]} \mathD^2 = (1 - J_0 - J_1)$ (as seen in~(\ref{e:Fexpansion})),
we can write the trees of~(\ref{e:second-trees-blocks}) as
\begin{equation}
  \begin{array}{lll}
    \Delta_q \ttwothreer{Y_{\varepsilon}} (\bar{\zeta}) & = & (f_{3,
    \varepsilon})^2 \int_{\zeta_1, \zeta_2}  (1 - J_0) (\llbracket
    Y^2_{\varepsilon, \zeta_1} \rrbracket \llbracket Y_{\varepsilon,
    \zeta_2}^2 \rrbracket) \mu_{q, \zeta_1, \zeta_2} \\
    &  & + \frac{f_{3, \varepsilon}}{3} \int_{\zeta_1, \zeta_2} \delta G_1
    \mathD (\hat{\Phi}^{[1]}_{\zeta_1} \llbracket Y_{\varepsilon, \zeta_2}^2
    \rrbracket + \llbracket Y^2_{\varepsilon, \zeta_1} \rrbracket
    \hat{\Phi}^{[1]}_{\zeta_2}) \mu_{q, \zeta_1, \zeta_2} + \frac{1}{9}
    \int_{\zeta_1, \zeta_2} \delta G_1 \mathD (\hat{\Phi}^{[1]}_{\zeta_1}
    \hat{\Phi}^{[1]}_{\zeta_2}) \mu_{q, \zeta_1, \zeta_2},\\
    \Delta_q \ttwothreer{\bar{Y}_{\varepsilon}} (\bar{\zeta}) & = &
    \varepsilon^{- \frac{1}{2}} f_{2, \varepsilon} f_{3, \varepsilon}
    \int_{\zeta_1, \zeta_2}  (1 - J_0) (\llbracket Y^2_{\varepsilon, \zeta_1}
    \rrbracket \llbracket Y_{\varepsilon, \zeta_2}^2 \rrbracket) \mu_{q,
    \zeta_1, \zeta_2} + \frac{1}{3} \int_{\zeta_1, \zeta_2} \delta G_1 \mathD
    (\bar{\Phi}^{[1]}_{\zeta_1} \hat{\Phi}^{[1]}_{\zeta_2}) \mu_{q, \zeta_1,
    \zeta_2},\\
    \Delta_q \tthreetwor{Y_{\varepsilon}} (\bar{\zeta}) & = & (f_{3,
    \varepsilon})^2 \int_{\zeta_1, \zeta_2} \llbracket Y^3_{\varepsilon,
    \zeta_1} \rrbracket Y_{\zeta_2} \mu_{q, \zeta_1, \zeta_2}\\
    &  & + \frac{f_{3, \varepsilon}}{6} \int_{\zeta_1, \zeta_2} \delta G_1
    \mathD (6 \hat{\Phi}^{[0]}_{\zeta_1} Y_{\varepsilon, \zeta_2} + \llbracket
    Y^3_{\varepsilon, \zeta_1} \rrbracket \hat{\Phi}^{[2]}_{\zeta_2}) \mu_{q,
    \zeta_1, \zeta_2} + \frac{1}{6} \int_{\zeta_1, \zeta_2} \delta G_1 \mathD
    (\hat{\Phi}^{[0]}_{\zeta_1} \hat{\Phi}^{[2]}_{\zeta_2}) \mu_{q, \zeta_1,
    \zeta_2},\\
    \Delta_q \tthreethreer{Y_{\varepsilon}} (\bar{\zeta}) & = & (f_{3,
    \varepsilon})^2 \int_{\zeta_1, \zeta_2}  (1 - J_1) (\llbracket
    Y^3_{\varepsilon, \zeta_1} \rrbracket \llbracket Y_{\varepsilon,
    \zeta_2}^2 \rrbracket) \mu_{q, \zeta_1, \zeta_2} + \frac{1}{3} \Delta_q
    \tthreetwor{R_{\varepsilon}} (\bar{\zeta})\\
    &  & + 6 (f_{3, \varepsilon})^2  \int_{s_{}, x} [\Delta_q Y_{\varepsilon}
    (t + s, \bar{x} - x) - \Delta_q Y_{\varepsilon} (t, \bar{x})] P_s (x)
    [C_{\varepsilon} (s, x)]^2 + \frac{1}{3} \ttwothreer{\Delta_q
    \hat{R}_{\varepsilon}} +\\
    &  & + \frac{1}{3} \int_{\zeta_1, \zeta_2} \delta^2 G_{[1]}^{[2]}
    \mathD^2 (3 \hat{\Phi}^{[0]}_{\zeta_1} \llbracket Y_{\varepsilon,
    \zeta_2}^2 \rrbracket + \llbracket Y^3_{\varepsilon, \zeta_1} \rrbracket
    \hat{\Phi}^{[1]}_{\zeta_2}) \mu_{q, \zeta_1, \zeta_2} + \frac{1}{3}
    \int_{\zeta_1, \zeta_2} \delta^2 G_{[1]}^{[2]} \mathD^2
    (\hat{\Phi}^{[0]}_{\zeta_1} \hat{\Phi}^{[1]}_{\zeta_2}) \mu_{q, \zeta_1,
    \zeta_2} .
  \end{array} \label{e:after-renorm}
\end{equation}
Let us summarize our results so far. We have shown that $\ttwothreer{\Delta_q
\hat{R}_{\varepsilon}} (\bar{\zeta}) + \Delta_q \tthreetwor{R_{\varepsilon}}
(\bar{\zeta}) \sim O_{L^{\infty}} (\varepsilon^{\delta} 2^{q (1 / 2 + 2 \kappa
+ 2 \delta)})$ in $L^p (\Omega)$ and then these terms converge to $0$ in the
right topology as $\varepsilon \rightarrow 0$. As already mentioned, the
convergence in law of
\begin{eqnarray*}
  (f_{3, \varepsilon})^2 \int_{\zeta_1, \zeta_2}  (1 - J_0) (\llbracket
  Y^2_{\varepsilon, \zeta_1} \rrbracket \llbracket Y_{\varepsilon, \zeta_2}^2
  \rrbracket) \mu_{q, \zeta_1, \zeta_2} & \rightarrow & (\lambda_3)^2 \Delta_q
  \ttwothreer{Y} (\bar{\zeta})\\
  \varepsilon^{- \frac{1}{2}} f_{2, \varepsilon} f_{3, \varepsilon}
  \int_{\zeta_1, \zeta_2}  (1 - J_0) (\llbracket Y^2_{\varepsilon, \zeta_1}
  \rrbracket \llbracket Y_{\varepsilon, \zeta_2}^2 \rrbracket) \mu_{q,
  \zeta_1, \zeta_2} & \rightarrow & \lambda_3 \lambda_2 \Delta_q
  \ttwothreer{Y} (\bar{\zeta})\\
  (f_{3, \varepsilon})^2 \int_{\zeta_1, \zeta_2} \llbracket Y^3_{\varepsilon,
  \zeta_1} \rrbracket Y_{\zeta_2} \mu_{q, \zeta_1, \zeta_2} & \rightarrow &
  (\lambda_3)^2 \Delta_q \tthreetwor{Y} (\bar{\zeta})\\
  (f_{3, \varepsilon})^2 \int_{\zeta_1, \zeta_2}  (1 - J_1) (\llbracket
  Y^3_{\varepsilon, \zeta_1} \rrbracket \llbracket Y_{\varepsilon, \zeta_2}^2
  \rrbracket) \mu_{q, \zeta_1, \zeta_2} &  & \\
  + 6 (f_{3, \varepsilon})^2  \int_{s_{}, x} [\Delta_q Y_{\varepsilon} (t + s,
  \bar{x} - x) - \Delta_q Y_{\varepsilon} (t, \bar{x})] P_s (x)
  [C_{\varepsilon} (s, x)]^2 & \rightarrow & (\lambda_3)^2 \Delta_q
  \tthreethreer{Y} (\bar{\zeta})
\end{eqnarray*}
is easy to establish with standard techniques (as done
in~{\cite{catellier_paracontrolled_2013}}, {\cite{mourrat_construction_2016}})
assuming the convergence of $\lambda_{\varepsilon}$ as in~(\ref{e:def-lambda})
to~$\lambda$. Then, comparing~(\ref{e:after-renorm}) with the canonical trees
in~(\ref{e:stoch-usual}) we can identify the remainder terms $\Delta_q
\hat{Y}_{\varepsilon}^{\tau}$ that still need to be bounded, that are
precisely those in which $\hat{\Phi}^{[n]}_{\zeta}$ appears. Estimating these
terms is the content of next section.

\subsubsection{Estimation of renormalised composite trees
}\label{s:second-trees-est}

In this section we prove the bound~(\ref{e:check-conv-1}) for composite trees.
The difficulty we encounter here is that the remainder
$\hat{Y}^{\tau}_{\varepsilon}$ cannot be written as an iterated Skorohod
integral as in Section~\ref{s:example}, but instead as a product of iterated
Skorohod integrals. We will then use the product
formula~(\ref{e:productformula-easy}) to write the remainder in the desired
form. We can write~(\ref{e:after-renorm}) in a much shorter way as:
\[ \begin{array}{lll}
     \Delta_q \ttwothreer{Y_{\varepsilon}} (\bar{\zeta}) & = & \frac{1}{9}
     \delta G_1 \int_{\zeta_1, \zeta_2} \mathD (\Phi^{[1]}_{\zeta_1}
     \Phi^{[1]}_{\zeta_2}) \mu_{q, \zeta_1, \zeta_2},\\
     \Delta_q \ttwothreer{\bar{Y}_{\varepsilon}} (\bar{\zeta}) & = &
     \frac{1}{3} \delta G_1 \int_{\zeta_1, \zeta_2} \mathD
     (\bar{\Phi}^{[1]}_{\zeta_1} \Phi^{[1]}_{\zeta_2}) \mu_{q, \zeta_1,
     \zeta_2},\\
     \Delta_q \tthreetwor{Y_{\varepsilon}} (\bar{\zeta}) & = & \frac{1}{6}
     \delta G_1 \int_{\zeta_1, \zeta_2} \mathD (\Phi^{[0]}_{\zeta_1}
     \Phi^{[2]}_{\zeta_2}) \mu_{q, \zeta_1, \zeta_2},\\
     \Delta_q \tthreethreer{Y_{\varepsilon}} (\bar{\zeta}) & = & \frac{1}{3}
     \delta^2 G_{[1]}^{[2]} \int_{\zeta_1, \zeta_2} \mathD^2
     (\Phi^{[0]}_{\zeta_1} \Phi^{[1]}_{\zeta_2}) \mu_{q, \zeta_1, \zeta_2} +
     \frac{1}{3} \Delta_q \tthreetwor{R_{\varepsilon}} (\bar{\zeta}) +
     \frac{1}{3} \ttwothreer{\Delta_q \hat{R}_{\varepsilon}}\\
     &  & + 6 (f_{3, \varepsilon})^2  \int_{s_{}, x} [\Delta_q
     Y_{\varepsilon} (t + s, \bar{x} - x) - \Delta_q Y_{\varepsilon} (t,
     \bar{x})] P_s (x) [C_{\varepsilon} (s, x)]^2,
   \end{array} \]
just substituting again $\delta G_1 \mathD = (1 - J_0)$ and $\delta^2
G_{[1]}^{[2]} \mathD^2 = (1 - J_0 - J_1)$. In order to treat all trees at the
same time, we can write the first terms in the r.h.s. above (modulo a constant
that we discard) as:
\[ \delta^r G^{[r]}_{[1]} \int_{\zeta_1, \zeta_2} \mathD^r 
   (\Phi^{[i]}_{\zeta_1} \Phi^{[j]}_{\zeta_2}) \mu_{q, \zeta_1, \zeta_2} \quad
   \text{with $r = 1, i + j = 2$ or $r = 2, i + j = 1$.} \]
First notice that by Lemma~\ref{l:delta-norm} we have
\[ \left\| \delta^r G_{[1]}^{[r]} \int_{\zeta_1, \zeta_2} \mathD^r
   (\Phi^{[i]}_{\zeta_1} \Phi^{[j]}_{\zeta_2}) \mu_{q, \zeta_1, \zeta_2}
   \right\|_{L^p (\Omega)} \lesssim \sum_{\theta = 0}^r \left\|
   \mathD^{\theta} G_{[1]}^{[r]} \int_{\zeta_1, \zeta_2} \mathD^r
   (\Phi^{[i]}_{\zeta_1} \Phi^{[j]}_{\zeta_2}) \mu_{q, \zeta_1, \zeta_2}
   \right\|_{L^p (H^{\otimes r + \theta})} \]
and from the boundedness of the operator $\mathD^{\theta} G_{[1]}^{[r]}$ given
by Corollary~\ref{l:bounded-q} we obtain:
\[ \left\| \delta^r G_{[1]}^{[r]} \int_{\zeta_1, \zeta_2} \mathD^r
   (\Phi^{[i]}_{\zeta_1} \Phi^{[j]}_{\zeta_2}) \mu_{q, \zeta_1, \zeta_2}
   \right\|_{L^p (\Omega)} \lesssim \left\| \int_{\zeta_1, \zeta_2} \mathD^r
   (\Phi^{[i]}_{\zeta_1} \Phi^{[j]}_{\zeta_2}) \mu_{q, \zeta_1, \zeta_2}
   \right\|_{L^p (\Omega, H^{\otimes r})} . \]
Computing the $r$-th derivative of the integrand we obtain
\begin{equation}
  \Phi^{[4 - m]}_{\zeta_1} \Phi^{[4 - n]}_{\zeta_2} h^{\otimes k}_{\zeta_1}
  \otimes h^{\otimes \ell}_{\zeta_2} = \left[ \frac{3! f_{3, \varepsilon}}{(3
  - m) !}  \llbracket Y_{\zeta_1}^{m - 1} \rrbracket +
  \widehat{\Phi_{}}_{\zeta_1}^{[4 - m]} \right] \left[ \frac{3! f_{3,
  \varepsilon}}{(3 - n) !}  \llbracket Y_{\zeta_2}^{n - 1} \rrbracket +
  \hat{\Phi}_{\zeta_2}^{[4 - n]} \right] h^{\otimes k}_{\zeta_1} \otimes
  h_{\zeta_2}^{\otimes \ell}  \label{e:other-term-est}
\end{equation}
for $m + n = 5$ and $0 \leqslant k + \ell \leqslant 2$. The constraints on $m,
n, k, \ell$ are related to the number of branches in the graphical notation of
the trees: each tree has $m + k - 1$ leaves with height 2 and $n + \ell - 1$
leaves with height 1, as follows
\begin{equation}
  \begin{array}{lll}
    \ttwothreer{Y_{\varepsilon}} & \leftrightarrow & m + k = 3, \quad n + \ell
    = 3\\
    \ttwothreer{\bar{Y}_{\varepsilon}} & \leftrightarrow & m = 3, k = 0, n =
    2, \ell = 1\\
    \tthreetwor{Y_{\varepsilon}} & \leftrightarrow & m + k = 4, \quad n + \ell
    = 2\\
    \tthreethreer{Y_{\varepsilon}} & \leftrightarrow & m + k = 4, \quad n +
    \ell = 3.
  \end{array} \label{e:legs-constraints}
\end{equation}

In~(\ref{e:other-term-est}), the terms which do not contain
$\hat{\Phi}^{[n]}_{\zeta}$ will generate finite contributions in the limit, as
seen in Section~\ref{s:second-trees-ren} by writing the
decomposition~(\ref{e:after-renorm}). We just consider the terms proportional
to $\hat{\Phi}^{[4 - m]}_{\zeta_1} \hat{\Phi}^{[4 - n]}_{\zeta_2}$, because
all the other similar terms featuring at least one remainder
$\widehat{\Phi_{}}_{\zeta}^{[m]}$ can be estimated with exactly the same
technique, and are easily shown to be vanishing in the appropriate topology.

We can use one of the key observations of this paper, the product
formula~(\ref{e:productformula-easy}), to rewrite products of Skorohod
integrals in the form $\delta^m (u) \delta^n (v)$ as a sum of iterated
Skorohod integrals $\delta^{\ell} (w)^{}$, which are bounded in $L^p$ by
Lemma~\ref{l:delta-norm}. We obtain
\begin{eqnarray*}
  \hat{\Phi}^{[4 - m]}_{\zeta_1} \hat{\Phi}^{[4 - n]}_{\zeta_2} & = & \delta^m
  (G_{[1]}^{[m]} \Phi^{[4]}_{\zeta_1} h^{\otimes m}_{\zeta_1}) \delta^n
  (G_{[1]}^{[n]} \Phi^{[4]}_{\zeta_2} h_{\zeta_2}^{\otimes n})\\
  & = & \sum_{(q, r, i) \in I} C_{_{q, r, i}} \delta^{m + n - q - r} (\langle
  \mathD^{r - i} G_{[1]}^{[m]} \Phi^{[4]}_{\zeta_1} h_{\zeta_1}^{\otimes m},
  \mathD^{q - i} G_{[1]}^{[n]} \Phi^{[4]}_{\zeta_2} h_{\zeta_2}^{\otimes n}
  \rangle_{H^{\otimes q + r - i}})^{}\\
  & = & \sum_{(q, r, i) \in I} C_{_{q, r, i}} \varepsilon^{1 + \frac{r +
  q}{2} - i} \delta^{m + n - q - r} (\langle \Theta^{[m + r - i]}_{[1 + r -
  i]} (\zeta_1) h_{\zeta_1}^{\otimes m + r - i}, \Theta_{[1 + q - i]}^{[n + q
  - i]} (\zeta_2) h^{\otimes n + q - i}_{\zeta_2} \rangle_{H^{\otimes q + r -
  i}})
\end{eqnarray*}
with $I = \{ (q, r, i) \in \mathbbm{N}^3 : 0 \leqslant q \leqslant m, 0
\leqslant r \leqslant n, 0 \leqslant i \leqslant q \wedge r \}$ and the
notation shortcut:
\[ \Theta_{[i]}^{[j]} (\zeta) \assign \varepsilon^{- \frac{i}{2}}
   G^{[j]}_{[i]} \Phi_{\zeta}^{[3 + i]} . \]
By Remark~\ref{r:delta-deterministic}, for every $n, m \geqslant 1$ and $\Psi
\in \tmop{Dom} \delta^n$ we can write $\delta^n (\Psi) h^{\otimes m} =
\delta^n (\Psi \otimes h^{\otimes m})$, and therefore
\[ \int \hat{\Phi}^{[4 - m]}_{\zeta_1} \hat{\Phi}_{\zeta_2}^{[4 - n]}
   h^{\otimes k}_{\zeta_1} \otimes h^{\otimes \ell}_{\zeta_2} \mu_{q, \zeta_1,
   \zeta_2} = \]
\[ \sum_I C_{_{q, r, i}} \varepsilon^{\frac{2 + q + r - 2 i}{2}} \delta^{m + n
   - q - r} \int \Theta^{[m + r - i]}_{[1 + r - i]} (\zeta_1) \Theta_{[1 + q -
   i]}^{[n + q - i]} (\zeta_2) h_{\zeta_1}^{\otimes m - q} \otimes h^{\otimes
   n - r}_{\zeta_2} \otimes h^{\otimes k}_{\zeta_1} \otimes h^{\otimes
   \ell}_{\zeta_2} | \langle h_{\zeta_1}, h_{\zeta_2} \rangle |^{q + r - i}
   \mu_{q, \zeta_1, \zeta_2} \]
With the term to estimate in this form, we can proceed as in
Section~\ref{s:example} to estimate separately the terms $\Theta_{[i]}^{[j]}
(\zeta) = \varepsilon^{- \frac{i}{2}} G^{[j]}_{[i]} \Phi_{\zeta}^{[3 + i]}$ in
$L^p (\Omega)$, which are bounded as discussed in
Remark~\ref{r:phi-greater-est}.

\begin{lemma}
  \label{l:est-binary}Under Assumption~\ref{a:main} (in particular if
  $F_{\varepsilon} \in C^8 (\mathbbm{R})$ and the first 8~derivatives have
  exponential growth) we have the bound:
  \[ \left\| \delta^{m + n - q - r} \int \Theta^{[m + r - i]}_{[1 + r - i]}
     (\zeta_1) \Theta_{[1 + q - i]}^{[n + q - i]} (\zeta_2)
     h_{\zeta_1}^{\otimes m - q} \otimes h^{\otimes n - r}_{\zeta_2} \otimes
     h^{\otimes k}_{\zeta_1} \otimes h^{\otimes \ell}_{\zeta_2} | \langle
     h_{\zeta_1}, h_{\zeta_2} \rangle |^{q + r - i} \mu_{q, \zeta_1, \zeta_2}
     \right\|^2_{L^p (H^{\otimes k + \ell})} \]
  \[ \lesssim \int | \langle h_{\zeta_1}, h_{\zeta_1'} \rangle |^{m + k - q} |
     \langle h_{\zeta_2}, h_{\zeta_2'} \rangle |^{n + \ell - r} | \langle
     h_{\zeta_1}, h_{\zeta_2} \rangle |^{q + r - i} | \langle h_{\zeta'_1},
     h_{\zeta'_2} \rangle |^{q + r - i} | \mu_{q, \zeta_1, \zeta_2} | |
     \mu_{q, \zeta_1', \zeta_2'} | . \]
\end{lemma}

\begin{proof}
  Thanks to Lemma~\ref{l:delta-norm} the integral can be estimated with
  \[ \sum_{j = 0, h \leqslant j}^{m + n - q - r} \left\| \int \mathD^h
     \Theta^{[m + r - i]}_{[1 + r - i]} (\zeta_1) \mathD^{j - h} \Theta_{[1 +
     q - i]}^{[n + q - i]} (\zeta_2) h_{\zeta_1}^{\otimes m - q} \otimes
     h^{\otimes n - r}_{\zeta_2} \otimes h^{\otimes k}_{\zeta_1} \otimes
     h_{\zeta_2}^{\otimes \ell} | \langle h_{\zeta_1}, h_{\zeta_2} \rangle
     |^{q + r - i} \mu_{q, \zeta_1, \zeta_2} \right\|^2_{L^p (V)}, \]
  with $V = H^{\otimes m + k + n + \ell - q - r + j}$. We have that $\|
  \cdummy \|_{L^p (H^{\otimes k + \ell})}^2 = \| \| \cdummy \|^2_{H^{\otimes k
  + \ell}} \|_{L^{p / 2}}^{1 / 2}$ and therefore we can bound each term in the
  sum above as
  \begin{eqnarray*}
    & \lesssim & (\int_{} \| \langle \mathD^h \Theta^{[m + r - i]}_{[1 + r -
    i]} (\zeta_1) \mathD^{j - h} \Theta_{[1 + q - i]}^{[n + q - i]} (\zeta_2),
    \mathD^h \Theta^{[m + r - i]}_{[1 + r - i]} (\zeta'_1) \mathD^{j - h}
    \Theta_{[1 + q - i]}^{[n + q - i]} (\zeta'_2) \rangle_{H^{\otimes j}}
    \|_{L^{p / 2}} \times\\
    &  & \times \nobracket | \langle h_{\zeta_1}, h_{\zeta_1'} \rangle |^{m +
    k - q} | \langle h_{\zeta_2}, h_{\zeta_2'} \rangle |^{n + \ell - r} |
    \langle h_{\zeta_1}, h_{\zeta_2} \rangle |^{q + r - i} | \langle
    h_{\zeta'_1}, h_{\zeta'_2} \rangle |^{q + r - i} | \mu_{q, \zeta_1,
    \zeta_2} | | \mu_{q, \zeta_1', \zeta_2'} |)^{1 / 2}
  \end{eqnarray*}
  Using H{\"o}lder's inequality we get the estimate
  \begin{eqnarray*}
    &  & \| \langle \mathD^h \Theta^{[m + r - i]}_{[1 + r - i]} (\zeta_1)
    \mathD^{j - h} \Theta_{[1 + q - i]}^{[n + q - i]} (\zeta_2), \mathD^h
    \Theta^{[m + r - i]}_{[1 + r - i]} (\zeta'_1) \mathD^{j - h} \Theta_{[1 +
    q - i]}^{[n + q - i]} (\zeta'_2) \rangle_{H^{\otimes j}} \|_{L^{p / 2}}\\
    & \lesssim & \| \langle \mathD^h \Theta^{[m + r - i]}_{[1 + r - i]}
    (\zeta_1), \mathD^h \Theta^{[m + r - i]}_{[1 + r - i]} (\zeta'_1)
    \rangle_{H^{\otimes h}} \|_{L^p} \| \langle \mathD^{j - h} \Theta_{[1 + q
    - i]}^{[n + q - i]} (\zeta_2), \mathD^{j - h} \Theta_{[1 + q - i]}^{[n + q
    - i]} (\zeta'_2) \rangle_{H^{\otimes j - h}} \|_{L^p}
  \end{eqnarray*}

  Now to bound $\| \langle \mathD^h \Theta_{[1 + a]}^{[m + a]} (\zeta),
  \mathD^h \Theta_{[1 + a]}^{[m + a]} (\zeta') \rangle_{H^{\otimes h}}
  \|_{L^p}$ (with $h \leqslant j \leqslant m + n - q - r$ and $a = r - i$) we
  can use the boundedness of the operator $\mathD^h G_{[1 + a]}^{[m + a]}$ for
  $h \leqslant 2 m$ given by Corollary~\ref{l:bounded-q}. Consider the two
  regions $h \leqslant 2 m$ and $h > 2 m$. In the first region we just use
  Corollary~\ref{l:bounded-q} to obtain:
  \begin{eqnarray*}
    \| \langle \mathD^h \Theta_{[1 + a]}^{[m + a]} (\zeta), \mathD^h
    \Theta_{[1 + a]}^{[m + a]} (\zeta') \rangle_{H^{\otimes h}} \|_{L^p} &
    \lesssim & \left\| \mathD^h G_{[1 + a]}^{[m + a]} \varepsilon^{- \frac{1 +
    a}{2}} \Phi^{[4 + a]}_{\zeta_1} \right\|^2_{L^{4 p} (H^{\otimes h})}
    \left\| \mathD^h G_{[1 + a]}^{[m + a]} \varepsilon^{- \frac{1 + a}{2}}
    \Phi^{[4 + a]}_{\zeta'_1} \right\|^2_{L^{4 p} (H^{\otimes h})}\\
    & \lesssim & \left\| \varepsilon^{- \frac{1 + a}{2}} \Phi^{[4 +
    a]}_{\zeta_1} \right\|_{L^{4 p}}^2 \left\| \varepsilon^{- \frac{1 + a}{2}}
    \Phi^{[4 + a]}_{\zeta'_1} \right\|_{L^{4 p}}^2 .
  \end{eqnarray*}
  If $h > 2 m$ we first use the bound $\mathD^{2 m} G_{[1 + a]}^{[m + a]}$ and
  then take the remaining $h - 2 m$ derivatives on $\varepsilon^{- \frac{1 +
  a}{2}} \Phi^{[4 + a]}_{\zeta}$:
  \begin{eqnarray*}
    \left\| \mathD^h G_{[1 + a]}^{[m + a]} \varepsilon^{- \frac{1 + a}{2}}
    \Phi_{\zeta_1}^{[4 + a]} \right\|_{L^{4 p} (H^{\otimes h})}^2 & \lesssim &
    \left\| \mathD^{h - 2 m} \varepsilon^{- \frac{1 + a}{2}}
    \Phi_{\zeta_1}^{[4 + a]} \right\|_{L^{4 p} (H^{\otimes h - 2 m})}^2\\
    & \lesssim & \left\| \varepsilon^{- \frac{1 + a + h - 2 m}{2}}
    \Phi_{\zeta_1}^{[4 + a + h - 2 m]} \right\|_{L^{4 p} (H^{\otimes h - 2
    m})}
  \end{eqnarray*}
  From Remark~\ref{r:phi-greater-est} we see that this last term is bounded by
  a constant if $F \in C^{4 + a + h - 2 m} (\mathbbm{R})$ with the first $4 +
  a + h - 2 m$ derivatives having an exponential growth, with $a = r - i$ and
  $h \leqslant m + n - q - r$.
  
  Applying the same reasoning to $\| \langle \mathD^{j - h} \Theta_{[1 + q -
  i]}^{[n + q - i]} (\zeta_2), \mathD^{j - h} \Theta_{[1 + q - i]}^{[n + q -
  i]} (\zeta'_2) \rangle_{H^{\otimes j - h}} \|_{L^p}$ we conclude that we
  need to control $4 + n \vee 4 + m$ derivatives of $F_{\varepsilon}$ in order
  to perform the estimates of this Lemma. From the
  constraints~(\ref{e:legs-constraints}) we see that $4 + n \vee 4 + m
  \leqslant 8$.
\end{proof}

From Lemma~\ref{l:est-binary} we obtain $\forall \delta \in [0, 1 / 2)$:
$\varepsilon^{\frac{2 + q + r - 2 i}{2}}$
\[ \left\| \varepsilon^{\frac{2 + q + r - 2 i}{2}} \delta^{m + n - q - r} \int
   \Theta^{[m + r - i]}_{[1 + r - i]} (\zeta_1) \Theta_{[1 + q - i]}^{[n + q -
   i]} (\zeta_2) h_{\zeta_1}^{\otimes m - q} \otimes h^{\otimes n -
   r}_{\zeta_2} \otimes h^{\otimes k}_{\zeta_1} \otimes h^{\otimes
   \ell}_{\zeta_2} | \langle h_{\zeta_1}, h_{\zeta_2} \rangle |^{q + r - i}
   \mu_{_{q, \zeta_1, \zeta_2}} \right\|_{L^p (H^{\otimes k + \ell})} \]
\[ \lesssim \varepsilon^{\delta} \left( \varepsilon^{2 + q + r - 2 i - \delta}
   \int | \langle h_{\zeta_1}, h_{\zeta_1'} \rangle |^{m + k - q} | \langle
   h_{\zeta_2}, h_{\zeta_2'} \rangle |^{n + \ell - r} | \langle h_{\zeta_1},
   h_{\zeta_2} \rangle |^{q + r - i} | \langle h_{\zeta'_1}, h_{\zeta'_2}
   \rangle |^{q + r - i} \left| \mu_{_{q, \zeta_1, \zeta_2}} \right| \left|
   \mu_{_{q, \zeta_1', \zeta_2'}} \right| \right)^{\frac{1}{2}} \]
\[ \assign \varepsilon^{\frac{\delta}{2}} (\mathfrak{I})^{\frac{1}{2}} . \]
Our aim now is to estimate the quantity $\mathfrak{I}$. The idea is to use the
bound $\varepsilon | \langle h_{\zeta}, h_{\zeta'} \rangle | = \varepsilon
C_{\varepsilon} (\zeta - \zeta') \lesssim 1$ of Lemma~\ref{l:cov-est-eps} to
cancel strategically some of the covariances $| \langle h_{\zeta}, h_{\zeta'}
\rangle |$. We will consider three regions:

If $q + r \leqslant 2$ we use the bounds
\[ \varepsilon^{q + r - 2 i} | \langle h_{\zeta_1}, h_{\zeta_2} \rangle |^{q +
   r - i} | \langle h_{\zeta'_1}, h_{\zeta'_2} \rangle |^{q + r - i} \lesssim
   \varepsilon^2 | \langle h_{\zeta_1}, h_{\zeta_2} \rangle |^q | \langle
   h_{\zeta'_1}, h_{\zeta'_2} \rangle |^r \]
and then (we suppose $r < 2$)
\[ \varepsilon^{2 - r - \delta} | \langle h_{\zeta_2}, h_{\zeta_2'} \rangle
   |^{n + \ell - r} \lesssim | \langle h_{\zeta_2}, h_{\zeta_2'} \rangle |^{n
   + \ell - 2 + \delta} \]
to obtain
\begin{eqnarray}
  \mathfrak{I} & \lesssim & \varepsilon^{r - \delta} \int | \langle
  h_{\zeta_1}, h_{\zeta_1'} \rangle |^{m + k - q} | \langle h_{\zeta_2},
  h_{\zeta_2'} \rangle |^{n + \ell - 2} | \langle h_{\zeta_1}, h_{\zeta_2}
  \rangle |^q | \langle h_{\zeta'_1}, h_{\zeta'_2} \rangle |^r | \mu_{q,
  \zeta_1, \zeta_2} | | \mu_{q, \zeta_1', \zeta_2'} | \nonumber\\
  & \lesssim & \int | \langle h_{\zeta_1}, h_{\zeta_1'} \rangle |^{m + k - q}
  | \langle h_{\zeta_2}, h_{\zeta_2'} \rangle |^{n + \ell - 2 + \delta} |
  \langle h_{\zeta_1}, h_{\zeta_2} \rangle |^q | \mu_{q, \zeta_1, \zeta_2} | |
  \mu_{q, \zeta_1', \zeta_2'} | .  \label{e:graph-est-1}
\end{eqnarray}
(If vice-versa $q < 2$ it suffices to put $\delta$ on the term $| \langle
h_{\zeta_1}, h_{\zeta_2} \rangle |^{q + \delta}$.) Notice that in this case $m
+ k - q > 0.$

In the case $q + r = 3$ if $m + k - q \geqslant 2$ we estimate like before to
obtain
\begin{eqnarray}
  \mathfrak{I} & \lesssim & \varepsilon^{2 - \delta} \int | \langle
  h_{\zeta_1}, h_{\zeta_1'} \rangle |^{m + k - q} | \langle h_{\zeta_2},
  h_{\zeta_2'} \rangle |^{n + \ell - r} | \langle h_{\zeta_1}, h_{\zeta_2}
  \rangle |^{\frac{q + r}{2}} | \langle h_{\zeta'_1}, h_{\zeta'_2} \rangle
  |^{\frac{q + r}{2}} | \mu_{q, \zeta_1, \zeta_2} | | \mu_{q, \zeta_1',
  \zeta_2'} | \nonumber\\
  & \lesssim & \int | \langle h_{\zeta_1}, h_{\zeta_1'} \rangle |^{m + k - q}
  | \langle h_{\zeta_2}, h_{\zeta_2'} \rangle |^{n + \ell - r} | \langle
  h_{\zeta_1}, h_{\zeta_2} \rangle |^{1 + \delta} | \mu_{q, \zeta_1, \zeta_2}
  | | \mu_{q, \zeta_1', \zeta_2'} | .  \label{e:graph-est-2}
\end{eqnarray}
Note that $m + k - q + \delta - 1 > 0$ and $m + k - q + 2 \delta - 3 > - 1$
here. If $m + k - q = 1$ we bound
\begin{eqnarray}
  \mathfrak{I} & \lesssim & \int | \langle h_{\zeta_1}, h_{\zeta_1'} \rangle
  |^{m + k - q} | \langle h_{\zeta_2}, h_{\zeta_2'} \rangle |^{n + \ell - r -
  2} | \langle h_{\zeta_1}, h_{\zeta_2} \rangle |^{\frac{3 + \delta}{2}} |
  \langle h_{\zeta'_1}, h_{\zeta'_2} \rangle |^{\frac{3 + \delta}{2}} |
  \mu_{q, \zeta_1, \zeta_2} | | \mu_{q, \zeta_1', \zeta_2'} | 
  \label{e:graph-est-3}
\end{eqnarray}
and note that $m + k - q - 1 / 2 + \delta / 2 > 0$, $m + k - q - 1 + \delta >
0$, $n + \ell - r - 2 \geqslant 0$. Finally if $m + k - q = 0$ we can only
have $m + k = 3, q = 3, r = 0, i = 0$ and thus
\begin{eqnarray}
  \mathfrak{I} & \lesssim & \varepsilon^{3 - 2 \delta} \int | \langle
  h_{\zeta_2}, h_{\zeta_2'} \rangle |^{n + \ell} | \langle h_{\zeta_1},
  h_{\zeta_2} \rangle |^{2 - \delta} | \langle h_{\zeta'_1}, h_{\zeta'_2}
  \rangle |^{2 - \delta} | \mu_{q, \zeta_1, \zeta_2} | | \mu_{q, \zeta_1',
  \zeta_2'} | \nonumber\\
  & \lesssim & \int | \langle h_{\zeta_2}, h_{\zeta_2'} \rangle |^{n + \ell +
  m + k - 6} | \langle h_{\zeta_1}, h_{\zeta_2} \rangle |^{2 - \delta} |
  \langle h_{\zeta'_1}, h_{\zeta'_2} \rangle |^{2 - \delta} | \mu_{q, \zeta_1,
  \zeta_2} | | \mu_{q, \zeta_1', \zeta_2'} |  \label{e:graph-est-4}
\end{eqnarray}

If $q + r \geqslant 4$ we bound first
\[ \varepsilon^{2 q + 2 r - 2 i + \delta - 4} | \langle h_{\zeta_1},
   h_{\zeta_2} \rangle |^{q + r - i} | \langle h_{\zeta'_1}, h_{\zeta'_2}
   \rangle |^{q + r - i} \lesssim | \langle h_{\zeta_1}, h_{\zeta_2} \rangle
   |^{2 - \frac{\delta}{2}} | \langle h_{\zeta'_1}, h_{\zeta'_2} \rangle |^{2
   - \frac{\delta}{2}} \]
(note that $2 q + 2 r - 2 i + \delta - 4 \geqslant \delta$) to obtain:
\[ \mathfrak{I} \lesssim \varepsilon^{6 - q - r - \delta} \int | \langle
   h_{\zeta_1}, h_{\zeta_1'} \rangle |^{m + k - q} | \langle h_{\zeta_2},
   h_{\zeta_2'} \rangle |^{n + \ell - r} | \langle h_{\zeta_1}, h_{\zeta_2}
   \rangle |^{2 - \frac{\delta}{2}} | \langle h_{\zeta'_1}, h_{\zeta'_2}
   \rangle |^{2 - \frac{\delta}{2}} | \mu_{q, \zeta_1, \zeta_2} | | \mu_{q,
   \zeta_1', \zeta_2'} | \]
Now in the cases $m + k = 3, n + \ell = 3$ and $m + k = 4, n + \ell = 2$ \ we
can just write $\varepsilon^{6 - q - r - \delta} = \varepsilon^{m + k - q}
\varepsilon^{6 - m - k - r - \delta}$ and cancel the corresponding number of
covariances to obtain
\begin{eqnarray}
  \mathfrak{I} & \lesssim & \int | \langle h_{\zeta_2}, h_{\zeta_2'} \rangle
  |^{\delta} | \langle h_{\zeta_1}, h_{\zeta_2} \rangle |^{2 -
  \frac{\delta}{2}} | \langle h_{\zeta'_1}, h_{\zeta'_2} \rangle |^{2 -
  \frac{\delta}{2}} | \mu_{q, \zeta_1, \zeta_2} | | \mu_{q, \zeta_1',
  \zeta_2'} |  \label{e:graph-est-5}
\end{eqnarray}
while for the case $m + k = 4, n + \ell = 3$ \ we have either $\ell \geqslant
1$ or $k \geqslant 1$ and therefore with one of the following bounds
\begin{eqnarray*}
  \varepsilon^{m + k - 1 - q} \varepsilon^{n + \ell - r - \delta} | \langle
  h_{\zeta_1}, h_{\zeta_1'} \rangle |^{m + k - q} | \langle h_{\zeta_2},
  h_{\zeta_2'} \rangle |^{n + \ell - r} & \lesssim & | \langle h_{\zeta_1},
  h_{\zeta_1'} \rangle | | \langle h_{\zeta_2}, h_{\zeta_2'} \rangle
  |^{\delta}\\
  \varepsilon^{m + k - q} \varepsilon^{n + \ell - 1 - r - \delta} | \langle
  h_{\zeta_1}, h_{\zeta_1'} \rangle |^{m + k - q} | \langle h_{\zeta_2},
  h_{\zeta_2'} \rangle |^{n + \ell - r} & \lesssim & | \langle h_{\zeta_2},
  h_{\zeta_2'} \rangle |^{1 + \delta}
\end{eqnarray*}
we obtain the estimates
\begin{eqnarray}
  \mathfrak{I} & \lesssim & \int | \langle h_{\zeta_2}, h_{\zeta_2'} \rangle
  |^{1 + \delta} | \langle h_{\zeta_1}, h_{\zeta_2} \rangle |^{2 -
  \frac{\delta}{2}} | \langle h_{\zeta'_1}, h_{\zeta'_2} \rangle |^{2 -
  \frac{\delta}{2}} | \mu_{q, \zeta_1, \zeta_2} | | \mu_{q, \zeta_1',
  \zeta_2'} |  \label{e:graph-est-6}\\
  \mathfrak{I} & \lesssim & \int | \langle h_{\zeta_1}, h_{\zeta_1'} \rangle |
  | \langle h_{\zeta_2}, h_{\zeta_2'} \rangle |^{\delta} | \langle
  h_{\zeta_1}, h_{\zeta_2} \rangle |^{2 - \frac{\delta}{2}} | \langle
  h_{\zeta'_1}, h_{\zeta'_2} \rangle |^{2 - \frac{\delta}{2}} | \mu_{q,
  \zeta_1, \zeta_2} | | \mu_{q, \zeta_1', \zeta_2'} | .  \label{e:graph-est-7}
\end{eqnarray}
We can use directly Lemma~\ref{l:integral-est-final} to obtain a final
estimate of (\ref{e:graph-est-1}), (\ref{e:graph-est-2}),
(\ref{e:graph-est-3}), (\ref{e:graph-est-6}). For $(\ref{e:graph-est-4}),
(\ref{e:graph-est-5})$ and (\ref{e:graph-est-7}) notice that the integral over
$\zeta_1, \zeta'_1$ is finite and thus the whole quantity is proportional to
$| \langle h_{\zeta_2}, h_{\zeta_2'} \rangle |^n$. Globally, we have
\[ \mathfrak{I} \lesssim 2^{(m + k + n + \ell - 6) q} \]
as needed to prove~(\ref{e:check-conv-1}).

\begin{remark}
  \label{r:howmanyder}Finally, by controlling one more derivative of
  $F_{\varepsilon}$ as done in Section~\ref{s:first-trees-conv}, we can
  show~(\ref{e:check-conv-2}) for $Y^{\tau} = \tthreetwor{Y_{\varepsilon}},
  \ttwothreer{Y_{\varepsilon}}, \ttwothreer{\bar{Y}_{\varepsilon}},
  \tthreethreer{Y_{\varepsilon}}$, thus proving that $\hat{Y}^{\tau}
  \rightarrow 0$ in $C^{\kappa / 2}_T \CC^{\alpha - \kappa}$ in probability
  $\forall \alpha < | \tau |$. From the proof of Lemma~\ref{l:est-binary}
  together with this observation, we conclude that we need to control the
  derivatives of $F_{\varepsilon}$ up to order $9$ to be able to show the
  convergence for composite trees.
\end{remark}

\section{Convergence of the remainder and a-priori bounds}

In this section we prove the convergence of the remainder
(Lemma~\ref{l:remconv-last}), as well as some technical results on the norm of
the solution, needed in the proof of Theorem~\ref{t:lim-ident}. In order to
prove Lemma~\ref{l:remconv-last} we need first to prove
Lemma~\ref{l:rem-estim}, Lemma~\ref{l:rem-converg} and
Lemma~\ref{l:apriori-bounds} in this order.

\subsection{Boundedness of the remainder}\label{s:bound-rem-again}

We show that the remainder $R_{\varepsilon} (v_{\varepsilon})$ that appears in
equation~(\ref{e:last-phisharp}) can be controlled by a stochastic term
$M_{\varepsilon, \delta}$ that converges to zero in probability, times a
function of the solution $v_{\varepsilon}$. Let.
\begin{equation}
  M_{\varepsilon, \delta} (Y_{\varepsilon}, u_{0, \varepsilon}) \assign
  \varepsilon^{\delta / 2} \|e^{c \varepsilon^{1 / 2}  | Y_{\varepsilon} | + c
  \varepsilon^{1 / 2} | P_{\cdot}  (u_{0, \varepsilon} - Y_{\varepsilon} (0))
  |} \|_{L^p [0, T] L^p (\mathbbm{T}^3)} . \label{eq:def-M}
\end{equation}
for $p \in [1, \infty)$, $\delta \in [0, 1]$ and define
\begin{equation}
  v_{\varepsilon}^{\Box} \assign v_{\varepsilon} - v^{\sharp}_{\varepsilon}
  \quad \text{with} \quad v^{\sharp}_{\varepsilon} : t \mapsto P_t (u_{0,
  \varepsilon} - Y_{\varepsilon} (0)) . \label{e:vsquare-def}
\end{equation}

\begin{lemma}[Boundedness of remainder]
  \label{l:rem-estim} For every $\gamma \in (0, 1)$, $\delta \in [0, 1]$ we
  have
  \[ \|R_{\varepsilon} (v_{\varepsilon}, v^{\flat}_{\varepsilon},
     v^{\sharp}_{\varepsilon}) (t, x) \|_{\mathcal{M}^{\gamma / (3 + \delta),
     p}_T L^p (\mathbbm{T}^3)} \lesssim M_{\varepsilon} (Y_{\varepsilon},
     u_{0, \varepsilon}) \| v_{\varepsilon} \|^{3 +
     \delta}_{\mathcal{M}_T^{\gamma / (3 + \delta)} L^{\infty}
     (\mathbbm{T}^3)} e^{c \varepsilon^{1 / 2} \|v^{\Box}_{\varepsilon}
     \|_{C_T L^{\infty}}} \]
  with $M_{\varepsilon}$ as in~(\ref{eq:def-M}), $v^{\Box}_{\varepsilon}$ as
  in~(\ref{e:vsquare-def}) and $\mathcal{M}^{\gamma / (3 + \delta), p}_T L^p
  (\mathbbm{T}^3)$, $\mathcal{M}_T^{\gamma / (3 + \delta)} L^{\infty}
  (\mathbbm{T}^3)$ defined in~(\ref{e:expl-space-def}).
\end{lemma}

\begin{proof}
  We can write the remainder in two ways:
  \begin{eqnarray*}
    R_{\varepsilon} (v_{\varepsilon}) & = & v_{\varepsilon}^3 \int_0^1
    [F_{\varepsilon}^{(3)} (\varepsilon^{\frac{1}{2}} Y_{\varepsilon} + \tau
    \varepsilon^{\frac{1}{2}} v_{\varepsilon}) - F_{\varepsilon}^{(3)}
    (\varepsilon^{\frac{1}{2}} Y_{\varepsilon})]  \frac{(1 - \tau)^2}{2!}
    \mathd \tau\\
    & = & \varepsilon^{\tfrac{1}{2}} v_{\varepsilon}^4 \int_0^1
    F_{\varepsilon}^{(4)} (\varepsilon^{\frac{1}{2}} Y_{\varepsilon} + \tau
    \varepsilon^{\frac{1}{2}} v_{\varepsilon})  \frac{(1 - \tau)^3}{3!} \mathd
    \tau .
  \end{eqnarray*}
  From assumption~(\ref{e:hyp-F-main}) on $F_{\varepsilon}$ we obtain by
  interpolation of these two expressions, $\forall \delta \in [0, 1]$,
  $\forall t \geqslant 0, x \in \mathbbm{T}^3$,
  \[ | R_{\varepsilon} (v_{\varepsilon}) (t, x) | \lesssim \varepsilon^{\delta
     / 2} | v_{\varepsilon} (t, x) |^{3 + \delta} e^{c
     \varepsilon^{\frac{1}{2}}  | Y_{\varepsilon} (t, x) | + c
     \varepsilon^{\frac{1}{2}} | v^{\sharp}_{\varepsilon} (t, x) | + c |
     \varepsilon^{\frac{1}{2}} v^{\Box}_{\varepsilon} (t, x) |}, \]
  and we estimate, $\forall \gamma \in [0, 1)$,
  \[ \| t \mapsto t^{\gamma} R_{\varepsilon} (v_{\varepsilon}) (t, x) \|_{L^p
     ((0, T], L^p (\mathbbm{T}^3))} \lesssim \|t^{\frac{\gamma}{3 + \delta}}
     v_{\varepsilon} (t) \|_{C_T L^{\infty}}^{3 + \delta} e^{c \varepsilon^{1
     / 2} \|v^{\Box}_{\varepsilon} \|_{C_T L^{\infty}}} \left\|
     \varepsilon^{\frac{\delta}{2}} e^{c \varepsilon^{\frac{1}{2}}  |
     Y_{\varepsilon} (t, x) | + c \varepsilon^{\frac{1}{2}} |
     v^{\sharp}_{\varepsilon} (t, x) |} \right\|_{L^p [0, T] L^p
     (\mathbbm{T}^3)} . \]
\end{proof}

We can also verify that $M_{\varepsilon, \delta} \rightarrow 0$ in probability
for every $\delta > 0$:

\begin{lemma}[Convergence of the stochastic term]
  \label{l:rem-converg}Under Assumption~\ref{a:main} the random variable
  $M_{\varepsilon, \delta} (Y_{\varepsilon}, u_{0, \varepsilon})$ defined
  in~(\ref{eq:def-M}) converges to zero in probability $\forall \delta \in (0,
  1]$. 
\end{lemma}

\begin{proof}
  We can use Young's inequality to estimate $M_{\varepsilon, \delta}
  (Y_{\varepsilon}, u_{0, \varepsilon})$ for some $c' > 0$ as
  \begin{eqnarray*}
    M_{\varepsilon, \delta} (Y_{\varepsilon}, u_{0, \varepsilon}) & \lesssim &
    \varepsilon^{\delta / 2} \|e^{c' \varepsilon^{1 / 2}  | Y_{\varepsilon} |}
    \|_{L^p [0, T] L^p (\mathbbm{T}^3)} + \varepsilon^{\delta / 2} \|e^{c'
    \varepsilon^{1 / 2}  | P_{\point} Y_{\varepsilon} (0) |} \|_{L^p [0, T]
    L^p (\mathbbm{T}^3)}\\
    &  & + \varepsilon^{\delta / 2} T^{1 / p} e^{c'  \| \varepsilon^{1 / 2}
    u_{0, \varepsilon} \|_{L^{\infty}}} .
  \end{eqnarray*}
  Under Assumptions~\ref{a:main} the term $\| \varepsilon^{1 / 2} u_{0,
  \varepsilon} \|_{L^{\infty} (\mathbbm{T}^3)}$ is uniformly bounded, so the
  third term above converges to zero in probability. Note that $\varepsilon^{1
  / 2} Y_{\varepsilon} (t, x)$ and $P_t \varepsilon^{1 / 2} Y_{\varepsilon} (t
  = 0)$ are centered Gaussian random variables, and then both $\mathbbm{E}
  \|e^{c' \varepsilon^{1 / 2}  | Y_{\varepsilon} |} \|^p_{L^p [0, T] L^p
  (\mathbbm{T}^3)}$ and $\mathbbm{E} \|e^{c' \varepsilon^{1 / 2}  | P_{\point}
  Y_{\varepsilon} (0) |} \|^p_{L^p [0, T] L^p (\mathbbm{T}^3)}$ are uniformly
  bounded in $\varepsilon > 0$ for every $p \in [1, \infty)$. This yields the
  convergence in probability of $M_{\varepsilon, \delta} (Y_{\varepsilon},
  u_{0, \varepsilon})$.
\end{proof}

In order to show that $\| R_{\varepsilon} (v_{\varepsilon})
\|_{\mathcal{M}^{\gamma', p}_T L^p} \rightarrow 0$ in probability for $\gamma'
> \frac{1}{4} + \frac{3}{2} \kappa$ as needed in the proof of
Theorem~\ref{t:lim-ident}, we still need to control the norms $\|
v_{\varepsilon} \|^{3 + \delta}_{\mathcal{M}_T^{\gamma / (3 + \delta)}
L^{\infty}}$ and $\|v^{\Box}_{\varepsilon} \|_{C_T L^{\infty}}$ that appear in
Lemma~\ref{l:rem-estim}. This is done in next section.

\subsection{Apriori bounds on the solution}\label{s:a-priori}

\begin{lemma}[Apriori bound on the solution]
  \label{l:apriori-bounds}
  
  Fix $T > 0$. There exists $\kappa > 0$, $T_{\star} = T_{\star} \left( \|
  \mathbbm{Y}_{\varepsilon} \|_{\mathcal{X}_T}, \| u_{\varepsilon, 0}
  \|_{\CC^{- 1 / 2 - \kappa}}, | \lambda_{\varepsilon} | \right) \in (0, T]$ a
  lower semicontinuous function depending only on $\left( \|
  \mathbbm{Y}_{\varepsilon} \|_{\mathcal{X}_T}, \| u_{\varepsilon, 0}
  \|_{\CC^{- 1 / 2 - \kappa}}, | \lambda_{\varepsilon} | \right)$ and a
  collection of events $(\mathcal{E}_{\varepsilon})_{\varepsilon > 0}$ such
  that
  \[ \mathbbm{P} (\mathcal{E}_{\varepsilon}) \rightarrow 1 \qquad \text{as
     $\varepsilon \rightarrow 0$} \]
  and conditionally on $\mathcal{E}_{\varepsilon}$ there exists a universal
  constant $C > 0$ such that:
  \[ \| v_{\varepsilon} \|_{C_{T_{\star}} \CC^{- \frac{1}{2} - \kappa}} + \|
     v_{\varepsilon} \|_{\mathcal{M}_{T_{\star}}^{\frac{1}{4} + \frac{3
     \kappa}{2}} L^{\infty}} \leqslant C (1 + | \lambda_{\varepsilon} |)_{} 
     (1 + \| \mathbbm{Y}_{\varepsilon} \|_{\mathcal{X}_T})^3 \left( 1 + \|
     u_{\varepsilon, 0} \|_{\CC^{- 1 / 2 - \kappa}} \right)^3 \]
  for any $v_{\varepsilon}$ that solves equation~(\ref{e:last-phisharp}).
  Moreover, still conditionally on $\mathcal{E}_{\varepsilon}$ we have
  \[ \| v_{\varepsilon}^{\Box} \|_{C_{T_{\star}} L^{\infty}} \leqslant C (1 +
     | \lambda_{\varepsilon} |)_{}  (1 + \| \mathbbm{Y}_{\varepsilon}
     \|_{\mathcal{X}_T})^3 \left( 1 + \| u_{\varepsilon, 0} \|_{\CC^{- 1 / 2 -
     \kappa}} \right)^3 \]
  with $v_{\varepsilon}^{\Box}$ as in~(\ref{e:vsquare-def}).
\end{lemma}

\begin{proof}
  We know from Lemma~\ref{l:apriori-determ} that the bounds above on $\|
  v_{\varepsilon} \|_{C_{T_{\star}} \CC^{- \frac{1}{2} - \kappa}} + \|
  v_{\varepsilon} \|_{\mathcal{M}_{T_{\star}}^{\frac{1}{4} + \frac{3
  \kappa}{2}} L^{\infty}}$ and $\| v_{\varepsilon}^{\Box} \|_{C_{T_{\star}}
  L^{\infty}}$ hold whenever $M_{\varepsilon, \delta} \leqslant
  T_{\star}^{\kappa / 2}$. The event $\mathcal{E}_{\varepsilon} = \{
  M_{\varepsilon, \delta} \leqslant T_{\star}^{\kappa / 2} \}$ has
  $\mathbbm{P} (\mathcal{E}_{\varepsilon}) \rightarrow 1$ by
  Lemma~\ref{l:rem-converg} and this proves the result.
\end{proof}

The only thing left to prove is Lemma~\ref{l:apriori-determ}, which just a
standard application of some well-known bounds on paraproducts, that are
recalled in Appendices~\ref{sec:schauder}, and~\ref{sec:bony}.

First observe that for $\varepsilon > 0$ a pair $(v_{\varepsilon},
v^{\natural}_{\varepsilon})$ solves the paracontrolled
equation~(\ref{e:last-phisharp}) if and only if $v^{\natural}_{\varepsilon} =
v^{\flat}_{\varepsilon} + v^{\sharp}_{\varepsilon}$ and $(v_{\varepsilon},
v^{\flat}_{\varepsilon})$ solves:
\begin{equation}
  \left\{ \begin{array}{lll}
    v_{\varepsilon} & = & - \tthreeone{Y_{\varepsilon}} -
    \ttwoone{\bar{Y}_{\varepsilon}} - 3 v_{\varepsilon} \precprec
    \ttwoone{Y_{\varepsilon}} + v^{\flat}_{\varepsilon} +
    v^{\sharp}_{\varepsilon}\\
    \LL v_{\varepsilon}^{\flat} & = & U (\lambda_{\varepsilon},
    \mathbbm{Y}_{\varepsilon} ; v_{\varepsilon}, v^{\flat}_{\varepsilon} +
    v^{\sharp}_{\varepsilon}) - R_{\varepsilon} (v_{\varepsilon})\\
    v^{\flat}_{\varepsilon} (0) & = & \tthreeone{Y_{\varepsilon}} (0) +
    \ttwoone{\bar{Y}_{\varepsilon}} (0) + 3 v_{\varepsilon, 0} \prec
    \ttwoone{Y_{\varepsilon}} (0)
  \end{array} \right. \label{e:fixpoint-final}
\end{equation}
Here $U$ is the same as in~(\ref{e:last-phisharp}). The initial condition of
~(\ref{e:fixpoint-final}) is given by $v_{\varepsilon, 0} \assign u_{0,
\varepsilon} - Y_{\varepsilon} (0)$. The a-priori bounds of
Lemma~\ref{l:apriori-determ} come from being able to find closed estimates
for~(\ref{e:fixpoint-final}).

Let us specify now all the notations we are going to use in the rest of this
section. We consider the spaces
\[ \mathcal{V}_T^{\flat} \assign \LL^{2 \kappa}_T \cap \LL^{1 / 4, 1 / 2 + 2
   \kappa}_T \cap \LL^{1 / 2, 1 + 2 \kappa}_T, \quad \mathcal{V}_T \assign
   \LL^{1 / 2, 1 / 2 - \kappa}_T \cap \LL_T^{1 / 4 + 3 \kappa / 2, 2 \kappa},
\]
with the corresponding norms
\begin{eqnarray}
  \| v^{\flat}_{\varepsilon} \|_{\mathcal{V}_T^{\flat}} & \assign & \|
  v^{\flat}_{\varepsilon} \|_{\LL^{2 \kappa}_T} + \| v^{\flat}_{\varepsilon}
  \|_{\LL^{1 / 4, 1 / 2 + 2 \kappa}_T} + \| v_{\varepsilon}^{\flat} \|_{\LL^{1
  / 2, 1 + 2 \kappa}_T}, \label{e:vb-norm} 
\end{eqnarray}
\begin{eqnarray}
  \| v_{\varepsilon} \|_{\mathcal{V}_T} & \assign & \| v_{\varepsilon}
  \|_{\LL^{1 / 2, 1 / 2 - \kappa}} + \| v_{\varepsilon} \|_{\LL^{1 / 4 + 3
  \kappa / 2, 2 \kappa}} . \label{e:v-norm} 
\end{eqnarray}
We refer to Appendix~\ref{s:notations} for the definition of the parabolic
spaces $\LL^{\gamma, \alpha}_T$. We let
\[ \begin{array}{lllll}
     v^{\Box}_{\varepsilon} & \assign & v_{\varepsilon} - v^{\sharp} & = & -
     \tthreeone{Y_{\varepsilon}} - \ttwoone{\bar{Y}_{\varepsilon}} - 3
     (v^{\Box}_{\varepsilon} + v^{\sharp}_{\varepsilon}) \precprec
     \ttwoone{Y_{\varepsilon}} + v^{\flat}_{\varepsilon},\\
     v_{\varepsilon}^{\boxslash} & \assign & v^{\Box}_{\varepsilon} +
     \tthreeone{Y_{\varepsilon}} & = & - \ttwoone{\bar{Y}_{\varepsilon}} - 3
     (v^{\Box}_{\varepsilon} + v^{\sharp}_{\varepsilon}) \precprec
     \ttwoone{Y_{\varepsilon}} + v^{\flat}_{\varepsilon},
   \end{array} \]
and $v^{\Box}_{\varepsilon} (t = 0) = v_{\varepsilon}^{\boxslash} (t = 0) =
0$. We define also the norm
\begin{eqnarray*}
  \| v^{\Box}_{\varepsilon} \|_{\mathcal{V}_T^{\Box}}^{} & \assign & \|
  v^{\Box} \|_{\LL^{2 \kappa}_T} + \| v^{\Box} \|_{\mathcal{M}^{1 / 4}_T
  \CC^{1 / 2 + 2 \kappa}},
\end{eqnarray*}
with $\mathcal{M}^{1 / 4}_T \CC^{1 / 2 + 2 \kappa}$ given in
Appendix~\ref{s:notations}. In order not to get lost in these definitions the
reader can keep in mind the following:
\begin{itemizeminus}
  \item $v_{\varepsilon}$ is the solution without the linear term;
  
  \item $v^{\sharp}_{\varepsilon}$ is the contribution of the initial
  condition, which give origin to some explosive norm (near the initial time);
  
  \item $v^{\flat}_{\varepsilon}$ is the regular part of the solution;
  
  \item $v^{\Box}_{\varepsilon}, v_{\varepsilon}^{\boxslash}$ enter in the
  estimation of the remainder, they are just convenient shortcuts for certain
  contributions appearing in $v_{\varepsilon}$.
\end{itemizeminus}
\begin{lemma}
  \label{l:apriori-determ}There exists $T_{\star} = T_{\star} \left( \|
  \mathbbm{Y}_{\varepsilon} \|_{\mathcal{X}_T}, \| u_{\varepsilon, 0}
  \|_{\CC^{- 1 / 2 - \kappa}}, | \lambda_{\varepsilon} | \right) \in (0, T]$ a
  lower semicontinuous function depending only on $\|
  \mathbbm{Y}_{\varepsilon} \|_{\mathcal{X}_T}$, $\| u_{\varepsilon, 0}
  \|_{\CC^{- 1 / 2 - \kappa}}$ and $| \lambda_{\varepsilon} |$, a constant
  $M_{\varepsilon, \delta} = M_{\varepsilon, \delta} (Y_{\varepsilon}, u_{0,
  \varepsilon}) > 0$ defined by~(\ref{eq:def-M}), and a universal constant $C
  > 0$ such that, whenever $M_{\varepsilon, \delta} \leqslant
  T_{\star}^{\kappa / 2}$ we have
  \begin{eqnarray*}
    \| v^{\flat}_{\varepsilon} \|_{\mathcal{V}_{T_{\star}}^{\flat}} &
    \leqslant & C (1 + | \lambda_{\varepsilon} |)_{}  (1 + \|
    \mathbbm{Y}_{\varepsilon} \|_{\mathcal{X}_T})^3 \left( 1 + \|
    u_{\varepsilon, 0} \|_{\CC^{- 1 / 2 - \kappa}} \right)^3,\\
    \| v_{\varepsilon} \|_{\mathcal{V}_{T_{\star}}} & \leqslant & C \left( \|
    \mathbbm{Y}_{\varepsilon} \|_{\mathcal{X}_T} + \| u_{\varepsilon, 0}
    \|_{\CC^{- 1 / 2 - \kappa}} + \| v^{\flat}_{\varepsilon}
    \|_{\mathcal{V}_{T_{\star}}^{\flat}} \right) .
  \end{eqnarray*}
\end{lemma}

\begin{proof}
  Using the well-known Schauder estimates of Lemma~\ref{lemma:schauder} (and
  the fact that $\| f \|_{\LL_T^{\kappa, \alpha}} \lesssim T^{\kappa} \| f
  \|_{\LL_T^{\alpha}}$) we obtain for $\kappa, \theta > 0$ small enough
  \begin{equation}
    \| I f \|_{\LL^{- \theta + 2 \kappa, 2 \kappa}_T} + \| I f \|_{\LL^{1 / 4
    - \theta + 2 \kappa, 1 / 2 + 2 \kappa}_T} + \| I f \|_{\LL^{1 / 2 - \theta
    + 2 \kappa, 1 + 2 \kappa}_T} \lesssim T^{\frac{\kappa}{2}} \left( \| f
    \|_{\mathcal{M}^{1 - \theta}_T \CC^{- \kappa}} + \| f \|_{\mathcal{M}^{1 /
    2 + 2 \kappa} \CC^{- 1 / 2 - 2 \kappa}} \right) . \label{e:est-1}
  \end{equation}
  We choose $\theta > 2 \kappa$ small enough so that
  \[ \LL^{- \theta + 3 \kappa / 2, 2 \kappa}_T \cap \LL^{1 / 4 - \theta + 3
     \kappa / 2, 1 / 2 + 2 \kappa}_T \cap \LL^{1 / 2 - \theta + 3 \kappa / 2,
     1 + 2 \kappa}_T \subseteq \mathcal{V}_T^{\flat} . \]
  Now
  \begin{eqnarray*}
    \| v_{\varepsilon}^{\Box} \|_{\mathcal{V}_T^{\Box}} & \lesssim & \|
    \tthreeone{Y_{\varepsilon}} + \ttwoone{\bar{Y}_{\varepsilon}}
    \|_{\mathcal{V}_T^{\Box}} + \|v^{\Box}_{\varepsilon} \|_{C_{_T}
    L^{\infty}} (\| \ttwoone{Y_{\varepsilon}} \|_{C_T \CC^{1 - \kappa}} +\|
    \ttwo{Y_{\varepsilon}} \|_{C_T \CC^{- 1 - \kappa}})\\
    &  & + \left( \| v_{\varepsilon}^{\sharp} \|_{C_T \CC^{- 1 / 2 - \kappa}}
    + \| v_{\varepsilon}^{\sharp} \|_{\mathcal{M}_T^{1 / 4} \CC^{- \kappa}}
    \right) (\| \ttwoone{Y_{\varepsilon}} \|_{C_T \CC^{1 - \kappa}} +\|
    \ttwo{Y_{\varepsilon}} \|_{C_T \CC^{- 1 - \kappa}}) + \|
    v_{\varepsilon}^{\flat} \|_{\mathcal{V}_T^{\Box}}\\
    & \lesssim & \| \mathbbm{Y}_{\varepsilon} \|_{\mathcal{X}_T} + T^{\kappa}
    \| v_{\varepsilon}^{\Box} \|_{\mathcal{V}_T^{\Box}} + \| v_{\varepsilon,
    0} \|_{\CC^{- 1 / 2 - \kappa}} + \| v^{\flat}_{\varepsilon}
    \|_{\mathcal{V}_T^{\flat}}
  \end{eqnarray*}
  where we used that $v^{\Box} (0) = 0$ and as a consequence that
  $\|v^{\Box}_{\varepsilon} \|_{C_{_T} L^{\infty}} \leqslant T^{\kappa}
  \|v^{\Box}_{\varepsilon} \|_{C_{_T}^{\kappa} L^{\infty}} \leqslant
  T^{\kappa} \| v^{\Box}_{\varepsilon} \|_{\mathcal{V}_T^{\Box}}^{}$ to gain a
  small power of $T$. So provided $T$ is small enough (depending only on
  $\mathbbm{Y}_{\varepsilon}$) this yields the following a-priori estimation
  on $v_{\varepsilon}^{\Box}$:
  \[ \| v_{\varepsilon}^{\Box} \|_{C_T L^{\infty}} \lesssim \|
     v_{\varepsilon}^{\Box} \|_{\mathcal{V}_T^{\Box}} \lesssim \|
     \mathbbm{Y}_{\varepsilon} \|_{\mathcal{X}_T} + \| v_{\varepsilon, 0}
     \|_{\CC^{- 1 / 2 - \kappa}} + \| v^{\flat}_{\varepsilon}
     \|_{\mathcal{V}_T^{\flat}} . \]
  Therefore we have an estimation on $v_{\varepsilon}$:
  \[ \| v_{\varepsilon} \|_{\mathcal{V}_T} \leqslant \|
     v_{\varepsilon}^{\sharp} \|_{\mathcal{V}_T} + \| v^{\Box}
     \|_{\mathcal{V}_T} \lesssim \| v_{\varepsilon, 0} \|_{\CC^{- 1 / 2 -
     \kappa}} + \| v_{\varepsilon}^{\Box} \|_{\mathcal{V}_T^{\Box}} \lesssim
     \| \mathbbm{Y}_{\varepsilon} \|_{\mathcal{X}_T} + \| v_{\varepsilon, 0}
     \|_{\CC^{- 1 / 2 - \kappa}} + \| v^{\flat}_{\varepsilon}
     \|_{\mathcal{V}_T^{\flat}} . \]
  In order to estimate terms in $U (\lambda_{\varepsilon},
  \mathbbm{Y}_{\varepsilon} ; v_{\varepsilon}, v^{\flat}_{\varepsilon} +
  v^{\sharp}_{\varepsilon})$ we decompose the renormalised products as
  \[ \begin{array}{lll}
       \ttwo{Y_{\varepsilon}} \hat{\diamond} v_{\varepsilon} & = &
       v_{\varepsilon} \succ \ttwo{Y_{\varepsilon}} -
       \ttwothreer{\bar{Y}_{\varepsilon}} - \tthreethreer{Y_{\varepsilon}} - 3
       v_{\varepsilon}  \ttwothreer{Y_{\varepsilon}} + v^{\flat}_{\varepsilon}
       \circ \ttwo{Y_{\varepsilon}} + v^{\sharp}_{\varepsilon} \circ
       \ttwo{Y_{\varepsilon}} - 3 \overline{\tmop{com}}_1 (v_{\varepsilon},
       \ttwoone{Y_{\varepsilon}}, \ttwo{Y_{\varepsilon}})\\
       v_{\varepsilon} \diamond Y_{\varepsilon} & = & -
       \ttwoone{\bar{Y}_{\varepsilon}} Y_{\varepsilon} - 3 (v_{\varepsilon}
       \precprec \ttwoone{Y_{\varepsilon}}) Y_{\varepsilon} + Y_{\varepsilon}
       \preccurlyeq (v^{\flat}_{\varepsilon} + v^{\sharp}_{\varepsilon})_{} +
       Y_{\varepsilon} \succ (v^{\flat}_{\varepsilon} +
       v^{\sharp}_{\varepsilon})_{}\\
       &  & - \tthreeone{Y_{\varepsilon}} \prec Y_{\varepsilon} -
       \tthreeone{Y_{\varepsilon}} \succ Y_{\varepsilon} -
       \tthreetwor{Y_{\varepsilon}}\\
       \tone{Y_{\varepsilon}} \diamond v_{\varepsilon}^2 & = &
       \tone{Y_{\varepsilon}} \diamond (\tthreeone{Y_{\varepsilon}})^2 + 2
       (\tone{Y_{\varepsilon}} \diamond \tthreeone{Y_{\varepsilon}})
       (\ttwoone{\bar{Y}_{\varepsilon}} + 3 v_{\varepsilon} \precprec
       \ttwoone{Y_{\varepsilon}}) - 2 (\tone{Y_{\varepsilon}} \diamond
       \tthreeone{Y_{\varepsilon}}) \preccurlyeq (v^{\flat}_{\varepsilon} +
       v^{\sharp}_{\varepsilon})\\
       &  & + 2 (\tone{Y_{\varepsilon}} \diamond \tthreeone{Y_{\varepsilon}})
       \succ (v^{\flat}_{\varepsilon} + v^{\sharp}_{\varepsilon}) +
       \tone{Y_{\varepsilon}} \preccurlyeq (v^{\boxslash}_{\varepsilon} +
       v^{\sharp})^2 + \tone{Y_{\varepsilon}} \succ
       (v^{\boxslash}_{\varepsilon} + v^{\sharp})^2 .
     \end{array} \]
  We decompose $U (\lambda_{\varepsilon}, \mathbbm{Y}_{\varepsilon} ;
  v_{\varepsilon}, v^{\flat}_{\varepsilon} + v^{\sharp}_{\varepsilon})$ as
  \begin{eqnarray*}
    U (\lambda_{\varepsilon}, \mathbbm{Y}_{\varepsilon} ; v_{\varepsilon},
    v^{\flat}_{\varepsilon} + v^{\sharp}_{\varepsilon}) & = & Q_{- 1 / 2}
    (\lambda_{\varepsilon}, \mathbbm{Y}_{\varepsilon}, v_{0, \varepsilon},
    v_{\varepsilon}, v^{\flat}_{\varepsilon}) + Q_0 (\lambda_{\varepsilon},
    \mathbbm{Y}_{\varepsilon}, v_{0, \varepsilon}, v_{\varepsilon},
    v^{\flat}_{\varepsilon}) + Q_{\lambda_{\varepsilon},
    \mathbbm{Y}_{\varepsilon}}\\
    Q_{- 1 / 2} & \assign & - 3 [v_{\varepsilon} \succ \ttwo{Y_{\varepsilon}}
    - 3 \overline{\tmop{com}}_1 (v_{\varepsilon}, \ttwoone{Y_{\varepsilon}},
    \ttwo{Y_{\varepsilon}}) + \tone{Y_{\varepsilon}} \succ
    (v^{\boxslash}_{\varepsilon} + v^{\sharp})^2]\\
    &  & - 6 [(\tone{Y_{\varepsilon}} \diamond \tthreeone{Y_{\varepsilon}})
    (3 v_{\varepsilon} \precprec \ttwoone{Y_{\varepsilon}}) +
    (\tone{Y_{\varepsilon}} \diamond \tthreeone{Y_{\varepsilon}}) \succ
    (v^{\flat}_{\varepsilon} + v^{\sharp}_{\varepsilon})]\\
    &  & + 2 \lambda_{2, \varepsilon} (3 (v_{\varepsilon} \precprec
    \ttwoone{Y_{\varepsilon}}) Y_{\varepsilon} - Y_{\varepsilon} \succ
    (v^{\flat}_{\varepsilon} + v^{\sharp}_{\varepsilon})_{}) + 3 \tmop{com}_3
    (v_{\varepsilon}, \ttwoone{Y_{\varepsilon}}) + 3 \tmop{com}_2
    (v_{\varepsilon}, \ttwo{Y_{\varepsilon}})\\
    Q_0 & \assign & 3 [3 v_{\varepsilon}  \ttwothreer{Y_{\varepsilon}} -
    v^{\flat}_{\varepsilon} \circ \ttwo{Y_{\varepsilon}} -
    v^{\sharp}_{\varepsilon} \circ \ttwo{Y_{\varepsilon}} + 2
    (\tone{Y_{\varepsilon}} \diamond \tthreeone{Y_{\varepsilon}}) \preccurlyeq
    (v^{\flat}_{\varepsilon} + v^{\sharp}_{\varepsilon}) -
    \tone{Y_{\varepsilon}} \preccurlyeq (v^{\boxslash}_{\varepsilon} +
    v^{\sharp})^2]\\
    &  & - \tzero{Y_{\varepsilon}} v_{\varepsilon}^3 - \lambda_{2,
    \varepsilon} [v_{\varepsilon}^2 + 2 Y_{\varepsilon} \preccurlyeq
    (v^{\flat}_{\varepsilon} + v^{\sharp}_{\varepsilon})]\\
    Q_{\lambda_{\varepsilon}, \mathbbm{Y}_{\varepsilon}} & \assign & (1 -
    \lambda_{1, \varepsilon}) Y_{\varepsilon} - \lambda_{0, \varepsilon} + 3
    [\ttwothreer{\bar{Y}_{\varepsilon}} + \tthreethreer{Y_{\varepsilon}} -
    \tone{Y_{\varepsilon}} \diamond (\tthreeone{Y_{\varepsilon}})^2 - 2
    (\tone{Y_{\varepsilon}} \diamond \tthreeone{Y_{\varepsilon}})
    \ttwoone{\bar{Y}_{\varepsilon}}]\\
    &  & + 2 \lambda_{2, \varepsilon} (\ttwoone{\bar{Y}_{\varepsilon}}
    Y_{\varepsilon} + \tthreeone{Y_{\varepsilon}} \prec Y_{\varepsilon} +
    \tthreeone{Y_{\varepsilon}} \succ Y_{\varepsilon} +
    \tthreetwor{Y_{\varepsilon}}) .
  \end{eqnarray*}
  Here $Q_{\lambda_{\varepsilon}, \mathbbm{Y}_{\varepsilon}}$ does not depend
  from the solution but only on $\lambda_{\varepsilon},
  \mathbbm{Y}_{\varepsilon}$ (as the notation suggests) and we have grouped
  the other terms which we expect to have regularity $\CC^{- 1 / 2 - 2
  \kappa}$ in $Q_{- 1 / 2}$, (and the same for $Q_0$ and regularity $\CC^{-
  k}$). With the same technique we used above for $v^{\Box}_{\varepsilon}$, we
  obtain the following estimate on $v^{\boxslash}_{\varepsilon}$
  \[ \| v^{\boxslash}_{\varepsilon} \|_{\LL_T^{1 / 2 + 3 \kappa / 2, 1 / 2 + 2
     \kappa}} + \| v^{\boxslash}_{\varepsilon} \|_{\LL_T^{1 / 4 + \kappa,
     \kappa}} \lesssim \| \mathbbm{Y}_{\varepsilon} \|_{\mathcal{X}_T} + \|
     v_{\varepsilon, 0} \|_{\CC^{- 1 / 2 - \kappa}} + \|
     v^{\flat}_{\varepsilon} \|_{\mathcal{V}^{\flat}_T} \]
  and this yields
  \[ \| (v^{\boxslash}_{\varepsilon})^2 \|_{\LL_T^{3 / 4 + 5 \kappa / 2, 1 / 2
     + 2 \kappa}} + \| (v^{\boxslash}_{\varepsilon})^2 \|_{\LL_T^{1 / 2 + 2
     \kappa, \kappa}} \lesssim \left( \| \mathbbm{Y}_{\varepsilon}
     \|_{\mathcal{X}_T} + \| v_{\varepsilon, 0} \|_{\CC^{- 1 / 2 - \kappa}} +
     \| v^{\flat}_{\varepsilon} \|_{\mathcal{V}^{\flat}_T} \right)^2 . \]
  Then we are ready to bound $Q_{- 1 / 2}, Q_0, Q_{\lambda_{\varepsilon},
  \mathbbm{Y}_{\varepsilon}}$ using the standard paraproducts estimations
  recalled in Appendix~\ref{sec:bony}:
  \[ \begin{array}{lll}
       \| Q_{- 1 / 2} \|_{\mathcal{M}_T^{1 / 2 + 2 \kappa} \CC^{- 1 / 2 - 2
       \kappa}} + \| Q_0 \|_{\mathcal{M}^{1 - \theta}_T \CC^{- \kappa}} &
       \lesssim & (1 + | \lambda_{\varepsilon} |) (1 + \|
       \mathbbm{Y}_{\varepsilon} \|_{\mathcal{X}_T})^3 \left( 1 + \|
       v_{\varepsilon, 0} \|_{\CC^{- 1 / 2 - \kappa}} + \|
       v^{\flat}_{\varepsilon} \|_{\mathcal{V}^{\flat}_T} \right)^3\\
       \| Q_{\lambda_{\varepsilon}, \mathbbm{Y}_{\varepsilon}} \|_{C_T \CC^{-
       1 / 2 - \kappa}} & \lesssim & (1 + | \lambda_{\varepsilon} |)_{}  (1 +
       \| \mathbbm{Y}_{\varepsilon} \|_{\mathcal{X}_T})^3 .
     \end{array} \]

  In order to conclude the estimation of $\| v^{\flat}_{\varepsilon}
  \|_{\mathcal{V}^{\flat}_T}$ we have to control $\| I R_{\varepsilon}
  (v_{\varepsilon}) \|_{\mathcal{V}^{\flat}_T}$. This is achieved easily by
  the using the results of Section~\ref{s:bound-rem-again}. Thanks to
  Lemma~\ref{l:rem-estim} $\forall \delta \in (0, 1)$, $\forall \theta > 0$
  such that $\frac{1 - \theta}{3 + \delta} > \frac{1}{4} + \frac{3 \kappa}{2}$
  (note that it is possible to choose $\theta > 2 \kappa$ that satisfies this
  property as long as $k$ and $\delta$ are small enough) we have:
  \[ \|R_{\varepsilon} (v_{\varepsilon}) \|_{\mathcal{M}^{1 - \theta, p}_T L^p
     (\mathbbm{T}^3)} \lesssim M_{\varepsilon, \delta} (Y_{\varepsilon}, u_{0,
     \varepsilon}) \| v_{\varepsilon} \|^{3 + \delta}_{\mathcal{V}_T} e^{c
     \varepsilon^{1 / 2} \|v^{\Box}_{\varepsilon} \|_{\mathcal{V}^{\Box}_T}} .
  \]
  By Lemma~\ref{l:lp-integ} together with~(\ref{eq:schauder time reg}) we
  obtain then
  \begin{eqnarray*}
    \| I R_{\varepsilon} (v_{\varepsilon}) \|_{\mathcal{V}^{\flat}_T} &
    \lesssim & M_{\varepsilon, \delta} (Y_{\varepsilon}, u_{0, \varepsilon})
    \| v_{\varepsilon} \|^{3 + \delta}_{\mathcal{V}_T} e^{c \varepsilon^{1 /
    2} \|v^{\Box}_{\varepsilon} \|_{\mathcal{V}_T^{\Box}}} .
  \end{eqnarray*}
  Using that
  \[ \| P_{\cdummy} v^{\flat}_{\varepsilon} (0) \|_{\mathcal{V}_T^{\flat}}
     \lesssim \| v^{\flat}_{\varepsilon} (0) \|_{C_T \CC^{1 / 2 - 2 \kappa}}
     \lesssim \left( 1 + \| v_{\varepsilon, 0} \|_{\CC^{- 1 / 2 - \kappa}}
     \right) \| \mathbbm{Y}_{\varepsilon} \|_{\mathcal{X}_T} \]
  we obtain that $\exists C' > 0$ such that
  \begin{eqnarray*}
    \| v^{\flat}_{\varepsilon_n} \|_{\mathcal{V}_T^{\flat}} & \leqslant & C'
    (1 + | \lambda_{\varepsilon_n} |)_{}  (1 + \| \mathbbm{Y}_{\varepsilon_n}
    \|_{\mathcal{X}_T})^3 \left( 1 + \| v_{\varepsilon, 0} \|_{\CC^{- 1 / 2 -
    \kappa}} \right)^3 + C' T^{\kappa / 2} (1 + | \lambda_{\varepsilon} |) (1
    + \| \mathbbm{Y}_{\varepsilon} \|_{\mathcal{X}_T})^3 \|
    v^{\flat}_{\varepsilon} \|_{\mathcal{V}^{\flat}_T}^3\\
    &  & + C' M_{\varepsilon, \delta} (Y_{\varepsilon}, u_{0, \varepsilon})
    e^{c \varepsilon^{1 / 2} \left( \| \mathbbm{Y}_{\varepsilon}
    \|_{\mathcal{X}_T} + \| v_{\varepsilon, 0} \|_{\CC^{- 1 / 2 - \kappa}}
    \right)} e^{c \varepsilon^{1 / 2} \| v^{\flat}_{\varepsilon}
    \|_{\mathcal{V}_T^{\flat}}} \| v_{\varepsilon} \|^{3 +
    \delta}_{\mathcal{V}_T}\\
    & \leqslant & D + CM_{\varepsilon, \delta} (Y_{\varepsilon}, u_{0,
    \varepsilon}) e^{c \varepsilon^{1 / 2} \| v^{\flat}_{\varepsilon}
    \|_{\mathcal{V}_T^{\flat}}} + CT^{\kappa / 2} \| v^{\flat}_{\varepsilon}
    \|_{\mathcal{V}^{\flat}_T}^3 + CM_{\varepsilon, \delta} (Y_{\varepsilon},
    u_{0, \varepsilon}) e^{c \varepsilon^{1 / 2} \| v^{\flat}_{\varepsilon}
    \|_{\mathcal{V}_T^{\flat}}} \| v^{\flat}_{\varepsilon}
    \|_{\mathcal{V}^{\flat}_T}^{3 + \delta}
  \end{eqnarray*}
  with
  \[ C \assign C' [(1 + | \lambda_{\varepsilon} |) (1 + \|
     \mathbbm{Y}_{\varepsilon} \|_{\mathcal{X}_T})^3 + e^{c \varepsilon^{1 /
     2} \left( \| \mathbbm{Y}_{\varepsilon} \|_{\mathcal{X}_T} + \|
     v_{\varepsilon, 0} \|_{\CC^{- 1 / 2 - \kappa}} \right)} \left( 1 + \left(
     \| \mathbbm{Y}_{\varepsilon} \|_{\mathcal{X}_T} + \| v_{\varepsilon, 0}
     \|_{\CC^{- 1 / 2 - \kappa}} \right)^{3 + \delta} \right)], \]
  and
  \[ D \assign C' (1 + | \lambda_{\varepsilon_n} |)_{}  (1 + \|
     \mathbbm{Y}_{\varepsilon_n} \|_{\mathcal{X}_T})^3 \left( 1 + \|
     v_{\varepsilon, 0} \|_{\CC^{- 1 / 2 - \kappa}} \right)^3 . \]
  Let $T_{\star} \in (0, T]$ such that:
  \[ CT_{\star}^{\kappa / 2} [(5 C)^2 + e^{c \varepsilon^{1 / 2} (5 C)} (5
     C)^{2 + \delta}] \leqslant \frac{1}{2} \text{, \quad and }
     CT_{\star}^{\kappa / 2} e^{c \varepsilon^{1 / 2} (5 C)} \leqslant D. \]
  Assume that $M_{\varepsilon, \delta} \leqslant T_{\star}^{\kappa / 2}$.
  Define a closed interval $[0, S] = \{t \in [0, T_{\star}] : \|
  v^{\flat}_{\varepsilon_n} \|_{\mathcal{V}_t^{\flat}} \leqslant 4 D\}
  \subseteq [0, T_{\star}]$. This interval is well defined and non--empty
  since $t \mapsto \| v^{\flat}_{\varepsilon_n} \|_{\mathcal{V}_t^{\flat}}$ is
  continuous and nondecreasing and $\| v^{\flat}_{\varepsilon_n}
  \|_{\mathcal{V}_0^{\flat}} \leqslant 4 D$. Let us assume that $S <
  T_{\star}$, then we can take $\epsilon > 0$ small enough such that $S +
  \epsilon < T_{\star}$ and by continuity $\| v^{\flat}_{\varepsilon}
  \|_{\mathcal{V}_{S + \epsilon}^{\flat}} \leqslant 5 C$, then
  \begin{eqnarray*}
    \| v^{\flat}_{\varepsilon_n} \|_{\mathcal{V}_{S + \epsilon}^{\flat}} &
    \leqslant & D + CM_{\varepsilon, \delta} (Y_{\varepsilon}, u_{0,
    \varepsilon}) e^{c \varepsilon^{1 / 2} \| v^{\flat}_{\varepsilon}
    \|_{\mathcal{V}_{S + \epsilon}^{\flat}}} + C (S + \epsilon)^{\kappa / 2}
    \| v^{\flat}_{\varepsilon} \|_{\mathcal{V}^{\flat}_{S + \epsilon}}^3 +
    CM_{\varepsilon, \delta} (Y_{\varepsilon}, u_{0, \varepsilon}) e^{c
    \varepsilon^{1 / 2} \| v^{\flat}_{\varepsilon} \|_{\mathcal{V}_{S +
    \epsilon}^{\flat}}} \| v^{\flat}_{\varepsilon} \|_{\mathcal{V}^{\flat}_{S
    + \epsilon}}^{3 + \delta}\\
    & \leqslant & D + CM_{\varepsilon, \delta} (Y_{\varepsilon}, u_{0,
    \varepsilon}) e^{c \varepsilon^{1 / 2} (5 C)} + CT_{\star}^{\kappa / 2} (5
    C)^2 \| v^{\flat}_{\varepsilon} \|_{\mathcal{V}^{\flat}_{S + \epsilon}} +
    CT_{\star}^{\kappa / 2} e^{c \varepsilon^{1 / 2} (5 C)} (5 C)^{2 + \delta}
    \| v^{\flat}_{\varepsilon} \|_{\mathcal{V}^{\flat}_{S + \epsilon}}\\
    & \leqslant & 2 D + \frac{1}{2} \| v^{\flat}_{\varepsilon}
    \|_{\mathcal{V}^{\flat}_{S + \epsilon}}
  \end{eqnarray*}
  which gives $\| v^{\flat}_{\varepsilon_n} \|_{\mathcal{V}_{S +
  \epsilon}^{\flat}} \leqslant 4 D$. This implies $S = T_{\star}$ (by
  contradiction). From the construction of $T_{\star}$ it is easy to see that
  $T_{\star} \left( \| \mathbbm{Y}_{\varepsilon} \|_{\mathcal{X}_T}, \|
  u_{\varepsilon, 0} \|_{\CC^{- 1 / 2 - \kappa}}, | \lambda_{\varepsilon} |
  \right)$ is lower semicontinuous.
\end{proof}

\subsection{\label{s:conv-rem}Convergence of the remainder}

It suffices to put together the results obtained in
Sections~\ref{s:bound-rem-again} and~\ref{s:a-priori} to obtain the
convergence of $R_{\varepsilon} (v_{\varepsilon})$:

\begin{lemma}
  \label{l:remconv-last}The remainder $R_{\varepsilon} (v_{\varepsilon})$ that
  appears in equation~(\ref{e:last-phisharp}) converges in probability to $0$
  as $\varepsilon \rightarrow 0$ in the space $\mathcal{M}^{\gamma,
  p}_{T_{\star}} L^p (\mathbbm{T}^3)$.
\end{lemma}

\begin{proof}
  From the estimation on $R_{\varepsilon} (v_{\varepsilon})$ of
  Lemma~\ref{l:rem-estim}, together with the fact that $M_{\varepsilon,
  \delta} \rightarrow 0$ in probability (Lemma~\ref{l:rem-converg}) and the
  bounds on $\| v_{\varepsilon} \|_{\mathcal{M}_{T_{\star}}^{\frac{1}{4} +
  \frac{3 \kappa}{2}} L^{\infty}}$ and $\| v_{\varepsilon}^{\Box} \|_{C_T
  L^{\infty}}$ of Lemma~\ref{l:apriori-bounds} we see immediately that 
  $$
  \|
  R_{\varepsilon} (v_{\varepsilon}) \|_{\mathcal{M}^{\gamma, p}_{T_{\star}}
  L^p} \rightarrow 0 \quad \mbox{ in probability.}
  $$ 
\end{proof}

\begin{appendices}
\addtocontents{toc}{\protect\setcounter{tocdepth}{1}}

\section{Paracontrolled analysis and kernel
estimations}\label{a:basics}

In this section we first recall the the basic results of paracontrolled
calculus first introduced in~{\cite{gubinelli_paracontrolled_2012}}, without
proofs. For more details on Besov spaces, Littlewood--Paley theory, and Bony's
paraproduct the reader can refer to the monograph~{\cite{Bahouri2011}}. We
then proceed to give some results on the convolution of functions with known
singularity and the estimation of finite-chaos Gaussian trees. We refer to
Section~10 of~{\cite{hairer_theory_2014}} and to the nice pedagogic
exposition~{\cite{mourrat_construction_2016}} for further details.

\subsection{Notation}\label{s:notations}

Throughout the paper, we use the notation $a \lesssim b$ if there exists a
constant $c > 0$, independent of the variables under consideration, such that
$a \leqslant c \cdot b$. If we want to emphasize the dependence of $c$ on the
variable $x$, then we write $a (x) \lesssim_x b (x)$. If $f$ is a map from $A
\subset \mathbbm{R}$ to the linear space $Y$, then we write $f_{s, t} = f (t)
- f (s)$. For $f \in L^p (\mathbbm{T}^d)$ we write $\|f (x)\|^p_{L^p_x
(\mathbb{T}^3)} \assign \int_{\mathbbm{T}^3} |f (x) |^p \mathd x$.

Given a Banach space $X$ with norm $\| \cdummy \|_X$ and $T > 0$, we note $C_T
X = C ([0, T], X)$ for the space of continuous maps from $[0, T]$ to $X$,
equipped with the supremum norm $\lVert \cdummy \rVert_{C_T X}$, and we set $C
X = C (\mathbbm{R}_+, X)$. For $\alpha \in (0, 1)$ we also define
$C^{\alpha}_T X$ as the space of $\alpha$-H{\"o}lder continuous functions from
$[0, T]$ to $X$, endowed with the seminorm $\|f\|_{C^{\alpha}_T X} = \sup_{0
\leqslant s < t \leqslant T} \|f (t) - f (s)\|_X / |t - s|^{\alpha}$, and we
write $C^{\alpha}_{\tmop{loc}} X$ for the space of locally $\alpha$-H{\"o}lder
continuous functions from $\mathbbm{R}_+$ to $X$. For $\gamma > 0$, $p \in [1,
\infty)$, we define
\begin{equation}
  \begin{array}{lll}
    \mathcal{M}^{\gamma, p}_T X & = & \{ v : L^p ((0, T], X) : \| v
    \|_{\mathcal{M}^{\gamma, p}_T X} = \| t \mapsto t^{\gamma} v (t) \|_{L^p
    ((0, T], X)} < \infty \},\\
    \mathcal{M}^{\gamma}_T X & = & \{ v : C ((0, T], X) : \| v
    \|_{\mathcal{M}^{\gamma}_T X} = \| t \mapsto t^{\gamma} v (t) \|_{C_T X} <
    \infty \} .
  \end{array} \label{e:expl-space-def}
\end{equation}

The space of distributions on the torus is denoted by $\DD' (\mathbbm{T}^3)$
or $\DD'$. The Fourier transform is defined with the normalization
\[ \CF u (k) = \hat{u} (k) = \int_{\mathbbm{T}^d} e^{- \iota \langle k, x
   \rangle} u (x) \mathd x, \qquad k \in \mathbbm{Z}^3, \]
so that the inverse Fourier transform is given by $\CF^{- 1} v (x) = (2
\pi)^{- 1} \sum_k e^{\iota \langle k, x \rangle} v (k)$. Let $(\chi, \rho)$
denote a dyadic partition of unity such that $\tmop{supp} (\rho (2^{- i}
\cdummy)) \cap \tmop{supp} (\rho (2^{- j} \cdummy)) = \emptyset$ for $|i - j|
> 1$. The family of operators $(\Delta_j)_{j \ge - 1}$ will denote the
Littlewood-Paley projections associated to this partition of unity, that is
$\Delta_{- 1} u = \CF^{- 1} \left( \chi \CF u \right)$ and $\Delta_j = \CF^{-
1} \left( \rho (2^{- j} \nosymbol \cdummy) \CF u \right)$ for $j \ge 0$. Let
$S_j = \sum_{i < j} \Delta_i$, and $K_q$\quad be the kernel of $\Delta_q$ so
that
\[ \Delta_q f (\bar{x}) = \int_{\mathbbm{T}^3} K_{\bar{x}, q} (x) f (x) \mathd
   x. \]
For the precise definition and properties of the Littlewood-Paley
decomposition $f = \sum_{q \geqslant - 1} \Delta_q f$ in $\DD'
(\mathbbm{T}^3)$, see Chapter~2 of~{\cite{Bahouri2011}}. The H{\"o}lder-Besov
space $B^{\alpha}_{p, q} (\mathbbm{T}^3, \mathbbm{R})$ for $\alpha \in
\mathbbm{R}$ ,$p, q \in [1, \infty]$ with $B^{\alpha}_{p, q} (\mathbbm{T}^3,
\mathbbm{R}) = : \CC^a$ is and equipped with the norm
\[ \begin{array}{lll}
     \lVert f \rVert_{\alpha} = \lVert f \rVert_{B^{\alpha}_{\infty, \infty}}
     & = & \sup_{i \geqslant - 1} (2^{i \alpha} \| \Delta_i f \|_{L^{\infty}
     (\mathbbm{T}^3)}),\\
     \lVert f \rVert_{B^{\alpha}_{p, q}} & = & \| 2^{i \alpha} \| \Delta_i f
     \|_{L^p (\mathbbm{T}^3)} \|_{\ell^q} .
   \end{array} \]
If $f$ is in $\CC^{\alpha - \varepsilon}$ for all $\varepsilon > 0$, then we
write $f \in \CC^{\alpha -}$. For $\alpha \in (0, 2)$, we define the space
$\LL_T^{\alpha} = C^{\alpha / 2}_T L^{\infty} \cap C_T \CC^{\alpha}$, equipped
with the norm
\[ \| f \|_{\LL_T^{\alpha}} = \max \left\{ \| f \|_{C^{\alpha / 2}_T
   L^{\infty}}, \| f \|_{C_T \CC^{\alpha}} \right\} . \]
The notation is chosen to be reminiscent of
\[ \LL \assign \partial_t -, \]
by which we will always denote the heat operator with periodic boundary
conditions on $\mathbbm{T}^d$. We also write $\LL^{\alpha} = C^{\alpha /
2}_{\tmop{loc}} L^{\infty} \cap C \CC^{\alpha}$. When working with irregular
initial conditions, we will need to consider explosive spaces of parabolic
type. For $\gamma \geqslant 0$, $\alpha \in (0, 1)$, and $T > 0$ we define the
norm
\[ \| f \|_{\LL^{\gamma, \alpha}_T} = \max \left\{ \| t \mapsto t^{\gamma} f
   (t) \|_{C^{\alpha / 2}_T L^{\infty}}, \| f \|_{\mathcal{M}^{\gamma}_T
   \CC^{\alpha}} \right\}, \]
and the space $\LL^{\gamma, \alpha}_T = \left\{ f : [0, T] \rightarrow
\mathbbm{R}: \| f \|_{\LL^{\gamma, \alpha}_T} < \infty \right\}$. In
particular, we have $\LL^{0, \alpha}_T = \LL^{\alpha}_T$. We introduce the
linear operator $I : C \left( \mathbbm{R}_+, \DD' (\mathbbm{T}) \right)
\rightarrow C \left( \mathbbm{R}_+, \DD' (\mathbbm{T}) \right)$ given by
\[ I f (t) = \int_0^t P_{t - s} f (s) \mathd s, \]
where $(P_t)_{t \geqslant 0}$ is the heat semigroup with kernel $P_t (x) =
\frac{1}{(4 \mathpi t)^{3 / 2}} e^{- \frac{| x |^2}{4 t}} \mathbbm{I}_{t
\geqslant 0}$.

Paraproducts are bilinear operations introduced by Bony~{\cite{Bony1981}} in
order to linearize a class of non-linear PDE problems. They appear naturally
in the analysis of the product of two Besov distributions. In terms of
Littlewood--Paley blocks, the product $fg$ of two distributions $f$ and $g$
can be formally decomposed as
\[ fg = f \prec g + f \succ g + f \circ g, \]
where
\[ f \prec g = g \succ f \assign \sum_{j \geqslant - 1} \sum_{i = - 1}^{j - 2}
   \Delta_i f \Delta_j g \quad \text{and} \quad f \circ g \assign \sum_{|i -
   j| \leqslant 1} \Delta_i f \Delta_j g. \]
This decomposition behaves nicely with respect to Littlewood--Paley theory. We
call $f \prec g$ and $f \succ g$ \tmtextit{paraproducts}, and $f \circ g$ the
\tmtextit{resonant} term. We use the notation $f \preccurlyeq g = f \prec g +
f \circ g$. The basic result about these bilinear operations is given by the
following estimates, essentially due to Bony~{\cite{Bony1981}} and
Meyer~{\cite{Meyer1981}}.

When dealing with paraproducts in the context of parabolic equations it would
be natural to introduce parabolic Besov spaces and related paraproducts. But
to keep a simpler setting, we choose to work with space--time distributions
belonging to the scale of spaces $\left( C_T \CC^{\alpha} \right)_{\alpha \in
\mathbbm{R}}$ for some $T > 0$. To do so efficiently, we will use a modified
paraproduct which introduces some smoothing in the time variable that is tuned
to the parabolic scaling. Let therefore $\varphi \in C^{\infty} (\mathbbm{R},
\mathbbm{R}_+)$ be nonnegative with compact support contained in
$\mathbbm{R}_+$ and with total mass $1$, and define for all $i \geqslant - 1$
the operator
\[ Q_i : C \CC^{\beta} \rightarrow C \CC^{\beta}, \qquad Q_i f (t) =
   \int_0^{\infty} 2^{- 2 i} \varphi (2^{2 i} (t - s)) f (s) \mathd s. \]
We will often apply $Q_i$ and other operators on $C \CC^{\beta}$ to functions
$f \in C_T \CC^{\beta}$ which we then simply extend from $[0, T]$ to
$\mathbbm{R}_+$ by considering $f (\cdummy \wedge T)$. With the help of $Q_i$,
we define a modified paraproduct
\[ f \precprec g \assign \sum_i (Q_i S_{i - 1} f) \Delta_i g \]
for $f, g \in C \left( \mathbbm{R}_+, \CD' (\mathbbm{T}) \right)$. We define
the commutators $\tmop{com}_1, \overline{\tmop{com}}_1, \tmop{com}_2,
\tmop{com}_3$ in Lemma~\ref{lem:all-commutators}. We write $\check{P} (t, x) =
\frac{1}{(4 \mathpi t)^{3 / 2}} e^{- \frac{| x |^2}{4 t}} e^{- t} 
\mathbbm{1}_{t \geqslant 0}$ for a modified heat kernel which has the same
bounds of the usual heat kernel $P (t, x)$. Let $Y_{\varepsilon}$ as
in~(\ref{e:Y-stat-eq}) and recall that $C_{\varepsilon}$ is the covariance of
$Y_{\varepsilon}$ , i.e. $C_{\varepsilon} (t, x) =\mathbbm{E} (Y_{\varepsilon}
(t, x) Y_{\varepsilon} (0, 0))$. We will sometimes write $Y_{\varepsilon,
\zeta} \assign Y_{\varepsilon} (t, x)$ for $\zeta = (t, x) \in \mathbbm{R}
\times \mathbbm{T}^3$. Let $\sigma_{\varepsilon}^2 = \varepsilon \mathbbm{E} [
(Y_{\varepsilon} (0, 0))^2] = \varepsilon C_{\varepsilon} (0, 0)$.

\subsection{Basic paracontrolled calculus
results\label{sec:schauder}{\hspace{3em}} \label{sec:bony} }

First, let us recall some interpolation results on the parabolic time-weighted
spaces $\LL^{\gamma, \alpha}_T$:

\begin{lemma}
  \label{l:L-interpol}For all $\alpha \in (0, 2)$, $\gamma \in [0, 1)$,
  $\varepsilon \in [0, \alpha \wedge 2 \gamma)$, $T > 0$ and $f \in
  \LL^{\gamma, \alpha}_T$ with $f (0) = 0$ we have
  \begin{equation}
    \label{eq:schauder time reg} \| f \|_{\LL^{\gamma - \varepsilon / 2,
    \alpha - \varepsilon}_T} \lesssim \| f \|_{\LL^{\gamma, \alpha}_T} .
  \end{equation}
  Let $\alpha \in (0, 2)$, $\gamma \in (0, 1),$ $T > 0$, and let $f \in
  \LL^{\alpha}_T$. Then for all $\delta \in (0, \alpha]$ we have
  \begin{equation}
    \begin{array}{lll}
      \| f \|_{\LL^{\delta}_T} & \lesssim & \| f (0) \|_{\CC^{\delta}} +
      T^{(\alpha - \delta) / 2} \| f \|_{\LL^{\alpha}_T},\\
      \| f \|_{\LL^{\gamma, \delta}_T} & \lesssim & T^{(\alpha - \delta) / 2}
      \| f \|_{\LL^{\gamma, \alpha}_T} .
    \end{array} \label{eq:schauder-gain-time}
  \end{equation}
\end{lemma}

\paragraph{Schauder estimates}

\begin{lemma}
  \label{lemma:schauder}Let $\alpha \in (0, 2)$ and $\gamma \in [0, 1)$. Then
  \begin{equation}
    \label{eq:schauder-heat} \|I f\|_{\LL^{\gamma, \alpha}_t} \lesssim \| f
    \|_{\mathcal{M}^{\gamma}_t \CC^{\alpha - 2}},
  \end{equation}
  for all $t > 0$. If further $\beta \geqslant - \alpha$, then
  \begin{equation}
    \label{eq:schauder initial contribution} \| s \mapsto P_s u_0
    \|_{\LL^{(\beta + \alpha) / 2, \alpha}_t} \lesssim \| u_0 \|_{\CC^{-
    \beta}} .
  \end{equation}
  For all $\alpha \in \mathbbm{R}$, $\gamma \in [0, 1)$, and $t > 0$ we have
  \begin{equation}
    \label{eq:schauder without time hoelder} \|I f\|_{\mathcal{M}^{\gamma}_t
    \CC^{\alpha}} \lesssim \| f \|_{\mathcal{M}^{\gamma}_t \CC^{\alpha - 2}} .
  \end{equation}
\end{lemma}

Proofs can be found e.g. in~{\cite{gubinelli_kpz_2017}}. We need also some
well known bounds for the solutions of the heat equation with sources in
space--time Lebesgue spaces.

\begin{lemma}
  \label{l:lp-integ}Let $\beta \in \mathbbm{R}$ and $f \in L^p_T B^{\beta}_{p,
  \infty}$, then for every $\kappa \in [0, 1]$ we have $I f \in C_T^{\kappa /
  q} \mathcal{\CC}^{\beta + 2 (1 - \kappa) - (2 - 2 \kappa + d) / p}$ with
  \[ \|I f\|_{C_T^{\kappa / q}  \mathcal{\CC}^{\beta + 2 (1 - \kappa) - (2 - 2
     \kappa + d) / p}} \lesssim_T \|f\|_{L^p_T B^{\beta}_{p, \infty}}, \]
  with $\frac{1}{q} + \frac{1}{p} = 1$. Moreover, for every $\gamma < \gamma'
  < 1 - 1 / p$ and every $0 < \alpha < (2 - 5 / p + \beta) \wedge 2$ \ we have
  \[ \| I f \|_{\LL^{\gamma', \alpha}_T} \lesssim_T \| f
     \|_{\mathcal{M}^{\gamma, p}_T B^{\beta}_{p, \infty}} . \]
\end{lemma}

\begin{proof}
  We only show the second inequality as the first one is easier and obtained
  with similar techniques. Let $u = I f$, we have
  \begin{eqnarray*}
    t^{\gamma} \| \Delta_i u (t)\|_{L^{\infty}} & \leqslant & t^{1 / q} 2^{di
    / p} \left[ \int_0^1 s^{- \gamma q} e^{- cq 2^{2 i} t (1 - s)} \mathd s
    \right]^{1 / q} \left[ \int_0^t s^{\gamma p} \| \Delta_i f (s)\|_{L^p}^p
    \mathd s \right]^{1 / p}\\
    & \lesssim_{\gamma, q} & 2^{id / p} 2^{- 2 i / q}  \left[ \int_0^t
    s^{\gamma p} \| \Delta_i f (s)\|_{L^p}^p \mathd s \right]^{1 / p}
  \end{eqnarray*}
  which allows us to bound $\| I f \|_{\mathcal{M}^{\gamma}_T \CC^{\alpha}}$.
  In order to estimate $\| t \mapsto t^{\gamma'} I f \|_{C^{\alpha / 2}_T
  L^{\infty}}$ we write
  \begin{eqnarray*}
    \|t^{\gamma'} \Delta_i u (t) - s^{\gamma'} \Delta_i u (s)\|_{L^{\infty}} &
    \lesssim & \int_s^t v^{\gamma' - 1} \| \Delta_i u (v) \|_{L^{\infty}}
    \mathd v + | t - s | 2^{i (d + 2) / p} \| \Delta_i f
    \|_{\mathcal{M}^{\gamma, p}_T L^p (\mathbbm{T}^3)}\\
    &  & + \left\| \int_s^t v^{\gamma'} \Delta_i f (v) \mathd v
    \right\|_{L^{\infty}}
  \end{eqnarray*}
  We can estimate the first term as
  \[ \int_s^t v^{\gamma' - 1} \| \Delta_i u (v) \|_{L^{\infty}} \mathd v
     \lesssim 2^{i (d + 2) / p} \| \Delta_i f \|_{\mathcal{M}^{\gamma, p}_T
     L^p (\mathbbm{T}^3)} \int_s^t v^{\gamma' - \gamma - 1} \mathd v. \]
  For the third term we have
  \begin{eqnarray*}
    \left\| \int_s^t v^{\gamma} \Delta_i f (v) \mathd v \right\|_{L^{\infty}}
    & \lesssim & \left[ \int_s^t \mathd v \right]^{1 / q} \left[ \int_s^t
    v^{\gamma p} \| \Delta_i f (s)\|_{L^{\infty}}^p \mathd v \right]^{1 / p}\\
    & \lesssim & 2^{id / p}  |t - s|^{1 / q}  \| \Delta_i f
    \|_{\mathcal{M}^{\gamma, p}_T L^p (\mathbbm{T}^3)}
  \end{eqnarray*}
  We obtain then if $2^{2 i} | t - s | \leqslant 1$
  \[ \|t^{\gamma'} \Delta_i u (t) - s^{\gamma'} \Delta_i u (s)\|_{L^{\infty}}
     \lesssim 2^{id / p} | t - s |^{1 / q} \| \Delta_i f
     \|_{\mathcal{M}^{\gamma, p}_T L^p (\mathbbm{T}^3)} \]
  and if $2^{2 i} | t - s | > 1$ we just use the trivial estimate
  \[ \|t^{\gamma'} \Delta_i u (t) - s^{\gamma'} \Delta_i u (s)\|_{L^{\infty}}
     \lesssim 2^{id / p} 2^{- 2 i / q}  \| \Delta_i f \|_{\mathcal{M}^{\gamma,
     p}_T L^p (\mathbbm{T}^3)} \lesssim 2^{id / p} | t - s |^{1 / q} \|
     \Delta_i f \|_{\mathcal{M}^{\gamma, p}_T L^p (\mathbbm{T}^3)} . \]
  Therefore, for every $\kappa \in [0, 1]$:
  \[ \|t^{\gamma \prime} \Delta_i u (t) - s^{\gamma \prime} \Delta_i u
     (s)\|_{L^{\infty}} \lesssim 2^{(\frac{d + 2}{p} - 2) i} 2^{2 \kappa i /
     q}  |t - s|^{\kappa / q}  \| \Delta_i f \|_{\mathcal{M}^{\gamma, p}_T L^p
     (\mathbbm{T}^3)} . \]
  Choosing $\kappa / q = \alpha / 2$ we obtain the desired estimate.
\end{proof}

\paragraph{Estimates on Bony's paraproducts and commutators}

\begin{lemma}
  \label{thm:paraproduct} For any $\beta \in \mathbb{R}$ we have
  \begin{equation}
    \label{eq:para-1} \|f \prec g\|_{\CC^{\beta}} \lesssim_{\beta}
    \|f\|_{L^{\infty}} \|g\|_{\CC^{\beta}},
  \end{equation}
  and for $\alpha < 0$ furthermore
  \begin{equation}
    \label{eq:para-2} \|f \prec g\|_{\CC^{\alpha + \beta}} \lesssim_{\alpha,
    \beta} \|f\|_{\CC^{\alpha}} \|g\|_{\CC^{\beta}} .
  \end{equation}
  For $\alpha + \beta > 0$ we have
  \begin{equation}
    \label{eq:para-3} \|f \circ g\|_{\CC^{\alpha + \beta}} \lesssim_{\alpha,
    \beta} \|f\|_{\CC^{\alpha}} \|g\|_{\CC^{\beta}} .
  \end{equation}
\end{lemma}

A natural corollary is that the product $fg$ of two elements $f \in
\CC^{\alpha}$ and $g \in \CC^{\beta}$ is well defined as soon as $\alpha +
\beta > 0$, and that it belongs to $\CC^{\gamma}$, where $\gamma = \min
\{\alpha, \beta, \alpha + \beta\}$. We will also need a commutator estimation:

\begin{lemma}
  \label{lem:bony commutator}Let $\alpha > 0$, $\beta \in \mathbbm{R}$, and
  let $f, g \in \CC^{\alpha}$, and $h \in \CC^{\beta}$. Then
  \[ \| f \prec (g \prec h) - (f g) \prec h \|_{\CC^{\alpha + \beta}} \lesssim
     \| f \|_{\CC^{\alpha}} \| g \|_{\CC^{\alpha}} \| h \|_{\CC^{\beta}} . \]
\end{lemma}

We collect in the following lemma various estimates for the modified
paraproduct $f \precprec g$, proofs are again in~{\cite{gubinelli_kpz_2017}}.

\begin{lemma}
  {\tmdummy}
  
  \begin{enumeratealpha}
    \item For any $\beta \in \mathbb{R}$ and $\gamma \in [0, 1)$ we have
    \begin{equation}
      \label{eq:mod para-1} t^{\gamma} \|f \precprec g (t) \|_{\CC^{\beta}}
      \lesssim \|f\|_{\mathcal{M}^{\gamma}_t L^{\infty}} \|g (t)
      \|_{\CC^{\beta}},
    \end{equation}
    for all $t > 0$, and for $\alpha < 0$ furthermore
    \begin{equation}
      \label{eq:mod para-2} t^{\gamma} \|f \precprec g (t) \|_{\CC^{\alpha +
      \beta}} \lesssim \|f\|_{\mathcal{M}^{\gamma}_t \CC^{\alpha}} \|g (t)
      \|_{\CC^{\beta}} .
    \end{equation}
    \item Let $\alpha, \delta \in (0, 2)$, $\gamma \in [0, 1)$, $T > 0$, and
    let $f \in \LL_T^{\gamma, \delta}$, $g \in C_T \CC^{\alpha}$, and $\LL g
    \in C_T \CC^{\alpha - 2}$. Then
    \begin{equation}
      \| f \precprec g \|_{\LL^{\gamma, \alpha}_T} \lesssim \| f
      \|_{\LL^{\gamma, \delta}_T} \left( \| g \|_{C_T \CC^{\alpha}} + \left\|
      \LL g \right\|_{C_T \CC^{\alpha - 2}} \right) .
      \label{eq:mod-para-parabolic-est}
    \end{equation}
  \end{enumeratealpha}
\end{lemma}

We introduce various commutators which allow to control non-linear functions
of paraproducts and also the interaction of the paraproducts with the heat
kernel.

\begin{lemma}
  \label{lem:all-commutators}{\tmdummy}
  
  \begin{enumeratealpha}
    \item For $\alpha, \beta, \gamma \in \mathbbm{R}$ such that $\alpha +
    \beta + \gamma > 0$ and $\alpha \in (0, 1)$ there exists bounded trilinear
    maps
    \[ \tmop{com}_1, \overline{\tmop{com}}_1 : \CC^{\alpha} \times
       \CC^{\beta} \times \CC^{\gamma} \rightarrow \CC^{\alpha + \beta +
       \gamma}, \]
    such that for smooth $f, g, h$ they satisfy
    \begin{equation}
      \tmop{com}_1 (f, g, h) = (f \prec g) \circ h - f (g \circ h) .
      \label{e:firstcomm-nonsmooth}
    \end{equation}
    \begin{equation}
      \overline{\tmop{com}}_1 (f, g, h) = (f \precprec g) \circ h - f (g \circ
      h) . \label{e:firstcomm}
    \end{equation}
    \item Let $\alpha \in (0, 2)$, $\beta \in \mathbbm{R}$, and $\gamma \in
    [0, 1)$. Then the bilinear maps
    \begin{equation}
      \tmop{com}_2 (f, g) : = f \prec g - f \precprec g. \label{e:secondcomm}
    \end{equation}
    \begin{equation}
      \tmop{com}_3 (f, g) : = \left[ \LL, f \precprec \right] g \assign \LL (f
      \precprec g) - f \precprec \LL g. \label{e:thirdcomm}
    \end{equation}
    have the bounds
    \begin{equation}
      t^{\gamma} \| \tmop{com}_2 (f, g) (t) \|_{\alpha + \beta} \lesssim \| f
      \|_{\LL^{\gamma, \alpha}_t} \| g (t) \|_{\CC^{\beta}}, \qquad t > 0.
      \label{eq:bound-second-comm}
    \end{equation}
    as well as
    \begin{equation}
      t^{\gamma} \| \tmop{com}_3 (f, g) (t) \|_{\alpha + \beta - 2} \lesssim
      \| f \|_{\LL^{\gamma, \alpha}_t} \| g (t) \|_{\CC^{\beta}}, \qquad t >
      0. \label{eq:bound-third-comm}
    \end{equation}
  \end{enumeratealpha}
\end{lemma}

Proofs can be found in~{\cite{gubinelli_kpz_2017}}.

\subsection{\label{a:ker-est}Estimation of finite-chaos diagrams}

In this section we recall a few well-known results on the estimation of finite
chaos diagrams. For additional details see Chapter~10
of~{\cite{hairer_theory_2014}} and~{\cite{mourrat_construction_2016}}. First
of all we need to characterize the local behaviour of the heat kernel $P_t
(x)$ and of the covariance $C_{\varepsilon} (t, x)$ of the Gaussian field
$Y_{\varepsilon}$.

\begin{remark}
  We sometimes use a slightly different version of the heat kernel, namely
  \[ \check{P}_t (x) = \frac{1}{(4 \mathpi t)^{3 / 2}} e^{- \frac{| x |^2}{4
     t}} e^{- t}  \mathbbm{1}_{t \geqslant 0} \]
  in order to have that $X (t, x) = \int_{- \infty}^t \int_{\mathbbm{T}^3}
  P_{t - s} (x - y) v (s, y) \mathd s \mathd y$ is the stationary solution to
  $\LL X = - X + v$. However, the bounds on the covariance $C_{\varepsilon}$
  remain trivially valid in this setting.
\end{remark}

\begin{lemma}
  \label{l:heat-est} The heat kernel $P (\zeta) \assign P (t, x) = \frac{1}{(4
  \mathpi t)^{3 / 2}} e^{- \frac{| x |^2}{4 t}} \mathbbm{I}_{t \geqslant 0}$
  has the bound
  \[ | P (\zeta) | \lesssim (| t |^{1 / 2} + | x |)^{- 3} . \]
  Let $k \in \mathbbm{N}^4$ a multi-index with $| k | = 2 k_1 + k_2 + \cdots +
  k_4$. Then for every multi-index $| k | \leqslant 2$ we have:
  \[ | D^k P_t (x) | \lesssim (| t |^{1 / 2} + | x |)^{- 3 - | k |} . \]
\end{lemma}

\begin{proof}
  
  \[ | P_t (x) | (| t |^{1 / 2} + | x |)^3 \lesssim [1 + (| x | | t |^{- 1 /
     2})^3] e^{- \frac{| x |^2}{4 | t |}} = (1 + | \alpha |^3) e^{- \frac{|
     \alpha |}{4}} < + \infty \]
  In the same way we prove that $| \partial_t P_t (x) | \lesssim (| t |^{1 /
  2} + | x |)^5$, $| \partial_{x_i} P_t (x) | \lesssim (| t |^{1 / 2} + | x
  |)^4$ and $| \partial_{x_i} \partial_{x_j} P_t (x) | \lesssim (| t |^{1 / 2}
  + | x |)^5$.
\end{proof}

We recall a special case of Lemma~10.14 of {\cite{hairer_theory_2014}}, which
is enough for our purpose. We use the notation $\interleave \zeta \interleave
\assign (| t |^{1 / 2} + | x |)$ for $\zeta = (t, x) \in \mathbbm{R} \times
\mathbbm{T}^3$.

\begin{lemma}
  \label{l:hom-conv}Let $f, g : \mathbbm{R} \times \mathbbm{T}^3 \setminus
  \{0\} \rightarrow \mathbbm{R}$ smooth, integrable at infinity and such that
  $|f (\zeta) | \lesssim \interleave \zeta \interleave^{\alpha}$ and $|g
  (\zeta) | \lesssim \interleave \zeta \interleave^{\beta}$ in a ball \ $B =
  \{\zeta \in \mathbbm{R} \times \mathbbm{T}^3 : \interleave \zeta \interleave
  < 1, \zeta \neq 0\}$. Then if $\alpha, \beta \in (- 5, 0)$ and $\alpha +
  \beta + 5 < 0$ we have
  \[ |f \ast g (\zeta) | \lesssim \interleave \zeta \interleave^{\alpha +
     \beta + 5} \]
  in a ball centered in the origin.
  
  Moreover, if $\alpha, \beta \in (- 5, 0)$ and $0 < \alpha + \beta + 5 < 1$
  and for every multi-index $| k | \leqslant 2$ we have $| \mathD^k f (\zeta)
  | \lesssim \interleave \zeta \interleave^{\alpha - | k |}$ and $| \mathD^k g
  (\zeta) | \lesssim \interleave \zeta \interleave^{\beta - | k |}$, then
  \[ |f \ast g (\zeta) - f \ast g (0) | \lesssim \interleave \zeta
     \interleave^{\alpha + \beta + 5} \]
  in a ball centered in the origin.
\end{lemma}

\begin{remark}
  \label{r:rescaling-properties-cov}The covariance $C_{\varepsilon}$ of
  $Y_{\varepsilon}$ can be written as $C_{\varepsilon} = \check{P} \ast
  \check{C}_{\varepsilon} \ast \check{P}_{}$ with $\check{C}_{\varepsilon} (t,
  x) \assign \mathbbm{E} (\eta_{\varepsilon} (t, x) \eta_{\varepsilon} (0,
  0))$. Recall from the introduction that $\check{C}_{\varepsilon} (t, x) =
  \varepsilon^{- 5} \tilde{C}^{\varepsilon} (\varepsilon^{- 2} t,
  \varepsilon^{- 1} x)$ where $\tilde{C}^{\varepsilon}$ is the covariance of
  the Gaussian process $\eta$ defined on $\mathbbm{R} \times (\mathbbm{T}/
  \varepsilon)^3$, and $\tilde{C}^{\varepsilon} (t - s, x - y) = \Sigma (t -
  s, x - y)$ if $\tmop{dist} (x, y) \leqslant 1$ and $0$ otherwise (so that
  the family of functions $\tilde{C}^{\varepsilon}$ is bounded uniformly on
  $\varepsilon$ by a $C^{\infty}_c$ function). \ Then there exists a family of
  functions $C_Y^{\varepsilon}$ defined on $\mathbbm{R} \times (\mathbbm{T}/
  \varepsilon)^3$ such that $C_{\varepsilon} (t, x) = \varepsilon^{- 1}
  C_Y^{\varepsilon} (\varepsilon^{- 2} t, \varepsilon^{- 1} x)$ and
  $C_Y^{\varepsilon} (t, x) = [\check{P} \ast \check{C}_{\varepsilon} \ast
  \check{P}_{}] (t, x)$.
\end{remark}

\begin{lemma}
  \label{l:cov-est-eps}The covariance $C_{\varepsilon}$
  \label{l:cov-est-tx}has the bound, for every multi-index $| k | \leqslant
  2$:
  \[ | D^k C_{\varepsilon} (t, x) | \lesssim (| t |^{1 / 2} + | x |)^{- 1 - |
     k |} . \]
  Moreover, we have
  \[ \varepsilon^{| k |_{} + 1} | D^k C_{\varepsilon} (t, x) | \lesssim 1 \]
\end{lemma}

\begin{proof}
  The first bound is obtained directly from Lemma~\ref{l:heat-est} and
  Lemma~\ref{l:hom-conv}. Indeed, since by hypothesis
  $\tilde{C}^{\varepsilon}$ has compact support, it is easy to see that $|
  \check{C}_{\varepsilon} (t, x) | \lesssim (| t |^{1 / 2} + | x |)^{- 5}$.
  The second bound is obtained by a simple change of variables in the
  convolution defining $C_{\varepsilon}$. 
\end{proof}

\begin{lemma}
  \label{l:PC-integrals}We have $\int_{s, x} P_s (x) [C_{\varepsilon} (s,
  x)]^2 \lesssim | \log \varepsilon |$ and for every $n \geqslant 3$
  $\varepsilon^{n - 2} \int_{s, x} P_s (x) [C_{\varepsilon} (s, x)]^n \lesssim
  1$.
\end{lemma}

\begin{proof}
  From the fact that $P_{\varepsilon^2 s} (\varepsilon x) = \varepsilon^{- 3}
  P_s (x)$ together with Remark~\ref{r:rescaling-properties-cov} we obtain
  \[ \int_{\mathbbm{R} \times \mathbbm{T}^3} P_s (x) [C_{\varepsilon} (s,
     x)]^2 \mathd s \mathd x \lesssim \int_{B (0, \varepsilon^{- 1})} P_s (x)
     [C_Y^{\varepsilon} (s, x)]^2 \mathd s \mathd x \lesssim | \log
     \varepsilon | \]
  with $B (0, R) = \{\zeta \in \mathbbm{R}^4 : \interleave \zeta \interleave <
  R, \zeta \neq 0\}$ a ``parabolic'' ball centered in the origin. The second
  estimation is obtained in the same way.
\end{proof}

\begin{lemma}
  \label{l:integral-firstorder-est} For $m \in (0, 3)$, $n \in (3, 5)$, define
  for $\zeta, \zeta' \in \mathbbm{R} \times \mathbbm{T}^3$
  \[ I_m \assign \int | C_{\varepsilon} (\zeta - \zeta') |^m  | \mu_{q, \zeta}
     \mu_{q, \zeta'} |, \qquad \tilde{I}_n \assign \int | C_{\varepsilon}
     (\zeta - \zeta') |^n  | \tilde{\mu}_{q, \zeta}  \tilde{\mu}_{q, \zeta'} |
  \]
  with $\mu_{q, \zeta} \assign K_{q, \bar{x}} (y) \delta (t - s) \mathd \zeta$
  and $\tilde{\mu}_{q, \zeta} \assign \left[ \int_x K_{q, \bar{x}} (x) P_{t -
  s} (x - y) \right] d \zeta$ for $\zeta = (s, y)$. Then
  \[ I_m \lesssim 2^{mq} \qquad \text{and} \qquad \tilde{I}_n \lesssim 2^{(n -
     4) q} . \]
\end{lemma}

\begin{proof}
  The estimation of $I_m$ is easily obtained by Lemma~\ref{l:cov-est-tx} and a
  change of variables. For $\tilde{I}_n$ observe that for every $q > 0$
  \[ \tilde{\mu}_{q, \zeta} = [\int_x K_{q, \bar{x}} (x) (P_{t - s} (x - y) -
     P_{t - s} (\bar{x} - y))] d \zeta \]
  and then we can apply Lemma~\ref{l:hom-conv} to obtain the result.
\end{proof}

\begin{lemma}
  \label{l:covar-difference}We have for every $\sigma \in [0, 1]$
  \[ \sup_{x \in \mathbbm{T}^3} | C_{\varepsilon} (t, x) - C_{\varepsilon} (0,
     x) | \lesssim \varepsilon^{- 1 - 2 \sigma} | t |^{\sigma} \]
\end{lemma}

\begin{proof}
  It is easy to obtain by interpolation knowing that $| \partial_t
  C_{\varepsilon} (t, x) | \lesssim \varepsilon^{- 3}$ from
  Lemma~\ref{l:cov-est-eps}.
\end{proof}

\begin{lemma}
  \label{l:res-product-1}We have for every $\alpha < 3$
  \[ \sum_{i \sim j} \left| \int K_i (x - y) P_t (y) \mathd y \right|  \int
     \frac{| K_j (x - y) |}{(| y | + t^{1 / 2})^{\alpha}} \mathd y \lesssim
     \frac{1}{(| x | + t^{1 / 2})^{3 + \alpha}} \]
\end{lemma}

\begin{proof}
  We will show that
  \begin{equation}
    \left| \int K_i (x - y) P_t (y) \mathd y \right| \lesssim 2^{- i} (| x | +
    t^{1 / 2} + 2^{- i})^{- 4}, \label{e:est-KP}
  \end{equation}
  and that
  \begin{equation}
    \int \frac{| K_i (x - y) |}{(| y | + t^{1 / 2})^{\alpha}} \mathd y
    \lesssim (| x | + t^{1 / 2} + 2^{- i})^{- \alpha}, \label{e:est-bicorr}
  \end{equation}
  from which we deduce that
  \[ \sum_{i \sim j} \left| \int K_i (x - y) P_t (y) \mathd y \right|  \int
     \frac{| K_j (x - y) |}{(| y | + t^{1 / 2})^{\alpha}} \mathd y \lesssim
     \sum_i \frac{2^{- i}}{(| x | + t^{1 / 2} + 2^{- i})^{4 + \alpha}} . \]
  Bounding the sum over $i$ with an integral, we conclude
  \[ \int_0^1 \frac{\mathd \lambda}{\lambda} \frac{\lambda}{(| x | + t^{1 / 2}
     + \lambda)^{4 + \alpha}} = \frac{1}{(| x | + t^{1 / 2})^{3 + \alpha}}
     \int_0^{1 / (| x | + t^{1 / 2})} \frac{\mathd \lambda}{(1 + \lambda)^{4 +
     \alpha}} \lesssim \frac{1}{(| x | + t^{1 / 2})^{3 + \alpha}} . \]
  Let us show~(\ref{e:est-KP}). We want to estimate
  \begin{eqnarray*}
    I & = & \int K_i (x - y) P_t (y) \mathd y = \int K_i (x - y) [P_t (y) -
    P_t (x)] \mathd y\\
    & = & \int_0^1 \mathd \tau \int K_i (x - y) [P_t' (x + \tau (y - x)) (y -
    x)] \mathd y\\
    | I | & \lesssim & \int_0^1 \mathd \tau \int | (y - x) K_i (x - y) | |
    P_t' (x + \tau (y - x)) | \mathd y \lesssim 2^{- i} \int_0^1 \mathd \tau
    \int | y K_1 (y) | | P_t' (x + \tau 2^{- i} y) | \mathd y\\
    & \lesssim & 2^{- i} \int_0^1 \mathd \tau \int | y K_1 (y) | \frac{e^{- c
    | x + \tau 2^{- i} y |^2 / t}}{t^2} \mathd y
  \end{eqnarray*}
  where
  \[ | P'_t (z) | = \left| C \frac{e^{- | z |^2 / t}}{t^{4 / 2}} \frac{z}{t^{1
     / 2}} \right| \leqslant C \frac{e^{- c | z |^2 / t}}{t^2} . \]
  When $t^{1 / 2} \geqslant 2^{- i}, | x |$ we have
  \[ | I | \lesssim 2^{- i} t^{- 2} \lesssim 2^{- i} (| x | + t^{1 / 2} + 2^{-
     i})^{- 4} . \]
  When \ $2^{- i} \geqslant t^{1 / 2}, | x |$ we estimate simply
  \[ | I | \lesssim \int | K_i (x - y) | P_t (y) \mathd y \lesssim 2^{3 i}
     \lesssim 2^{- i} (| x | + t^{1 / 2} + 2^{- i})^{- 4} . \]
  When $| x | \geqslant 2^{- i}, t^{1 / 2}$ we have instead that either $| x |
  \geqslant 2 \tau 2^{- i} | y |$ or $| x | < 2 \tau 2^{- i} | y |$. In the
  first region $| x + \tau 2^{- i} y | \geqslant c | x |$ so
  \[ | I | \lesssim 2^{- i} \int_0^1 \mathd \tau \int | y K_1 (y) | \frac{e^{-
     c' | x |^2 / t}}{t^2} \mathd y \lesssim 2^{- i} \frac{e^{- c' | x |^2 /
     t}}{t^2} \lesssim 2^{- i} | x |^{- 4} \lesssim 2^{- i} (| x | + t^{1 / 2}
     + 2^{- i})^{- 4} . \]
  while in the second region $| y | \geqslant 2^i | x | / (2 \tau)$, then $| y
  K_1 (y) | \leqslant | y K_1 (y) |^{1 / 2} f (2^i | x | / (2 \tau))$ where
  $f$ is a rapidly decreasing function and in this region
  \begin{eqnarray*}
    | I | & \lesssim & 2^{- i} \int_0^1 \mathd \tau f (2^i | x | / (2 \tau))
    \int \frac{e^{- c | x + \tau 2^{- i} y |^2 / t}}{t^2} \mathd y\\
    & \lesssim & 2^{- i} \int_0^1 \mathd \tau f (2^i | x | / (2 \tau)) \int
    \frac{e^{- c' | \tau 2^{- i} y |^2 / t}}{t^{3 / 2} | \tau 2^{- i} y |}
    \mathd y \lesssim 2^{- i} \int_0^1 \mathd \tau f (2^i | x | / (2 \tau))
    \frac{2^{3 i}}{\tau^3 | x |} \int \frac{e^{- c' | y |^2 / t}}{t^{3 / 2}}
    \mathd y\\
    & \lesssim & \frac{2^{- i}}{| x |^4} \int_0^1 \mathd \tau f (2^i | x | /
    (2 \tau)) \frac{2^{3 i} | x |^3}{\tau^3} \lesssim \frac{2^{- i}}{| x |^4}
    \lesssim 2^{- i} (| x | + t^{1 / 2} + 2^{- i})^{- 4} .
  \end{eqnarray*}
  So we conclude that~(\ref{e:est-KP}) holds. Let us turn
  to~(\ref{e:est-bicorr}). When $t^{1 / 2} \geqslant 2^{- i}, | x |$ we have
  \[ \int \frac{| K_i (x - y) |}{(| y | + t^{1 / 2})^{\alpha}} \mathd y
     \lesssim \frac{1}{t^{\alpha}} \int | K_i (x - y) | \mathd y \lesssim
     \frac{1}{t^{\alpha}} \lesssim (| x | + t^{1 / 2} + 2^{- i})^{- \alpha} .
  \]
  When \ $2^{- i} \geqslant t^{1 / 2}, | x |$ we estimate
  \[ \int \frac{| K_i (x - y) |}{(| y | + t^{1 / 2})^{\alpha}} \mathd y
     \lesssim 2^{\alpha i} \int \frac{| K_1 (y) |}{| 2^i x + y |^{\alpha}}
     \mathd y \lesssim 2^{2 i} \sup_z \int \frac{| K_1 (y) |}{| z + y
     |^{\alpha}} \mathd y \lesssim 2^{2 i} \lesssim (| x | + t^{1 / 2} + 2^{-
     i})^{- \alpha}, \]
  and finally when $| x | \geqslant 2^{- i}, t^{1 / 2}$ we have either $| x |
  \geqslant 2^{- i + 1} | y |$ or $| x | < 2^{- i + 1} | y |$. In the first
  region $| x + 2^{- i} y | \geqslant c | x |$ so
  \[ \int \frac{| K_i (x - y) |}{(| y | + t^{1 / 2})^{\alpha}} \mathd y
     \lesssim \int \frac{| K_1 (y) |}{| x + 2^{- i} y |^{\alpha}} \mathd y
     \lesssim | x |^{- \alpha} \lesssim (| x | + t^{1 / 2} + 2^{- i})^{-
     \alpha}, \]
  while in the second $| y | \geqslant 2^i | x | / 2$, then $| K_1 (y) |
  \leqslant | K_1 (y) |^{1 / 2} f (2^i | x | / 2)$ where $f$ is another
  rapidly decreasing function and in this region
  \[ \int \frac{| K_i (x - y) |}{(| y | + t^{1 / 2})^{\alpha}} \mathd y
     \lesssim f (2^i | x | / 2) \int \frac{| K_1 (y) |^{1 / 2}}{| 2^{- i} y
     |^{\alpha}} \mathd y \lesssim 2^{\alpha i} f (2^i | x | / 2) \lesssim | x
     |^{- \alpha} \lesssim (| x | + t^{1 / 2} + 2^{- i})^{- \alpha}, \]
  concluding our argument.
\end{proof}

\begin{lemma}
  \label{l:integral-est-final}For $m, n \in (0, 5)$, $k, \ell \in [0, 2)$
  define
  \[ I_{k, m, n} \assign \int_{\zeta_{1, 2}, \zeta'_{1, 2}} C_{\varepsilon}
     (\zeta_1 - \zeta_2)^k C_{\varepsilon} (\zeta_1 - \zeta_2)^{\ell}
     C_{\varepsilon} (\zeta_1 - \zeta'_1)^m C_{\varepsilon} (\zeta_2 -
     \zeta'_2)^n | \mu_{q, \zeta_1, \zeta_2} | | \mu_{q, \zeta_1', \zeta_2'}
     |, \]
  with $\mu_{q, \zeta_1, \zeta_2}$ for $\bar{\zeta} = (t, \bar{x})$, $\zeta_i
  = (s_i, x_i)$ $i = 1, 2$ defined as
  \[ \mu_{q, \zeta_1, \zeta_2} \assign [\int_{x, y} K_{q, \bar{x}} (x) \sum_{i
     \sim j} K_{i, x} (y) K_{j, x} (x_2) P_{t - s_1} (y - x_1)] \delta (t -
     s_2) \mathd \zeta_1 \mathd \zeta_2 . \]
  If $\ell = 0$, $0 < m + k - 2 < 5$, $m + k - 2 \in (- 1, 5)$ and $k + m + n
  - 4 \in (0, 5)$ we have the bound
  \[ I_{k, m, n} \lesssim 2^{(k + m + n - 4) q} . \]
  If $(k + m - 2), (\ell + m - 2) \in (0, 5)$, $k + m + \ell - 4 \in (0, 5)$
  and $k + \ell + m + n - 4 \in (0, 5)$ we have the bound
  \[ I_{k, m, n} \lesssim 2^{(k + \ell + m + n - 4) q} . \]
\end{lemma}

\begin{proof}
  Observe that
  \begin{eqnarray*}
    \mu_{q, \zeta_1, \zeta_2} & = & [\sum_{i \sim j} \int_{x, y} K_{q,
    \bar{x}} (x_2) K_{i, x} (y) K_{j, x} (x_2) (P_{t - s_1} (y - x_1) - P_{t -
    s_1} (\bar{x} - x_1))] \delta (t - s_2) \mathd \zeta_1 \mathd \zeta_2\\
    &  & + [\sum_{i \sim j} \int_{x, y} (K_{q, \bar{x}} (x) - K_{q, \bar{x}}
    (x_2)) K_{i, x} (y) K_{j, x} (x_2) P_{t - s_1} (y - x_1)] \delta (t - s_2)
    \mathd \zeta_1 \mathd \zeta_2\\
    & = & K_{q, \bar{x}} (x_2) [P_{t - s_1} (x_2 - x_1) - P_{t - s_1}
    (\bar{x} - x_1)] \delta (t - s_2) \mathd \zeta_1 \mathd \zeta_2\\
    &  & + [\sum_{i \sim j} \int_{x, y} [K_{q, \bar{x}} (x) - K_{q, \bar{x}}
    (x_2)] K_{i, x} (y) K_{j, x} (x_2) P_{t - s_1} (y - x_1)] \delta (t - s_2)
    \mathd \zeta_1 \mathd \zeta_2\\
    & = & \bar{\mu}_{q, \zeta_1, \zeta_2} + \hat{\mu}_{q, \zeta_1, \zeta_2}
  \end{eqnarray*}
  where in the first line we used $\int_y K_{i, x} (y) = 0$ and the fact that
  $\int_x K_{i, x} (x_1') K_{j, x} (x_2) = 0$ if $| i - j | > 1$ and $\sum_{i,
  j} K_{i, x} (y) K_{j, x} (x_2) = \delta (x_2 - y) \delta (x_2 - x)$. Now the
  estimation of the term
  \[ \bar{I}_{k, m, n} \assign \int_{\zeta_{1, 2}, \zeta'_{1, 2}}
     C_{\varepsilon} (\zeta_1 - \zeta_2)^k C_{\varepsilon} (\zeta_1' -
     \zeta'_2)^{\ell} C_{\varepsilon} (\zeta_1 - \zeta'_1)^m C_{\varepsilon}
     (\zeta_1 - \zeta_2)^n | \bar{\mu}_{q, \zeta_1, \zeta_2} | | \bar{\mu}_{q,
     \zeta_1', \zeta_2'} |, \]
  with $\bar{\mu}_{q, \zeta_1, \zeta_2} = K_{q, \bar{x}} (x_2) [P_{t - s_1}
  (x_2 - x_1) - P_{t - s_1} (\bar{x} - x_1)] \delta (t - s_2) \mathd \zeta_1
  \mathd \zeta_2$ can be done with Lemma~\ref{l:hom-conv} and gives the
  expected result. The integral
  \[ \hat{I}_{k, m, n} \assign \int_{\zeta_{1, 2}, \zeta'_{1, 2}}
     C_{\varepsilon} (\zeta_1 - \zeta_2)^k C_{\varepsilon} (\zeta_1' -
     \zeta'_2)^{\ell} C_{\varepsilon} (\zeta_1 - \zeta'_1)^m C_{\varepsilon}
     (\zeta_1 - \zeta_2)^n | \hat{\mu}_{q, \zeta_1, \zeta_2} | | \hat{\mu}_{q,
     \zeta_1', \zeta_2'} |, \]
  with $\hat{\mu}_{q, \zeta_1, \zeta_2} = [\sum_{i \sim j} \int_{x, y} [K_{q,
  \bar{x}} (x) - K_{q, \bar{x}} (x_2)] K_{i, x} (y) K_{j, x} (x_2) P_{t - s_1}
  (y - x_1)] \delta (t - s_2) \mathd \zeta_1 \mathd \zeta_2$ can be estimated
  by multiple changes of variables. We have $K_{q, \bar{x}} (x) - K_{q,
  \bar{x}} (x_2) = 2^{3 q} (x_2 - x) \int_0^1 K' (2^q (x_2 - x) \tau - 2^q
  (\bar{x} - x_2)) \mathd \tau$, and by the scaling properties of
  $C_{\varepsilon}$ and $P_{t, y}$, namely $C_{\varepsilon} (2^{- 2 i} s, 2^{-
  i} x) \lesssim 2^i C_{\varepsilon} (s, x)$ and $P_{2^{- 2 i} s} (2^{- i} x)
  \lesssim 2^{3 i} P_s (x)$ given by Lemma~\ref{l:heat-est} and
  Lemma~\ref{l:cov-est-tx}, we obtain easily the bound on $\hat{I}_{k, m, n}$
  by rescaling the integral.
\end{proof}

\section{Some Malliavin calculus\label{a:malliavin}}

We recall here some tools from Malliavin calculus that are widely used in the
rest of the paper. An introduction to Malliavin calculus and the proofs of
some results of this Appendix can be found
in~{\cite{nualart_2006,nourdin_peccati_2012,shigekawa_stochastic_2004}}.
Lemma~\ref{l:productformula} was inspired by the calculations
of~{\cite{nourdin_nualart_2007}}.

\subsection{Notation }\label{a:malliavin-notation}

Let $\{ W (h) \}_{h \in H}$ be an isonormal Gaussian process indexed by a real
separable Hilbert space $H$. Let $(\Omega, \mathcal{F}, \mathbbm{P})$ a
probability space with $\mathcal{F}$ generated by the isonormal Gaussian
process $W$, we note $L^2 (\Omega) = \bigoplus_{n \geqslant 0} \mathcal{H}_n$
the well-known Wiener chaos decomposition of $L^2 (\Omega)$. For any real
separable Hilbert space $V$ and $k \in \mathbbm{N}$, let $\mathD^k : L^2
(\Omega ; V) \rightarrow L^2 (\Omega ; H^{\odot k} \otimes V)$ be the
Malliavin derivative and $\delta^k : L^2 (\Omega ; H^{\otimes k} \otimes V)
\rightarrow L^2 (\Omega ; V)$ the divergence operator (also called Skorohod
integral) defined as the adjoint of $\mathD^k$. For $p \geqslant 1$ we will
write $\mathbbm{D}^{k, p} (V) \subset L^p (\Omega ; V)$ for the closure of
smooth random variables with respect to the norm
\[ \| \Psi \|_{\mathbbm{D}^{k, p} (V)} = [\mathbbm{E} (\| \Psi \|_V^p) +
   \sum_{j = 1}^k \mathbbm{E} (\| \mathD^j \Psi \|^p_{H^{\otimes j} \otimes
   V})]^{1 / p} \]
with the notation $\mathbbm{D}^{k, p} \assign \mathbbm{D}^{k, p}
(\mathbbm{R})$. Let $\{ P_t \}_{t \in \mathbbm{R}^+}$ the Ornstein-Uhlenbeck
semigroup and $L : L^2 (\Omega) \rightarrow L^2 (\Omega)$ its generator (i.e.
$e^{tL} = P_t$). Following~{\cite{shigekawa_stochastic_2004}} we introduce the
Green operator
\[ G_j = (j - L)^{- 1} \]
with the notation
\begin{equation}
  G_{[j]}^{[m]} \assign \prod_{k = j}^m G_k  \quad \text{for} \quad 1
  \leqslant j \leqslant m \label{e:green-op}
\end{equation}
so that $G_{[j]}^{[j]} = G_j$. To avoid confusion, it is worth stressing that
$G_{[j]}^{[m]}$ is \tmtextit{not} the $m$-th power of the operator $G_j$ but
just a shortcut for $\prod_{k = j}^m G_k $.

\subsection{Partial chaos expansion}

Let $\Psi \in L^2 (\Omega)$ which has the Wiener chaos decomposition $\Psi =
\sum_{n \geqslant 0} J_n \Psi$. Then by Proposition~1.2.2 of
{\cite{nualart_2006}} $\mathD J_n \Psi = J_{n - 1} \mathD \Psi$, and knowing
that $LJ_n \Psi = - nJ_n \Psi$ we obtain the commutation property
\begin{equation}
  \mathD (j - L)^{- \alpha} \Psi = \mathD \sum_{n \geqslant 0} \frac{1}{(j +
  n)^{\alpha}} J_n \Psi = \sum_{n \geqslant 1} \frac{1}{(j + n)^{\alpha}} J_{n
  - 1} \mathD \Psi = (j + 1 - L)^{- \alpha} \mathD \Psi \label{e:LDcomm}
\end{equation}
for every $\alpha > 0$, $j > 0$. The above formula holds also for $j = 0$ if
$\mathbbm{E} (\Psi) = 0$.

The results we have recalled so far let us write an $n$th-order Wiener chaos
expansion for a random variable in $\mathbbm{D}^{n, 2}$:

\begin{lemma}
  \label{l:Fexpansion}Let $\Psi \in \mathbbm{D}^{n, 2}$ and $G_{[j]}^{[m]}$ as
  in~(\ref{e:green-op}). Then for every $n \in \mathbbm{N}\backslash \{ 0 \}$
  $G_{[1]}^{[n]} \mathD^n \Psi \in \tmop{Dom} \delta^n$, $J_0 \mathD^k \Psi
  \in \tmop{Dom} \delta^k$ $\forall 0 \leqslant k < n$ and
  \begin{equation}
    \delta^n G_{[1]}^{[n]} \mathD^n \Psi = (\tmop{id} - J_0 - \ldots - J_{n -
    1}) \Psi = \Psi - \sum_{k = 0}^{n - 1} \frac{1}{k!} \delta^k J_0 \mathD^k
    \Psi . \label{e:Fexpansion}
  \end{equation}
  \[ \  \]
\end{lemma}

\begin{proof}
  We have for any $\Psi \in L^2 (\Omega)$, since $L = - \delta \mathD$
  ({\cite{nualart_2006}}, Proposition~1.4.3):
  \[ \Psi -\mathbbm{E} (\Psi) = L L^{- 1} (\Psi - J_0 \Psi) = - \delta \mathD
     L^{- 1} (\Psi - J_0 \Psi) = \delta (1 - L)^{- 1} \mathD \Psi \]
  where we used (\ref{e:LDcomm}), and the fact that $(1 - L)^{- 1} \mathD \Psi
  \in \tmop{Dom} \delta$ is obvious by construction. This yields the first
  order expansion $\Psi =\mathbbm{E} (\Psi) + \delta (1 - L)^{- 1} \mathD
  \Psi$. Iterating the expansion up to order $n$ we obtain
  (\ref{e:Fexpansion}). It is clear that $J_0 \mathD^k \Psi \in \tmop{Dom}
  \delta^k$ since $J_0 \mathD^k \Psi$ is constant with values in $H^{\otimes
  k}$. The second equality comes from the fact that $\delta^k J_0 \mathD^k
  \Psi \in \mathcal{H}_k$ $\forall k \in \mathbbm{N}$, indeed $\forall \Psi'
  \in \mathbbm{D}^{k, 2}$:
  \[ \mathbbm{E} (\delta^k (J_0 \mathD^k \Psi) J_k \Psi') = \langle J_0
     \mathD^k \Psi, J_0 \mathD^k \Psi' \rangle_{L^2 (\Omega ; H^{\otimes k})}
     =\mathbbm{E} (\delta^k (J_0 \mathD^k \Psi) \Psi') \]
\end{proof}

In order to obtain $L^p$ estimations of the remainder term $\delta^n
G_{[1]}^{[n]} \mathD^n \Psi$ generated by expansion (\ref{e:Fexpansion}), we
used the following lemmas:

\begin{lemma}[{\cite{nualart_2006}}, Prop.~1.5.7]
  \label{l:delta-norm}Let $V$ be a real separable Hilbert space. For every $p
  > 1$ and every $q \in \mathbbm{N}, k \geqslant q$ and every $\Psi \in
  \mathbbm{D}^{k, p} (H^q \otimes V)$ we have
  \[ \| \delta^q (\Psi) \|_{\mathbbm{D}^{k - q, p} (V)} \lesssim_{k, p} \|
     \Psi \|_{\mathbbm{D}^{k, p} (H^q \otimes V)} \]
\end{lemma}

\begin{remark}
  \label{r:delta-deterministic}Let $V$ be a real separable Hilbert space. For
  every $v \in V$ and every $\Psi \in \mathbbm{D}^{q, 2} (H^{\otimes q})$ with
  $q \in \mathbbm{N}$ we have $\Psi \otimes v \in \tmop{Dom} \delta^q$ and
  \[ \delta^q (\Psi) v = \delta^q (\Psi \otimes v) . \]
  Indeed, notice that for every smooth $\Psi' \in \mathbbm{D}^{q, 2} (V)$ and
  every smooth $\Psi \in \mathbbm{D}^{q, 2} (H^{\otimes q})$ we have
  \[ \mathbbm{E} (\langle \delta^q (\Psi \otimes v), \Psi' \rangle_V)
     =\mathbbm{E} (\langle \Psi \otimes v, \mathD^q \Psi' \rangle_{H^{\otimes
     q} \otimes V}) =\mathbbm{E} (\langle \delta^q (\Psi) v, \Psi' \rangle_V)
     . \]
  Now since $\mathD^q (\Psi \otimes v) = \mathD^q \Psi \otimes v$ and $\Psi
  \in \mathbbm{D}^{q, 2} (H^{\otimes q})$, we have $\Psi \otimes v \in
  \mathbbm{D}^{q, 2} (H^{\otimes q} \otimes V)$. Lemma~\ref{l:delta-norm}
  yields the bound $\| \delta^q (\Psi \otimes v) \|_{L^2 (V)} \lesssim \| \Psi
  \otimes v \|_{\mathbbm{D}^{q, 2} (H^{\otimes q} \otimes V)}$ which allows to
  pass to the limit for $\Psi'$ and $\delta^q (\Psi \otimes v)$ in $L^2 (V)$.
\end{remark}

\begin{lemma}[{\cite{shigekawa_stochastic_2004}}, Prop.~4.3]
  For every $j > 0$ the operator $(j - L)^{- 1 / 2}$ is bounded in $L^p$ for
  every $1 \leqslant p < \infty$. 
\end{lemma}

\begin{lemma}
  \label{l:bounded-oper}Let $j \in \mathbbm{N}\backslash \{ 0 \}$ and $V$ a
  real separable Hilbert space. There exists a finite constant $c_p$ such that
  for every $\Psi \in L^p (\Omega, V)$:
  \[ \| \mathD (j - L)^{- 1 / 2} \Psi \|_{L^p (\Omega, H \otimes V)} \leqslant
     c_p \| \Psi \|_{L^p (\Omega, V)} \]
  (where the operator $\mathD (j - L)^{- 1 / 2}$ is defined on every $\Psi$
  which is polynomial in $W (h_1), \ldots, W (h_n)$ and can be extended by
  density on $L^p$).
\end{lemma}

\begin{proof}
  First notice that we can suppose w.l.o.g. $\mathbbm{E} (\Psi) = 0$ thanks
  to~(\ref{e:LDcomm}). Therefore we can write $\mathD (j - L)^{- \frac{1}{2}}$
  as
  \[ \mathD (j - L)^{- 1 / 2} = \mathD (- C)^{- 1} (- C)  (j - L)^{- 1 / 2} \]
  with $C = - \sqrt{- L}$. We decompose the second part as $- C (j - L)^{- 1
  / 2} \Psi = \sum_{n = 1}^{\infty} \left( \frac{n}{j + n} \right)^{1 / 2} J_n
  \Psi = T_{\phi} \Psi$, with $T_{\phi} \Psi \assign \sum_{n = 0}^{\infty}
  \phi (n) J_n \Psi$. We apply Theorem~1.4.2 of {\cite{nualart_2006}} to show
  that $T_{\phi}$ is bounded in $L^p$, indeed $\phi (n) = h (1 / n)$ and $h
  (x) = (jx + 1)^{- 1 / 2}$ which is analytic in a neighbourhood of 0.
  Finally, we can apply Proposition~1.5.2 of {\cite{nualart_2006}} to show
  that $\mathD C^{- 1}$ is bounded in $L^p$, thus concluding the proof.
\end{proof}

The two lemmas above give the following immediate corollary:

\begin{corollary}
  \label{l:bounded-q}For every $1 \leqslant m \leqslant n$ the operator
  $G_{[m]}^{[n]} \assign \prod_{j = m}^n (j - L)^{- 1}$ is bounded in $L^p$
  for every $1 \leqslant p < \infty$.
  
  Moreover, Let $j \in \mathbbm{N}\backslash \{ 0 \}$ and $V$ a real separable
  Hilbert space. Then for every $\Psi \in L^p (\Omega, V)$ we have:
  \[ \| \mathD (j - L)^{- 1} \Psi \|_{L^p (\Omega, H \otimes V)} \lesssim \|
     \Psi \|_{L^p (\Omega, V)} . \]
  Moreover, for every $0 \leqslant k \leqslant 2 m$, $i \geqslant 0$ we have
  \[ \| \mathD^k G_{[i + 1]}^{[i + m]} \Psi \|_{L^p (\Omega, H^{\otimes k}
     \otimes V)} \lesssim \| \Psi \|_{L^p (\Omega, V)} . \]
\end{corollary}

The next lemma is one of the most useful tools of this paper. It allows us to
write products of decompositions of the type (\ref{e:Fexpansion}) as sums of
iterated Skorohod integrals. From now on we will note $\langle \cdummy,
\cdummy \rangle_{H^{\otimes r}}$ the $r$-th contraction, which to avoid
inconsistency has to be taken between symmetric tensors. We also note
$h^{\odot n}_{v_1, \ldots, v_n} \assign h_{v_1} \odot \ldots \odot h_{v_n}$
for $h_{v_1}, \ldots, h_{v_n} \in H$.

\begin{lemma}
  \label{l:productformula}
  
  Let $u = f (W (h_u)) h^{\otimes m}_u$ and $v = Fh^{\odot n}_{v_1, \ldots,
  v_n}$ with $f \in C^{m + n} (\mathbbm{R})$ and $F \in \mathbbm{D}^{m + n,
  2}$. Then
  \[ \delta^m (u) \delta^n (v) = \]
  \begin{equation}
    = \sum_{(q, r, i) \in I_{m, n}} C_{m, n, q, r, i} \delta^{m + n - q - r}
    [f^{(r - i)} (W (h_u)) \langle h_u^{\otimes m - i}, \mathD^{q - i} F
    \rangle_{H^{\otimes q - i}} \langle h^{\otimes r}_u, h^{\odot n}_{v_1,
    \ldots, v_n} \rangle_{H^{\otimes r}}] \label{e:productformula-variant}
  \end{equation}
  with $C_{m, n, q, r, i} \assign \binom{m}{q} \binom{n}{r} \binom{q}{i}
  \binom{r}{i} i!$ and $I_{m, n} \assign \{ (q, r, i) \in \mathbbm{N}^3 : 0
  \leqslant q \leqslant m, 0 \leqslant r \leqslant n, 0 \leqslant i \leqslant
  q \wedge r \}$.
  
  A trivial change of variables gives also:
  \[ \delta^m (u) \delta^n (v) = \]
  \begin{equation}
    = \sum_{(i, q, r) \in I'_{m, n}} C_{m, n, q + i, r + i, i} \delta^{m + n -
    q - r - 2 i} [f^{(r)} (W (h_u)) \langle h_u^{\otimes m - i}, \mathD^q F
    \rangle_{H^{\otimes q}} \langle h^{\otimes r + i}_u, h^{\odot n}_{v_1,
    \ldots, v_n} \rangle_{H^{\otimes r + i}}]
  \end{equation}
  with $I'_{m, n} \assign \{ (i, q, r) \in \mathbbm{N}^3 : 0 \leqslant i
  \leqslant m \wedge n, 0 \leqslant q \leqslant m - i, 0 \leqslant r \leqslant
  n - i \}$.
\end{lemma}

\begin{remark}
  \label{r:producformula-easier}In the special case $v = g (W (h_v))
  h^{\otimes n}_v$ eq.~(\ref{e:productformula-variant}) takes the form
  \begin{equation}
    \delta^m (u) \delta^n (v) = \sum_{(q, r, i) \in I_{m, n}} C_{m, n, q, r,
    i} \delta^{m + n - q - r} (\langle \mathD^{r - i} u, \mathD^{q - i} v
    \rangle_{H^{\otimes q + r - i}}) \label{e:productformula-easy}
  \end{equation}
  which is just a generalization to Skorohod integrals of the multiplication
  formula for multiple Wiener integrals~({\cite{shigekawa_derivatives_1980}},
  {\cite{ustunel_sophisticated_2014}}). We can write the above formula more
  explicitly as
  \[ \delta^m (u) \delta^n (v) = \sum_{(q, r, i) \in I_{m, n}} C_{m, n, q, r,
     i} \delta^{m + n - q - r} [f^{(r - i)} (W (h_u)) g^{(q - i)} (W (h_v))
     h_u^{\otimes m - q} \otimes h_v^{\otimes n - r}] \langle h_u, h_v
     \rangle^{q + r - i} . \]
\end{remark}

\begin{remark}
  Note that one can assume w.l.o.g. the argument of $\delta^{m + n - q - r}$
  in (\ref{e:productformula-variant}) to be symmetric, and this would allow to
  iterate Lemma~\ref{l:productformula}.
\end{remark}

\begin{remark}
  We can give the following intuition for the second formula in
  Lemma~\ref{l:productformula}. The random variables $u$ and $v$ have an
  infinite chaos decomposition, and following the tree-like notation
  of~{\cite{gubinelli_paracontrolled_2012}} or~{\cite{hairer_theory_2014}}
  they could be thought of as having an infinite number of leaves which need
  to be contracted with each other.
  
  It is apparent that the index $i$ in the second equation denotes
  contractions between the already existing leaves of the trees $u, v$. The
  indexes $r$ and $q$ count new leaves in each vertex that are created by the
  Malliavin derivatives, which are then contracted with other leaves from the
  other tree. There are then $m + n - r - q - 2 i$ overall unmatched leaves
  which are arguments to the iterated Skorokhod integral.
  
  The more intuitive interpretation of the second equation in
  Lemma~\ref{l:productformula} is the reason why we gave two distinct
  expression for the same quantity. Nevertheless, the
  formula~(\ref{e:productformula-variant}) is more practical in the
  calculations and is more widely used throughout the paper.
\end{remark}

\begin{proof}[Lemma~\ref{l:productformula}]
  Using Cauchy-Schwarz inequality and Lemma~\ref{l:delta-norm} we can show
  that $\langle D^r \delta^n (v), \delta^j (u) \rangle_{H^{\otimes r}} \in L^2
  (\Omega, H^{\otimes m - j - r})$ for every $0 \leqslant r + j \leqslant m$.
  Then we apply Lemma~\ref{l:lemma-nourdin} to get:
  \[ \delta^m (u) \delta^n (v) = \sum_{r = 0}^n \binom{n}{r} \delta^{n - r}
     (\langle \mathD^r \delta^m (u), v \rangle_{H^{\otimes r}}) . \]
  Using the commutation formula (\ref{e:Ddelta-comm}) we rewrite the r.h.s. as
  \[ \delta^m (u) \delta^n (v) = \sum_{r = 0}^n \binom{n}{r} \sum_{i = 0}^{r
     \wedge m} \binom{r}{i} \binom{m}{i} i! \delta^{n - r} (\langle \delta^{m
     - i} (\mathD^{r - i} u), v \rangle_{H^{\otimes r}}) . \]
  We obtain
  \[ \langle \delta^{m - i} (\mathD^{r - i} u), v \rangle_{H^{\otimes r}} =
     \delta^{m - i} (f^{(r - i)} (W (h_u)) h_u^{\otimes m - i}) F \langle
     h^{\otimes r}_u, h^{\odot n}_{v_1, \ldots, v_n} \rangle_{H^{\otimes r}}
  \]
  and using again Lemma~\ref{l:lemma-nourdin} we obtain
  \[ \langle \delta^{m - i} (\mathD^{r - i} u), v \rangle_{H^{\otimes r}} =
     \sum_{\ell = 0}^{m - i} \binom{m - i}{\ell} \delta^{m - i - \ell} (f^{(r
     - i)} (W (h_u)) \langle h_u^{\otimes m - i}, \mathD^{\ell} F
     \rangle_{H^{\otimes \ell}} \langle h^{\otimes r}_u, h_{v_1} \odot \ldots
     \odot h_{v_n} \rangle_{H^{\otimes r}}) \]
  where we used $\delta^k (\Psi) h^{\otimes n - r} = \delta^k (\Psi \otimes
  h^{\otimes n - r})$ for $\Psi \in \tmop{Dom} \delta^k$, as seen in
  Remark~\ref{r:delta-deterministic}. Substituting this expression into
  $\delta^m (u) \delta^n (v)$ we get
  \[ \delta^m (u) \delta^n (v) = \sum_{(r, i, \ell) \in J_{m, n}} A_{m, n, r,
     i, \ell} \delta^{m + n - r - i - \ell} [f^{(r - i)} (W (h_u)) \langle
     h_u^{\otimes m - i}, \mathD^{\ell} F \rangle_{H^{\otimes \ell}} \langle
     h^{\otimes r}_u, h_{v_1} \odot \ldots \odot h_{v_n} \rangle_{H^{\otimes
     r}}] \]
  where we set $A_{m, n, r, i, \ell} \assign \binom{m}{i} \binom{n}{r}
  \binom{m - i}{\ell} \binom{r}{i} i!$ and $J_{m, n} \assign \{ (r, i, \ell)
  \in \mathbbm{N}^3 : 0 \leqslant r \leqslant n, 0 \leqslant i \leqslant r
  \wedge n, 0 \leqslant \ell \leqslant m - i \}$.
  
  In order to complete the proof we just have to perform some basic changes
  of indexes. Taking $q = \ell + i$ and noting that $\binom{m}{i} \binom{m -
  i}{q - i} = \binom{q}{i} \binom{m}{q}$ we have $A_{m, n, r, i, \ell} =
  \binom{m}{q} \binom{q}{i} \binom{n}{r} \binom{r}{i} i!$ and this yields
  (\ref{e:productformula-variant}).
  
  Finally, we perform the change of variables $q - i \rightarrow q$, $r - i
  \rightarrow r$ to get the second formula.
\end{proof}

We give below the results we used to prove Lemma~\ref{l:productformula}.

\begin{lemma}[{\cite{nourdin_nualart_2007}}, Lemma~2.1]
  \label{l:lemma-nourdin}Let $q \in \mathbbm{N}\backslash \{ 0 \}$, $\Psi \in
  \mathbbm{D}^{q, 2}$, $u \in \tmop{Dom} \delta^q$ and symmetric. Assume also
  that $\forall 0 \leqslant r + j \leqslant q$ $\langle D^r \Psi, \delta^j (u)
  \rangle_{H^{\otimes r}} \in L^2 (\Omega, H^{\otimes q - r - j})$. Then
  $\forall 0 \leqslant r < q$ $\langle D^r \Psi, u \rangle_r \in \tmop{Dom}
  \delta^{q - r}$ and
  \[ \Psi \delta^q (u) = \sum_{r = 0}^q \binom{q}{r} \delta^{q - r} (\langle
     \mathD^r \Psi, u \rangle_{H^{\otimes r}}) . \]
\end{lemma}

\begin{remark}
  \label{r:delta-tensor}Note that
  \[ \delta^n (h^{\otimes n}) = \llbracket W^n (h) \rrbracket \]
  where $\llbracket \cdummy \rrbracket$ stands for the Wick product. Indeed
  $\forall \Psi \in \mathbbm{D}^{1, 2}$ we know that $\mathbbm{E} [\delta
  (h^{\otimes n}) \Psi] =\mathbbm{E} [W (h) h^{\otimes n - 1} \Psi]$ using the
  definition of $\delta$, and then $\delta^n (h^{\otimes n}) = \delta^{n - 1}
  (W (h) h^{\otimes n - 1})$. From Lemma~\ref{l:lemma-nourdin} we have, since
  $\mathD W (h) = h$:
  \[ \delta^{n - 1} (W (h) h^{\otimes n - 1}) = \delta^{n - 1} (h^{\otimes n -
     1}) W (h) - (n - 1) \langle h, h \rangle \delta^{n - 2} (h^{\otimes n -
     2}) \]
  which gives by induction $\delta^n (h^{\otimes n}) = \llbracket W^n (h)
  \rrbracket$.
\end{remark}

\begin{lemma}
  Let $j, k \in \mathbbm{N}$, $u \in \mathbbm{D}^{j + k, 2} (H^{\otimes j})$
  symmetric and such that all its derivatives are symmetric. We have
  \begin{equation}
    \mathD^k \delta^j (u) = \sum_{i = 0}^{k \wedge j} \binom{k}{i}
    \binom{j}{i} i! \delta^{j - i} (D^{k - i} u)  \label{e:Ddelta-comm}
  \end{equation}
\end{lemma}

\begin{proof}
  If $j = 0$, $k = 1$ or $k = 0, j = 1$ eq.~(\ref{e:Ddelta-comm}) is trivial.
  Let $j = k = 1$ and $u \in \mathbbm{D}^{2, 2} (H) \subset \mathbbm{D}^{1, 2}
  (H)$. We can apply Proposition~1.3.2 of {\cite{nualart_2006}} to obtain
  $\langle \mathD \delta (u), h \rangle = \langle u, h \rangle + \delta
  (\langle \mathD u, h \rangle)$ $\forall h \in H$. Since by hypothesis
  $\mathD u$ is symmetric we have $\delta (\langle \mathD u, h \rangle) =
  \langle \delta \mathD u, h \rangle$, and then $\mathD \delta (u) = u +
  \delta \mathD u$. The proof by induction is easy noticing that $\mathD
  \delta^j = \delta^j \mathD + j \delta^{j - 1}$.
\end{proof}

\end{appendices}

\

\

\end{document}